\documentclass[11pt]{amsart}

\usepackage[utf8]{inputenc}

\usepackage{enumitem} %allows e.g. to change labels in enumerations
\usepackage{tikz} %pictures
\usepackage{tikz-cd} %diagrams
\usetikzlibrary{arrows,calc,decorations.markings}
\usepackage{hyperref} %links
\usepackage[capitalise]{cleveref} %automatic references
\usepackage{amsmath}%useful math symbols
\usepackage{amssymb}%useful math symbols
\usepackage{thmtools}%useful theorem tools
\usepackage{bbm}
\usepackage{mathtools}
\usepackage[margin=1.2in]{geometry}
\usepackage{comment}
\usepackage{ifthen} %allows to make complicated newcommands
\usepackage{stmaryrd} %for the \boxslash

\newenvironment{tz}{\begin{center}\begin{tikzpicture}[scale=1]}{\end{tikzpicture}\end{center}}
\tikzstyle{d}=[double distance=.3ex]
\newcommand{\arrowdot}{
\ensuremath{\begin{tikzpicture}
\node (A) at (0,-.4) {};
\node (B) at (.4,-.4) {};
\draw[->, line width=.1ex] (0,-.6) -- (.4,-.6);
\node[shape=circle, fill=black, scale=0.35] (A) at  (.17,-.6) {};
\end{tikzpicture}
}}

\tikzset{over/.style={auto=false,fill=white,inner sep=1.5pt, minimum size=0, outer sep=0}, 
pro/.style={postaction={decorate,decoration={
        markings,
        mark=at position .5 with {\node at (0,0) {$\bullet$};}
      }},
      inner sep=.9ex,
      },n/.style={double equal sign distance, -implies},
  }
  
\newcommand{\vcong}{\rotatebox{270}{$\cong$}}
\tikzset{%
node distance=1.5cm, la/.style={scale=0.8}, rr/.style={xshift=1.5cm},
space/.style={xshift=.5cm},
    symbol/.style={%
        draw=none,
        every to/.append style={%
            edge node={node [sloped, allow upside down, auto=false]{$#1$}}},
            
    }
}

\newcommand{\sq}[5]{{#1}\colon({#4} \; ^{{#2}}_{\substack{{#3}}} \; {#5})}
\newcommand{\pushprod}[7]{{#1}\,\square_{#7}\, {#4}\colon {#2}\otimes_{#7} {#6}\coprod_{{#2}\otimes_{#7} {#5}} {#3}\otimes_{#7} {#5}\to {#3}\otimes_{#7} {#6}}
\newcommand{\pushout}[3]{{#1}\,\square_{#3}\, {#2}}

\newtheorem{thm}{Theorem}[section] % numbered theorem
\newtheorem{cor}[thm]{Corollary}
\newtheorem{lemma}[thm]{Lemma}
\newtheorem{prop}[thm]{Proposition}
\declaretheorem[name=Theorem,numbered=yes]{theoremA}

\theoremstyle{definition}
\newtheorem{defn}[thm]{Definition}
\newtheorem{notation}[thm]{Notation}

\theoremstyle{remark}
\newtheorem{rmk}[thm]{Remark}

\newcounter{diagram}
\newenvironment{diagram}{\setcounter{diagram}{\value{thm}}\refstepcounter{diagram}\begin{center}\hfill\begin{tikzpicture}}
    {\end{tikzpicture}\hfill{\normalfont(\thediagram)}\end{center}\refstepcounter{thm}}
\numberwithin{diagram}{section}

\crefname{lem}{Lemma}{Lemmas}
\crefname{thm}{Theorem}{Theorems}
\crefname{theoremA}{Theorem}{Theorems}
\crefname{defn}{Definition}{Definitions}
\crefname{prop}{Proposition}{Propositions}
\crefname{rmk}{Remark}{Remarks}
\crefname{cor}{Corollary}{Corollaries}
\crefname{ex}{Example}{Examples}
\crefname{notation}{Notation}{Notations}

\newlist{rome}{enumerate}{7}
\setlist[rome]{label=(\roman*)}

\newcommand{\Bicats}{\mathrm{Bicat}_s}
\newcommand{\TwoCat}{2\mathrm{Cat}}
\newcommand{\Set}{\mathrm{Set}}
\newcommand{\Cat}{\mathrm{Cat}}
\newcommand{\id}{\mathrm{id}}
\newcommand{\DblCat}{\mathrm{DblCat}}
\newcommand{\wkDblCat}{\mathrm{wkDblCat}_s}
\renewcommand{\AA}{\mathbb{A}}
\newcommand{\BB}{\mathbb{B}}
\newcommand{\bC}{\mathbb{C}}

\newcommand{\A}{\mathcal{A}}
\newcommand{\B}{\mathcal{B}}
\newcommand{\C}{\mathcal{C}}

\newcommand{\cv}{\mathcal{V}}
\newcommand{\vtwo}{\mathbb{V}\mathbbm{2}}
\newcommand{\LV}{\mathbb{L}}
\newcommand{\bfH}{\mathbf{H}}
\newcommand{\bfV}{\mathbf{V}}
\newcommand{\bbH}{\mathbb{H}}
\newcommand{\bbV}{\mathbb{V}}
\newcommand{\cm}{\mathcal{M}}
\newcommand{\cn}{\mathcal{N}}

\newcommand{\cc}{\mathcal{C}}

\newcommand{\cp}[1]{\mathrm{Path}({#1})}
\newcommand{\ps}{{\mathrm{ps}}}
\newcommand{\Psd}{{\mathrm{Ps}}}
\newcommand{\Gray}{{\mathrm{Gr}}}
\newcommand{\Eadj}{E_{\mathrm{adj}}}
\newcommand{\op}{\mathrm{op}}

\title{A 2Cat-inspired model structure for double categories}

\author[L.\ Moser]{Lyne Moser}
\address{Max Planck Institute for Mathematics, Vivatsgasse 7, 53111 Bonn, Germany}
\email{moser@mpim-bonn.mpg.de}

\author[M.\ Sarazola]{Maru Sarazola}
\address{Department of Mathematics, Cornell University, Ithaca NY, 14853, USA}
\email{mes462@cornell.edu}

\author[P.\ Verdugo]{Paula Verdugo}
\address{Department of Mathematics, Macquarie University, NSW 2109, Australia}
\email{paula.verdugo@hdr.mq.edu.au}

\begin{document}

\maketitle

\begin{abstract}
    We construct a model structure on the category $\mathrm{DblCat}$ of double categories and double functors. Unlike previous model structures for double categories, it recovers the homotopy theory of 2-categories through the horizontal embedding $\mathbb{H}\colon2\mathrm{Cat}\to\mathrm{DblCat}$, which is both left and right Quillen, and homotopically fully faithful. Furthermore, we show that Lack's model structure on $2\mathrm{Cat}$ is both left- and right-induced along $\mathbb{H}$ from our model structure on $\mathrm{DblCat}$. In addition, we obtain a $2\mathrm{Cat}$-enrichment of our model structure on~$\mathrm{DblCat}$, by using a variant of the Gray tensor product. 
    
    Under certain conditions, we prove a Whitehead theorem, characterizing our weak equivalences as the double functors which admit an inverse pseudo double functor up to horizontal pseudo natural equivalence. This retrieves the Whitehead theorem for 2-categories.
    
    Analogous statements hold for the category $\mathrm{wkDblCat}_s$ of weak double categories and strict double functors, whose homotopy theory recovers that of bicategories. Moreover, we show that the full embedding $\mathrm{DblCat}\to\mathrm{wkDblCat}_s$ is a Quillen equivalence.
\end{abstract}

\setcounter{tocdepth}{1}
\tableofcontents

\section{Introduction}

In category theory as well as homotopy theory, we strive to find the correct notion of ``sameness'', often with a specific context or perspective in mind. When working with categories themselves, it is commonly agreed that having an isomorphism between categories is much too strong a requirement, and we instead concur that the right condition to demand is the existence of an \emph{equivalence} of categories.

There are many ways one can justify this in practice, but, at heart, it is due to  the fact that the category $\Cat$ of categories and functors actually forms a $2$-category, with $2$-cells given by the natural transformations. Therefore, instead of asking that a functor $F\colon \A\to \B$ has an inverse $G\colon \B\to \A$ such that their composites are \emph{equal} to the identities, it is more natural to ask for the existence of \emph{natural isomorphisms} $\id_{\A}\cong GF$ and $FG\cong \id_{\B}$. In particular, this characterizes $F$ as a functor that is surjective on objects up to isomorphism, and fully faithful on morphisms.

Ever since Quillen's seminal work \cite{Quillen}, and even more so in the last two decades, we have come to expect that any reasonable notion of equivalence in a category should lend itself to defining the class of weak equivalences of a model structure. This is in fact the case of the categorical equivalences: the category $\Cat$ can be endowed with a model structure, called the \emph{canonical model structure}, in which the weak equivalences are precisely the equivalences of categories.

Going one dimension up and focusing  on $2$-categories, the $2$-functors themselves now form a $2$-category, with higher cells given by the pseudo natural transformations, and the so-called modifications between them. We can then define a $2$-functor $F\colon  \A\to \B$ to be a \emph{biequivalence} if it has an inverse $G\colon \B\to \A$ together with pseudo natural equivalences $\id_{\A}\simeq GF$ and ${FG\simeq \id_{\B}}$, i.e., equivalences in the corresponding $2$-categories of $2$-dimensional functors. Note that this inverse $G$ is in general a pseudo functor rather than a $2$-functor. Furthermore, a Whitehead theorem for $2$-categories \cite[Theorem 7.4.1]{JohYau} is available, and characterizes the biequivalences as the $2$-functors that are surjective on objects up to equivalence, full on morphisms up to invertible $2$-cell, and fully faithful on $2$-cells.
 
As in the case of the equivalences of categories, the biequivalences of $2$-categories are part of the data of a model structure. Indeed, in \cite{Lack2Cat,LackBicat}, Lack defines a model structure on the category $\TwoCat$ of $2$-categories and $2$-functors in which the weak equivalences are precisely the biequivalences; we henceforth refer to it as the \emph{Lack model structure}. In particular, the canonical homotopy theory of categories embeds reflectively in this homotopy theory of $2$-categories. 

In this paper, we consider another type of $2$-dimensional objects, called \emph{double categories}, which have both horizontal and vertical morphisms between pairs of objects, related by $2\text{-dimensional}$ cells called \emph{squares}. These are more structured than $2$-categories, in the sense that a $2$-category $\A$ can be seen as a horizontal double category $\bbH\A$ with only trivial vertical morphisms. As a consequence, the study of various notions of $2$-category theory benefits from a passage to double categories. For example, a $2$-limit of a $2$-functor~$F$ does not coincide with a $2$-terminal object in the slice $2$-category of cones, as shown in~\cite[~Counter-example 2.12]{clingmanMoser}. However, by considering the $2$-functor $F$ as a horizontal double functor~$\bbH F$, Grandis and Par\'e prove that a $2$-limit of $F$ is precisely a double terminal object in the slice double category of cones over $\bbH F$; see \cite[\S 4.2]{GrandisPare,GraPar19bis} and \cite[Theorem 5.6.5]{Grandis}. 

This horizontal embedding of $2$-categories into double categories is fully faithful, and we expect to have a homotopy theory of double categories that contains that of $2$-categories; constructing such a homotopy theory is the aim of this paper. 

The idea of defining a model structure on the category of double categories is scarcely a new one. In \cite{FP}, Fiore and Paoli construct a Thomason model structure on the category $\DblCat$ of double categories and double functors (more precisely, on the category of $n$-fold categories), and in \cite{FPP}, Fiore, Paoli, and Pronk construct several categorical model structures on $\DblCat$. However, the horizontal embedding of $2$-categories does not induce a Quillen pair between the Lack model structure on $\TwoCat$ and any of these model structures on $\DblCat$; this follows from \cref{Hjnotwe}. Some intuition is provided by the fact that their categorical model structures on $\DblCat$ are constructed from the canonical model structure on $\Cat$. As a result, the weak equivalences in each of these model structures induce two equivalences of categories: one between the categories of objects and horizontal morphisms, and one between the categories of vertical morphisms and squares. However, a biequivalence between $2$-categories does not generally induce an equivalence between the underlying categories. Therefore, the horizontal embedding of $\TwoCat$ into $\DblCat$ will not preserve weak equivalences. 

In order to remedy this loss of higher data, we aim to extract from a double category~$\AA$ two $2$-categories whose underlying categories are precisely the ones mentioned above. First, we can promote the underlying category of objects and horizontal morphisms of $\AA$ to a $2$-category by using the right adjoint to the horizontal embedding $\bbH$: this is a well-known construction given by the underlying horizontal $2$-category~$\bfH \AA$, whose $2$-cells are given by those squares of $\AA$ with trivial vertical boundaries. As shown by Ehresmann and Ehresmann in~\cite{TwoEhresmann}, the category $\DblCat$ is cartesian closed, and we denote by $[-,-]$ its internal hom double categories. We can then alternatively describe the underlying horizontal $2$-category~$\bfH\AA$ as the $2$-category $\bfH[\mathbbm{1}, \AA]$, where $\mathbbm 1$ denotes the terminal category. 

From this same perspective, the category of vertical morphisms and squares can be seen as the underlying horizontal category of the double category $[\vtwo,\AA]$, where $\vtwo$ is the free double category on a vertical morphism. To promote this to a $2$-category we can simply consider instead the underlying horizontal $2$-category $\bfH[\vtwo,\AA]$; this defines a new functor~$\cv$ that sends a double category~$\AA$ to a $2$-category $\cv\AA$ of vertical morphisms, squares, and $2$-cells as described in \cref{defn_cv}. 

Using these constructions, we introduce a new notion of weak equivalences between double categories, that we suggestively call \emph{double biequivalences}; these are given by the double functors $F$ such that the induced $2$-functors $\bfH F$ and $\cv F$ are biequivalences in $\TwoCat$. This provides a $2$-categorical analogue of notions of equivalences between double categories already present in the literature. Notably, double biequivalences are the natural $2$-categorical version of equivalences described by Grandis in \cite[Theorem 4.4.5 (iv)]{Grandis}, which are precisely the double functors inducing equivalences between the categories of objects and horizontal morphisms, and the categories of vertical morphisms and squares.

Since biequivalences can be characterized as the $2$-functors which are surjective on objects up to equivalence, full on morphisms up to invertible $2$-cell, and fully faithful on $2$-cells, our double biequivalences admit a similar description. To give such a description, we introduce new notions of weak invertibility for horizontal morphisms and squares in a double category~$\AA$; namely, those of \emph{horizontal equivalences} and \emph{weakly horizontally invertible squares}, which correspond to the equivalences in the $2$-categories $\bfH \AA$ and $\cv\AA$, respectively. These notions were independently developed by Grandis and Par\'e in \cite[\S 2]{GraPar19}, where the weakly horizontally invertible squares are called \emph{equivalence cells}. Now the double biequivalences can be described as the double functors which are surjective on objects up to horizontal equivalence, full on horizontal morphisms up to vertically invertible square, surjective on vertical morphisms up to weakly horizontally invertible square, and fully faithful on squares. 

Double biequivalences are designed in such a way that a $2$-functor $F\colon \A\to \B$ is a biequivalence if and only if its associated horizontal double functor $\bbH F\colon \bbH \A\to \bbH\B$ is a double biequivalence. This can be seen as a first step towards showing that the homotopy theory of $2$-categories sits inside that of double categories. Note that ``surjectivity'' rather than ``fullness'' on vertical morphisms is necessary to achieve our goal of defining a model structure on~$\DblCat$ compatible with the horizontal embedding $\bbH\colon \TwoCat\to \DblCat$. Indeed, as we want~$\bbH$ to preserve weak equivalences, and as the $2$-category $\Eadj$ given by the free-living adjoint equivalence is biequivalent to the terminal category $\mathbbm 1$, the double functor $\bbH\Eadj\to \mathbbm 1$ should be a weak equivalence in $\DblCat$. It is then straightforward to check that such a double functor cannot be full on vertical morphisms, as there is no vertical morphism between the two distinct objects of the horizontal double category $\bbH\Eadj$. 

Our first main result, \cref{thm:modelstructonDblCat}, provides a model structure on the category of double categories in which the weak equivalences are precisely the double biequivalences, and which is obtained as a right-induced model structure from two copies of the Lack model structure on $\TwoCat$ along the functor $(\bfH,\cv)$.

\begin{theoremA} \label{MS-intro}
Consider the adjunction 
\begin{tz}
\node (A) at (0,0) {$\TwoCat\times \TwoCat$};
\node (B) at (3.25,0) {$\DblCat$};
\node at ($(B.east)-(0,4pt)$) {,};
\draw[->] ($(A.east)+(0,.25cm)$) to [bend left] node[above,scale=0.8]{$\bbH\sqcup \LV$} ($(B.west)+(0,.25cm)$);
\draw[->] ($(B.west)+(0,-.25cm)$) to [bend left] node[below,scale=0.8]{$(\bfH,\cv)$} ($(A.east)-(0,.25cm)$);
\node[scale=0.8] at ($(A.east)!0.5!(B.west)$) {$\bot$};
\end{tz}
where each copy of $\TwoCat$ is endowed with the Lack model structure. Then the right-induced model structure on $\DblCat$ exists. In particular, a double functor is a weak equivalence in this model structure if and only if it is a double biequivalence.
\end{theoremA}

Since the Lack model structure on $\TwoCat$ is cofibrantly generated, so is the model structure on $\DblCat$ constructed above. Moreover, every double category is fibrant, since all objects are fibrant in $\TwoCat$.

By taking a closer look at the homotopy equivalences in our model structure on $\DblCat$, we identify them as the double functors $F\colon \AA\to \BB$ such that there is a double functor $G\colon \BB\to \AA$ and two horizontal pseudo natural equivalences $\id_\AA\simeq GF$ and $FG\simeq \id_\BB$. In particular, the usual Whitehead theorem for model structures (see \cite[Lemma 4.24]{DS}) allows us to identify the double biequivalences between cofibrant double categories as the homotopy equivalences described above. 

In fact, we show in \cref{thm:Whitehead1} that a more lax version of this result, involving a horizontally pseudo double functor $G$, holds for an even larger class of double categories containing the cofibrant objects; this mirrors the definition of biequivalences in $\TwoCat$, which further supports the fact that our double biequivalences provide a good notion of weak equivalences between double categories. As a corollary, we retrieve the Whitehead theorem for $2$-categories mentioned above. 
\begin{theoremA}
Let $\AA$ and $\BB$ be double categories such that the underlying vertical category $U\bfV\BB$ is a disjoint union of copies of $\mathbbm 1$ and $\mathbbm 2$. Then a double functor $F\colon \AA\to \BB$ is a double biequivalence if and only if there is a normal horizontally pseudo double functor ${G\colon \BB\to \AA}$, and horizontal pseudo natural equivalences $\eta\colon \id_\AA\simeq GF$ and $\epsilon\colon FG\simeq \id_\BB$. 
\end{theoremA}

This Whitehead Theorem is reminiscent of a result by Grandis \cite[Theorem 4.4.5]{Grandis} which  characterizes the $1$-categorical version of our double biequivalences under a different assumption on the double categories involved; namely, that of horizontal invariance. In \cite[Definition 2.8]{WHI}, the authors introduce a notion of \emph{weakly horizontally invariant} double categories, and use them to prove yet another Whitehead Theorem for double biequivalences; see \cite[Theorem 8.1]{WHI}. Moreover, the weakly horizontally invariant double categories are identified as the fibrant objects in a different model structure on $\DblCat$, whose study is the purpose of \cite{WHI}.

We now address our original motivation of constructing a homotopy theory for double categories that contains that of $2$-categories through the horizontal embedding. Our model structure on $\DblCat$ successfully achieves this goal, and moreover, exhibits the greatest possible compatibility with respect to the horizontal embedding $\bbH\colon\TwoCat\to\DblCat$ that one could hope for, as studied in \cref{Sec:Quillenpairs}. Indeed, the functor $\bbH$ is both left and right Quillen, and homotopically fully faithful. This implies that $\bbH$ embeds the homotopy theory of $2$-categories in that of double categories in a reflective and coreflective way. Furthermore, the Lack model structure can be shown to be both left- and right-induced along $\bbH$ from our model structure on $\DblCat$. As a consequence, a $2$-functor $F$ is a cofibration, fibration or weak equivalence in $\TwoCat$ if and only if the double functor $\bbH F$ is a cofibration, fibration or weak equivalence in~$\DblCat$, respectively. This completely determines the model structure on $\TwoCat$ through its image under $\bbH$. 

\begin{theoremA}\label{theoremC}
The adjunctions 
\begin{tz}
\node[](A) {$\TwoCat$};
\node[right of=A,rr](B) {$\DblCat$};
\draw[->] ($(B.west)+(0,.25cm)$) to [bend right=50] node[above,la]{$L$} ($(A.east)+(0,.25cm)$);
\draw[->] (A) to node[la,over] {$\bbH$} (B);
\draw[->] ($(B.west)+(0,-.25cm)$) to [bend left=50] node[below,la]{$\bfH$} ($(A.east)-(0,.25cm)$);
\node[la] at ($(A.east)!0.5!(B.west)+(0,.35cm)$) {$\bot$};
\node[la] at ($(A.east)!0.5!(B.west)-(0,.35cm)$) {$\bot$};
\end{tz}
are both Quillen pairs between the Lack model structure on $\TwoCat$ and the model structure on~$\DblCat$ of \cref{MS-intro}. Moreover, the functor $\bbH$ is homotopically fully faithful, and the Lack model structure on $\TwoCat$ is both left- and right-induced from our model structure on~$\DblCat$ along $\bbH$. 
\end{theoremA}

Having established the exceptional behavior of our model structure with the horizontal embedding, we want to further investigate its relation with the Lack model structure on~$\TwoCat$. Lack shows in \cite{Lack2Cat} that the model structure on $\TwoCat$ is monoidal with respect to the Gray tensor product. In the double categorical setting, there is an analogous monoidal structure on $\DblCat$ given by the Gray tensor product constructed by B\"ohm in \cite{Bohm}. However, this monoidal structure is not compatible with our model structure on $\DblCat$ (see \cref{NotcompatiblewithGray}), since it treats the vertical and horizontal directions symmetrically, while our model structure does not. Nevertheless, restricting this Gray tensor product for double categories in one of the variables to $\TwoCat$ via $\bbH$ removes this symmetry and provides an enrichment of $\DblCat$ over~$\TwoCat$ that is compatible with our model structure. More precisely, this enrichment is given by the hom $2$-categories of double functors, horizontal pseudo natural transformations, and modifications between them, denoted by $\bfH[-,-]_\ps$.

\begin{theoremA}
The model structure on $\DblCat$ of \cref{MS-intro} is a $\TwoCat$-enriched model structure, where the enrichment is given by $\bfH[-,-]_\ps$. 
\end{theoremA}

The fact that horizontal pseudo natural transformations play a key role was to be expected, since they are the type of transformations that detect our weak equivalences, as established in our version of the Whitehead theorem above.

Just as the composition in $2$-categories can be weakened to obtain the notion of \emph{bicategories}, double categories also admit a weaker version, called \emph{weak double categories}, where horizontal composition is associative and unital up to vertically invertible squares. A bicategory $\B$ can then be seen as a horizontal weak double category $\bbH^{w} \B$. In \cite{LackBicat}, Lack shows that the category $\Bicats$ of bicategories and strict functors admits a model structure, in which the weak equivalences are again the biequivalences. Moreover, the full embedding of $2$-categories into bicategories induces a Quillen equivalence.

Similarly, we endow the category $\wkDblCat$ of weak double categories and strict double functors with a model structure, whose weak equivalences are again the double biequivalences, and show that analogue statements to those in \cref{MS-intro,theoremC} hold in this weaker setting. Furthermore, the full embedding of $\DblCat$ into $\wkDblCat$ is part of a Quillen equivalence, mirroring the Quillen equivalence between $\TwoCat$ and $\Bicats$ of~\cite{LackBicat}.

\begin{theoremA}
The inclusion $I\colon \DblCat\to \wkDblCat$ is a Quillen equivalence. 
\end{theoremA}

The model structures on $\TwoCat$, $\Bicats$, $\DblCat$, and $\wkDblCat$ interact as expected: we have a commutative square of right Quillen and homotopically fully faithful functors relating the model structures introduced in this paper and Lack's model structures.

\begin{tz}
\node (A) at (0,0) {$\TwoCat$};
\node (B) at (2.5,0) {$\Bicats$};
\node (C) at (0,-1.5) {$\DblCat$};
\node (D) at (2.5,-1.5) {$\wkDblCat$};
\draw[->] (A) to node[above,scale=0.8] {$I_2$} node[below,scale=0.8] {$\simeq_\mathrm{Q}$} (B);
\draw[->] (C) to node[above,scale=0.8] {$I$} node[below,scale=0.8] {$\simeq_\mathrm{Q}$} (D);
\draw[->] (A) to node[left, scale=0.8] {$\bbH$} (C);
\draw[->] (B) to node[right, scale=0.8] {$\bbH^{w}$} (D);
\end{tz}

\subsection*{Acknowledgements}
The authors would like to thank Martina Rovelli for reading an early version of this paper and providing many helpful comments; especially, for suggesting at the beginning of this project that we could induce a model structure on $\DblCat$ from two copies of $\TwoCat$. Philip Hackney and Martina Rovelli pointed out that the Lack model structure was further left-induced along the horizontal embedding from our model structure, and J\'er\^ome Scherer pointed us towards a better construction of path objects. The authors are also grateful to tslil clingman, for suggesting a construction that became our functor $\cv\colon \DblCat\to \TwoCat$, and to Viktoriya Ozornova and J\'er\^ome Scherer, for spotting many typos in previous versions.

This work started when the first- and third-named authors were at the Mathematical Sciences Research Institute in Berkeley, California, during the Spring 2020 semester. During the realization of this work, the first-named author was supported by the Swiss National Science Foundation under the project P1ELP2\_188039 and the Max Planck Institute of Mathematics. The third-named author was supported by an international Macquarie University Research Excellence Scholarship.

\section{Double categorical preliminaries} \label{section:prelim}

In this section, we recall the basic notions about double categories, and also introduce non-standard definitions and terminology that will be used throughout the paper. The reader familiar with double categories may wish to jump directly to \cref{defn_cv}.

\begin{defn}
A \textbf{double category} $\AA$ consists of
\begin{rome}
\item objects $A$, $B$, $C$, $\ldots$,
\item horizontal morphisms $a\colon A\to B$ with composition denoted by $b\circ a$ or $ba$,
\item vertical morphisms $u\colon A\arrowdot A'$ with composition denoted by $v\bullet u$ or $vu$,
\item squares (or cells) $\alpha\colon (u \; ^{a}_{\substack{b}} \; v)$ of the form 
\begin{tz}
\node (A) at (0,0) {$A$};
\node (B) at (1.5,0) {$B$};
\node (A') at (0,-1.5) {$A'$};
\node (B') at (1.5,-1.5) {$B'$};
\draw[->] (A) to node[above,scale=0.8] {$a$} (B);
\draw[->] (A') to node[below,scale=0.8] {$b$} (B');
\draw[->] (A) to node[left,scale=0.8] {$u\;$} (A');
\draw[->] (B) to node[right,scale=0.8] {$\;v$} (B');

\node at (0,-.75) {$\bullet$};
\node at (1.5, -.75) {$\bullet$};

\node[scale=0.8] at (.75,-.75) {$\alpha$};
\end{tz}
with both horizontal composition along their vertical boundaries and vertical composition along their horizontal boundaries, and
\item horizontal identities $\id_A\colon A\to A$ and vertical identities $e_A\colon A\arrowdot A$ for each object~$A$, vertical identity squares $\sq{e_a}{a}{a}{\id_A}{\id_B}$ for each horizontal morphism $a\colon A\to B$, horizontal identity squares $\sq{\id_u}{\id_A}{\id_{A'}}{u}{u}$ for each vertical morphism $u\colon A\arrowdot A'$, and identity squares $\square_A=\id_{e_A}=e_{\id_A}$ for each object $A$,
\end{rome}
such that all compositions are unital and associative, and such that the horizontal and vertical compositions of squares satisfy the interchange law.
\end{defn}

\begin{defn}
Let $\AA$ and $\BB$ be double categories. A \textbf{double functor} $F\colon \AA\to \BB$ consists of maps on objects, horizontal morphisms, vertical morphisms, and squares, which are compatible with domains and codomains and preserve all double categorical compositions and identities strictly.
\end{defn}

\begin{notation}
We write $\DblCat$ for the category of double categories and double functors.
\end{notation}

The category of double categories is cartesian closed, and therefore there is a double category whose objects are the double functors. We describe the horizontal morphisms, vertical morphisms, and squares of this hom double category.

\begin{defn} \label{def:modif}
Let $F,G,F',G'\colon \AA\to \BB$ be four double functors.

A \textbf{horizontal natural transformation} $h\colon F\Rightarrow G$ consists of 
\begin{rome}
    \item a horizontal morphism $h_A\colon FA\to GA$ in $\BB$, for each object $A\in \AA$, and
    \item a square $\sq{h_u}{h_A}{h_{A'}}{Fu}{Gu}$ in $\BB$, for each vertical morphism $u\colon A\arrowdot A'$ in $\AA$,
\end{rome}  
such that the assignment of squares is functorial with respect to the composition of vertical morphisms, and these data satisfy a naturality condition with respect to horizontal morphisms and squares.

Similarly, a \textbf{vertical natural transformation} $r\colon F\Rightarrow F'$ consists of 
\begin{rome}
    \item a vertical morphism $r_A\colon FA\arrowdot F'A$ in $\BB$, for each object $A\in \AA$, and
    \item a square $\sq{r_a}{Fa}{F'a}{r_A}{r_B}$ in $\BB$, for each horizontal morphism $a\colon A\to B$ in $\AA$,
\end{rome}  
satisfying transposed conditions.

Given another horizontal natural transformation $k\colon F'\Rightarrow G'$ and another vertical natural transformation $s\colon G\Rightarrow G'$, a \textbf{modification} $\sq{\mu}{h}{k}{r}{s}$ consists of 
\begin{rome}
    \item a square $\sq{\mu_A}{h_A}{k_A}{r_A}{s_A}$ in $\BB$, for each object $A\in \AA$,
\end{rome} 
satisfying horizontal and vertical coherence conditions with respect to the squares of the natural transformations $h$, $k$, $r$, and $s$.

See \cite[\S 3.2.7]{Grandis} for more explicit definitions.
\end{defn}

\begin{defn}\label{internalHom}
Let $\AA$ and $\BB$ be double categories. We define the \textbf{hom double category} $[\AA,\BB]$ whose 
\begin{rome}
    \item objects are the double functors $\AA\to \BB$,
    \item horizontal morphisms are the horizontal natural transformations,
    \item vertical morphisms are the vertical natural transformations, and
    \item squares are the modifications.
\end{rome} 
\end{defn}

\begin{prop}[{\cite[Proposition 2.11]{FPP}}] \label{prop:adjinternalhom}
For every double category $\AA$, there is an adjunction
\begin{tz}
\node (A) at (0,0) {$\DblCat$};
\node (B) at (2.75,0) {$\DblCat$};
\draw[->] ($(A.east)+(0,.25cm)$) to [bend left] node[above,scale=0.8]{$-\times \AA$} ($(B.west)+(0,.25cm)$);
\draw[->] ($(B.west)+(0,-.25cm)$) to [bend left] node[below,scale=0.8]{$[\AA,-]$} ($(A.east)-(0,.25cm)$);
\node[scale=0.8] at ($(A.east)!0.5!(B.west)$) {$\bot$};
\node at ($(B.east)-(0,4pt)$) {.};
\end{tz}
\end{prop}

It is possible to relax the strictness in \cref{def:modif,internalHom} to obtain the following notions.

\begin{defn} \label{def:pseudohortransf}
Let $F,G\colon \AA\to \BB$ be double functors. A \textbf{horizontal pseudo natural transformation} $h\colon F\Rightarrow G$ consists of 
\begin{rome}
    \item a horizontal morphism $h_A\colon FA\to GA$ in $\BB$, for each object $A\in \AA$, 
    \item a square $\sq{h_u}{h_A}{h_{A'}}{Fu}{Gu}$ in $\BB$, for each vertical morphism $u\colon A\arrowdot A'$ in $\AA$, and
    \item a vertically invertible square $\sq{h_a}{(Ga) h_A}{h_B (Fa)}{e_{FA}}{e_{GB}}$ in $\BB$, for each horizontal morphism $a\colon A\to B$ in $\AA$, expressing a pseudo naturality condition for horizontal morphisms.
\end{rome} 
These assignments of squares are functorial with respect to compositions of horizontal and vertical morphisms, and these data satisfy a naturality condition with respect to squares.

Similarly, one can define a transposed notion of \textbf{vertical pseudo natural transformation} between double functors.

A \textbf{modification} in a square of horizontal and vertical pseudo natural transformations is defined similarly to \cref{def:modif}, with the horizontal and vertical coherence conditions taking the pseudo data of the transformations into account.

See \cite[\S 3.8]{Grandis} or \cite[\S 2.2]{Bohm} for precise definitions.
\end{defn}

\begin{defn} \label{def:pseudohomdouble}
Let $\AA$ and $\BB$ be double categories. We define the \textbf{pseudo hom double category} $[\AA,\BB]_{\ps}$ whose
\begin{rome}
    \item objects are the double functors $\AA\to \BB$,
    \item horizontal morphisms are the horizontal pseudo natural transformations,
    \item vertical morphisms are the vertical pseudo natural transformations, and
    \item squares are the modifications.
\end{rome}
\end{defn}

Similar to the case of \cref{prop:adjinternalhom}, we will see in \cref{prop:Bohm} that the double functor $[\AA,-]_{\ps}$ admits a left adjoint $-\otimes_\Gray \AA$ \textemdash the Gray tensor product on $\DblCat$, introduced by B\"ohm in \cite{Bohm}.

As mentioned in the introduction, there is a full horizontal embedding of the category $\TwoCat$ of $2$-categories and $2$-functors into $\DblCat$. This is given by the following functor.

\begin{defn}
We define the \textbf{horizontal embedding functor} $\bbH\colon\TwoCat\to\DblCat$. It takes a $2$-category $\A$ to the double category $\bbH\A$ having the same objects as $\A$, the morphisms of $\A$ as horizontal morphisms, only identities as vertical morphisms, and squares
\begin{tz}
\node (A) at (0,0) {$A$};
\node (C) at (0,-1.5) {$A$};
\node (B) at (1.5,0) {$B$};
\node (D) at (1.5,-1.5) {$B$};
\draw[->] (A) to node[above,scale=0.8] {$a$} (B);
\draw[->] (C) to node[below,scale=0.8] {$b$} (D);
\draw[d] (A) to (C);
\draw[d] (B) to (D);
\node at (0,-.75) {$\bullet$};
\node at (1.5,-.75) {$\bullet$};

\node[scale=0.8] at (.75,-.75) {$\alpha$};
\end{tz}
given by the $2$-cells $\alpha\colon a\Rightarrow b$ in $\A$. It sends a $2$-functor $F\colon \A\to \B$ to the double functor $\bbH F\colon \bbH \A\to \bbH\B$ that acts as $F$ does on the corresponding data.
\end{defn}

The functor $\bbH$ admits a right adjoint given by the following.

\begin{defn}
We define the functor $\bfH\colon\DblCat\to\TwoCat$. It takes a double category $\AA$ to its \textbf{underlying horizontal $2$-category} $\bfH\AA$, i.e., the $2$-category whose objects are the objects of $\AA$, whose morphisms are the horizontal morphisms of $\AA$, and whose $2$-cells $\alpha\colon a\Rightarrow b$ are given by the squares in $\AA$ of the form 
\begin{tz}
\node (A) at (0,0) {$A$};
\node (C) at (0,-1.5) {$A$};
\node (B) at (1.5,0) {$B$};
\node (D) at (1.5,-1.5) {$B$};
\node at ($(D.east)-(0,4pt)$) {.};
\draw[->] (A) to node[above,scale=0.8] {$a$} (B);
\draw[->] (C) to node[below,scale=0.8] {$b$} (D);
\draw[d] (A) to (C);
\draw[d] (B) to (D);
\node at (0,-.75) {$\bullet$};
\node at (1.5,-.75) {$\bullet$};

\node[scale=0.8] at (.75,-.75) {$\alpha$};
\end{tz}
It sends a double functor $F\colon \AA\to \BB$ to the $2$-functor $\bfH F\colon \bfH \AA\to \bfH\BB$ that acts as $F$ does on the corresponding data.
\end{defn}

\begin{prop}[{\cite[Proposition 2.5]{FPP}}]\label{prop:adjHH}
The functors $\bfH$ and $\bbH$ form an adjunction
\begin{tz}
\node (A) at (0,0) {$\TwoCat$};
\node (B) at (2.5,0) {$\DblCat$};
\node at ($(B.east)-(0,4pt)$) {.};
\draw[->] ($(A.east)+(0,.25cm)$) to [bend left] node[above,scale=0.8]{$\bbH$} ($(B.west)+(0,.25cm)$);
\draw[->] ($(B.west)+(0,-.25cm)$) to [bend left] node[below,scale=0.8]{$\bfH$} ($(A.east)-(0,.25cm)$);
\node[scale=0.8] at ($(A.east)!0.5!(B.west)$) {$\bot$};
\end{tz}
Moreover, the unit $\eta\colon \id\Rightarrow \bfH\bbH$ is the identity.
\end{prop}

\begin{rmk}
Similarly, we can define a functor $\bbV\colon \TwoCat\to \DblCat$, sending a $2$-category to its associated vertical double category with only trivial horizontal morphisms, and a functor $\bfV\colon \DblCat\to \TwoCat$, sending a double category to its underlying vertical $2$-category. They also form an adjunction $\bbV\dashv \bfV$. 
\end{rmk}

We now introduce a new functor between $\DblCat$ and $\TwoCat$ that extracts, from a double category, a $2$-category whose objects and morphisms are the vertical morphisms and squares; this is the functor $\cv$ mentioned in the introduction. In order to do this, we use the category~$\vtwo$, where $\mathbbm{2}$ is the ($2$-)category $\{0\to 1\}$ free on a morphism. This double category $\vtwo$ is therefore the double category free on a vertical morphism.

\begin{defn}\label{defn_cv}
We define the functor $\cv\colon \DblCat\to \TwoCat$ to be the functor $\bfH[\vtwo,-]$. More explicitly, it sends a double category $\AA$ to the $2$-category $\cv \AA=\bfH[\vtwo,\AA]$ given by the following data.
\begin{rome}
    \item An object in $\mathcal V \AA$ is a vertical morphism $u\colon A\arrowdot A'$ in $\AA$.
    \item A morphism $(a,b,\alpha)\colon u\to v$ is a square in $\AA$ of the form
\begin{tz}
\node (A) at (0,0) {$A$};
\node (B) at (1.5,0) {$B$};
\node (A') at (0,-1.5) {$A'$};
\node (B') at (1.5,-1.5) {$B'$};
\node at ($(B'.east)-(0,4pt)$) {.};
\draw[->] (A) to node[above, scale=0.8] {$a$} (B);
\draw[->] (A') to node[below, scale=0.8] {$b$} (B');
\draw[->] (A) to node[left,scale=0.8] {$u\;$} (A');
\draw[->] (B) to node[right,scale=0.8] {$\;v$} (B');

\node at (0,-.75) {$\bullet$};
\node at (1.5, -.75) {$\bullet$};

\node[scale=0.8] at (.75,-.75) {$\alpha$};
\end{tz}
    \item A $2$-cell $(\sigma_0,\sigma_1)\colon (a,b,\alpha)\Rightarrow (c,d,\beta)$ consists of two squares $\sq{\sigma_0}{a}{c}{e_A}{e_B}$ and $\sq{\sigma_1}{b}{d}{e_{A'}}{e_{B'}}$ in $\AA$ such that the following pasting equality holds.
\begin{tz}
\node (A) at (0,0){$A$};
\node (B) at (1.5,0){$B$};
\node (A') at (0,-1.5){$A$};
\node (B') at (1.5,-1.5){$B$};
\node (A'') at (0,-3){$A'$};
\node (B'') at (1.5,-3){$B'$};
\draw[d] (A) to (A');
\draw[d] (B) to (B');
\draw[->] (A) to node[scale=0.8,above]{$a$} (B);
\draw[->] (A') to node[scale=0.8,above]{$c$} (B');
\draw[->] (A') to node[scale=0.8,left]{$u\;$} (A'');
\draw[->] (B') to node[scale=0.8,right]{$\;v$}(B'');
\draw[->] (A'') to node[scale=0.8,below]{$d$} (B'');

\node at (0,-.75) {$\bullet$};
\node at (1.5,-.75) {$\bullet$};
\node at (0,-2.25) {$\bullet$};
\node at (1.5,-2.25) {$\bullet$};

\node[scale=0.8] at (.75,-.75) {$\sigma_0$};
\node[scale=0.8] at (.75,-2.25) {$\beta$};

\node at (2.5,-1.5) {$=$};

\node (A) at (3.5,0){$A$};
\node (B) at (5,0){$B$};
\node (A') at (3.5,-1.5){$A'$};
\node (B') at (5,-1.5){$B'$};
\node (A'') at (3.5,-3){$A'$};
\node (B'') at (5,-3){$B'$};
\draw[d] (A') to (A'');
\draw[d] (B') to (B'');
\draw[->] (A) to node[scale=0.8,above]{$a$} (B);
\draw[->] (A') to node[scale=0.8,above]{$b$} (B');
\draw[->] (A) to node[scale=0.8,left]{$u\;$} (A');
\draw[->] (B) to node[scale=0.8,right]{$\;v$}(B');
\draw[->] (A'') to node[scale=0.8,below]{$d$} (B'');

\node at (3.5,-.75) {$\bullet$};
\node at (5,-.75) {$\bullet$};
\node at (3.5,-2.25) {$\bullet$};
\node at (5,-2.25) {$\bullet$};

\node[scale=0.8] at (4.25,-.75) {$\alpha$};
\node[scale=0.8] at (4.25,-2.25) {$\sigma_1$};

\end{tz}
\end{rome}
\end{defn}

\begin{prop} \label{prop:adjLV}
The functor $\cv$ has a left adjoint $\LV$
\begin{tz}
\node (A) at (0,0) {$\TwoCat$};
\node (B) at (2.5,0) {$\DblCat$};
\draw[->] ($(A.east)+(0,.25cm)$) to [bend left] node[above,scale=0.8]{$\LV$} ($(B.west)+(0,.25cm)$);
\draw[->] ($(B.west)+(0,-.25cm)$) to [bend left] node[below,scale=0.8]{$\cv$} ($(A.east)-(0,.25cm)$);
\node[scale=0.8] at ($(A.east)!0.5!(B.west)$) {$\bot$};
\end{tz}
given by $\LV=\bbH(-)\times\vtwo$.
\end{prop}

\begin{proof}
By definition, the functor $\cv\colon \DblCat\to \TwoCat$ is given by the composite
\begin{tz}
\node(A) at (-.5,0) {$\DblCat$};
\node(B) at (2,0) {$\DblCat$};
\node (C) at (4,0) {$\TwoCat$.};
\draw[->] (A) to node[above,scale=0.8] {$[\vtwo,-]$} (B);
\draw[->] (B) to node[above,scale=0.8] {$\bfH$} (C);
\end{tz}
Since $[\vtwo,-]$ has a left adjoint $-\times \vtwo$ given by \cref{prop:adjinternalhom}, and $\bfH$ has a left adjoint~$\bbH$ given by \cref{prop:adjHH}, it follows that $\cv$ has a left adjoint given by the composite of the two left adjoints, namely,  $\LV=\bbH(-)\times\vtwo$.
\end{proof}

\begin{notation} \label{not:pseudohom}
Given two $2$-categories $\A$ and $\B$, we denote by $\Psd(\A,\B)$ the $2$-category of $2$-functors from $\A$ to $\B$, pseudo natural transformations, and modifications. 
\end{notation}

The following technical result, which exhibits the behavior of the functors $\bbH$, $\bfH$, and $\cv$ with respect to pseudo homs, will be of use when we prove the existence of the desired model structure. 

\begin{lemma} \label{lem:HVpreservehom}
Let $\B$ be a $2$-category and $\AA$ be a double category. Then there are isomorphisms of $2$-categories 
\[ \bfH[\bbH\B,\AA]_\ps\cong \Psd[\B,\bfH\AA] \ \ \text{and} \ \ \cv[\bbH\B,\AA]_\ps \cong \Psd[\B,\cv\AA] \]
natural in $\B$ and $\AA$. 
\end{lemma}

\begin{proof}
We first consider the isomorphism $\bfH[\bbH\B,\AA]_\ps\cong \Psd[\B,\bfH\AA]$. On objects, this follows from the adjunction $\bbH\dashv \bfH$ given in \cref{prop:adjHH}. On morphisms, as there are no non-trivial vertical morphisms in $\bbH\B$, horizontal pseudo natural transformations out of $\bbH\B$ are canonically the same as pseudo natural transformations out of $\B$. The argument for $2$-morphisms is similar. 

For the second isomorphism, first note that $[\vtwo,\AA]_\ps=[\vtwo,\AA]$, since there are no non-trivial horizontal morphisms in $\vtwo$, and therefore horizontal pseudo natural transformations out of $\vtwo$ correspond to horizontal (strict) natural transformations out of $\vtwo$. Therefore, we have that 
\begin{align*}
    \cv[\bbH\B,\AA]_\ps&=\bfH[\vtwo,[\bbH\B,\AA]_\ps]_\ps  \cong \bfH[\bbH\B,[\vtwo,\AA]_\ps]_\ps \\
    & \cong \Psd[\B,\bfH[\vtwo,\AA]_\ps] =\Psd[\B,\cv\AA],
\end{align*} 
where the first isomorphism follows from the symmetry of the Gray tensor product on $\DblCat$; see \cref{prop:Bohm} below.
\end{proof}

We conclude this section by introducing new notions of weak invertibility for horizontal morphisms and squares in a double category, together with some technical results that will be of use later in the paper. We do not prove these results here, but instead refer the reader to work by the first author \cite[Appendix A]{Moserforth}. These notions and results were independently developed by Grandis and Par\'e in \cite[\S 2]{GraPar19}.

\begin{defn} \label{def:horeq}
A horizontal morphism $a\colon A\to B$ in a double category $\AA$ is a \textbf{horizontal equivalence} if it is an equivalence in the $2$-category $\bfH \AA$. 
\end{defn}

\begin{defn} \label{def:weakinvsq}
A square $\sq{\alpha}{a}{b}{u}{v}$ in a double category $\AA$ is \textbf{weakly horizontally invertible} if it is an equivalence in the $2$-category $\cv \AA$. See \cite[Definition 2.3]{WHI} for a more detailed description. 
\end{defn}

\begin{rmk}
In particular, the horizontal boundaries $a$ and $b$ of a weakly horizontally invertible square $\alpha$ are horizontal equivalences, which we refer to as the \emph{horizontal equivalence data} of $\alpha$. 
\end{rmk}

Since an equivalence in a $2$-category can always be promoted to an adjoint equivalence (see, for example, \cite[Lemma 2.1.11]{RiehlVerity}), we get the following result.

\begin{lemma}
Every horizontal equivalence can be promoted to a horizontal \emph{adjoint} equivalence. Similarly, every weakly horizontally invertible square can be promoted to one with horizontal \emph{adjoint} equivalence data.
\end{lemma}

Finally, we conclude with a result concerning weakly horizontally invertible squares.

\begin{lemma}[{\cite[Lemma A.2.1]{Moserforth}}] \label{lem:weaklyhorinvvsverinv}
A square whose horizontal boundaries are horizontal equivalences, and whose vertical boundaries are identities, is weakly horizontally invertible if and only if it is vertically invertible. 
\end{lemma} 

\begin{rmk}
It follows that, for a $2$-category $\A$, a weakly horizontally invertible square in the double category $\bbH \A$ corresponds to an invertible $2$-cell in $\A$. 
\end{rmk}

\section{Model structure for double categories}\label{Sec:modelstructDblCat}

This section contains our first main result, which proves the existence of a model structure on $\DblCat$ that is right-induced along the functor $(\bfH,\cv)\colon \DblCat\to \TwoCat\times \TwoCat$, where both copies of $\TwoCat$ are endowed with the Lack model structure. After recalling the main features of the Lack model structure on $\TwoCat$ in \cref{subsec:LackMS}, we define in \cref{subsec:MSDblCat} notions of \emph{double biequivalences} and \emph{double fibrations} in $\DblCat$, which extend the notions of biequivalences and fibrations in $\TwoCat$. These appear to be exactly the weak equivalences and fibrations of the right-induced model structure mentioned above. We then prove, using a result inspired by the Quillen Path Object Argument, that this right-induced model structure on $\DblCat$ exists.

\subsection{Lack model structure on \texorpdfstring{$\TwoCat$}{2Cat}}\label{subsec:LackMS}

We start by recalling the main features of Lack's model structure on $\TwoCat$; see \cite{Lack2Cat,LackBicat}. Its class of weak equivalences is given by the \emph{biequivalences}, and we refer to the fibrations in this model structure as \emph{Lack fibrations}.

\begin{defn} \label{biequiv}
Given $2$-categories $\A$ and $\B$, a $2$-functor $F\colon \A\to \B$ is a \textbf{biequivalence} if
\begin{enumerate}
    \item[(b1)] for every object $B\in \B$, there is an object $A\in \A$ and an equivalence $B\xrightarrow{\simeq} FA$ in~$\B$,
    \item[(b2)] for every pair of objects $A,C\in \A$ and every morphism $b\colon FA\to FC$ in $\B$, there is a morphism $a\colon A\to C$ in $\A$ and an invertible $2$-cell $b\cong Fa$ in $\B$, and
    \item[(b3)] for every pair of morphisms $a,c\colon A\to C$ in $\A$ and every $2$-cell $\beta\colon Fa\Rightarrow Fc$ in $\B$, there is a unique $2$-cell $\alpha\colon a\Rightarrow c$ in $\A$ such that $F\alpha=\beta$.
\end{enumerate}
\end{defn}

\begin{defn} \label{Lackfib}
Given $2$-categories $\A$ and $\B$, a $2$-functor $F\colon \A\to \B$ is a \textbf{Lack fibration} if
\begin{enumerate}
    \item[(f1)] for every object $C\in \A$ and every equivalence $b\colon B\xrightarrow{\simeq} FC$ in $\B$, there is an equivalence $a\colon A\xrightarrow{\simeq} C$ in~$\A$ such that $Fa=b$, and
    \item[(f2)] for every morphism $c\colon A\to C$ in $\A$ and every invertible $2$-cell $\beta\colon b\cong Fc$ in $\B$, there is an invertible $2$-cell $\alpha\colon a\cong c$ in $\A$ such that $F\alpha=\beta$. 
\end{enumerate}
\end{defn}

There exists a model structure on $\TwoCat$ determined by the classes above.

\begin{thm}[{\cite[Theorem 4]{LackBicat}}] \label{thm:LackMS}
There is a cofibrantly generated model structure on~$\TwoCat$, called the \emph{Lack model structure}, 
in which the weak equivalences are the biequivalences and the fibrations are the Lack fibrations. 
\end{thm}

\begin{rmk}\label{rmk:2catsarefibrant}
Note that every $2$-category is fibrant in the Lack model structure.
\end{rmk}

Recall that a \emph{monoidal model category} is a closed monoidal category which admits a model structure compatible with the internal homs in the following sense: the pullback-hom of a cofibration and a fibration is a fibration, which is trivial if one of them is a weak equivalence (see \cite[Definition 5.1]{Moser}). The Lack model structure on $\TwoCat$ is monoidal with respect to the Gray tensor product, whose definition we now recall. 

\begin{defn} \label{def:Gray2}
The \textbf{Gray tensor product} $\otimes_2\colon \TwoCat\times \TwoCat\to \TwoCat$ is defined by the following universal property: for all $2$-categories $\A$, $\B$, and $\C$, there is an isomorphism
\[ \TwoCat(\A\otimes_2\B,\C)\cong \TwoCat(\A, \Psd[\B,\C]) \] natural in $\A$, $\B$, and $\cc$, where $\Psd[-,-]$ is the $2$-category given in \cref{not:pseudohom}.
\end{defn}

\begin{thm}[{\cite[Theorem 7.5]{Lack2Cat}}]\label{2Catenrichedmodel}
The category $\TwoCat$ endowed with the Lack model structure is a monoidal model category with respect to the closed monoidal structure given by the Gray tensor product.
\end{thm}

\subsection{Constructing the model structure for \texorpdfstring{$\DblCat$}{DblCat}} \label{subsec:MSDblCat}

We define double biequivalences in~$\DblCat$ inspired by the characterization of biequivalences in $\TwoCat$ in terms of $2$-functors that are bi-essentially surjective on objects, essentially full on morphisms, and fully faithful on $2$-cells. Our convention of regarding $2$-categories as horizontal double categories justifies the choice of directions when emulating this characterization of biequivalences in the context of double categories. Thus, a double biequivalence will be required to be bi-essentially surjective on objects up to a \emph{horizontal equivalence} (see \cref{def:horeq}), essentially full on horizontal morphisms, and fully faithful on squares. However, this does not take into account the vertical structure of double categories, and so we need to add a condition of bi-essential surjectivity on vertical morphisms given up to a \emph{weakly horizontally invertible square} (see \cref{def:weakinvsq}). 

\begin{defn}\label{doublequiv}
Given double categories $\AA$ and $\BB$, a double functor $F\colon \AA\to \BB$ is a \textbf{double biequivalence} if 
\begin{enumerate}
    \item[(db1)] for every object $B\in \BB$, there is an object $A\in \AA$ and a horizontal equivalence $B\xrightarrow{\simeq} FA$ in $\BB$,
    \item[(db2)] for every pair of objects $A,C\in \AA$ and every horizontal morphism $b\colon FA\to FC$ in $\BB$, there is a horizontal morphism $a\colon A\to C$ in $\AA$ and a vertically invertible square in $\BB$ of the form
\begin{tz}
\node (A) at (0,0) {$FA$};
\node (B) at (1.5,0) {$FC$};
\node (A') at (0,-1.5) {$FA$};
\node (B') at (1.5,-1.5) {$FC$};
\node at ($(B'.east)-(0,4pt)$) {,};
\draw[->] (A) to node[above, scale=0.8] {$b$} (B);
\draw[->] (A') to node[below, scale=0.8] {$Fa$} (B');
\draw[d] (A) to (A');
\draw[d] (B) to (B');

\node at (0,-.75) {$\bullet$};
\node at (1.5, -.75) {$\bullet$};

\node[scale=0.8] at (.75,-.75) {$\vcong$};
\end{tz}
    \item[(db3)] for every vertical morphism $v\colon B\arrowdot B'$ in $\BB$, there is a vertical morphism $u\colon A\arrowdot A'$ in $\AA$ and a weakly horizontally invertible square in $\BB$ of the form
\begin{tz}
\node (A) at (0,0) {$B$};
\node (B) at (1.5,0) {$FA$};
\node (A') at (0,-1.5) {$B'$};
\node (B') at (1.5,-1.5) {$FA'$};
\node at ($(B'.east)-(0,4pt)$) {,};
\draw[->] (A) to node[above, scale=0.8] {$\simeq$} (B);
\draw[->] (A') to node[below, scale=0.8] {$\simeq$} (B');
\draw[->] (A) to node[left,scale=0.8] {$v\;$} (A');
\draw[->] (B) to node[right,scale=0.8] {$\;Fu$} (B');

\node at (0,-.75) {$\bullet$};
\node at (1.5, -.75) {$\bullet$};

\node[scale=0.8] at (.75,-.75) {$\simeq$};
\end{tz}
    \item[(db4)] for every data in $\AA$ as below left, and every square in $\BB$ as below right,
\begin{tz}
\node (A) at (0,0) {$A$};
\node (B) at (1.5,0) {$C$};
\node (A') at (0,-1.5) {$A'$};
\node (B') at (1.5,-1.5) {$C'$};
\draw[->] (A) to node[above, scale=0.8] {$a$} (B);
\draw[->] (A') to node[below, scale=0.8] {$c$} (B');
\draw[->] (A) to node[left,scale=0.8] {$u\;$} (A');
\draw[->] (B) to node[right,scale=0.8] {$\;u'$} (B');

\node at (0,-.75) {$\bullet$};
\node at (1.5, -.75) {$\bullet$};

\node (A) at (4,0) {$FA$};
\node (B) at (5.5,0) {$FC$};
\node (A') at (4,-1.5) {$FA'$};
\node (B') at (5.5,-1.5) {$FC'$};
\draw[->] (A) to node[above, scale=0.8] {$Fa$} (B);
\draw[->] (A') to node[below, scale=0.8] {$Fc$} (B');
\draw[->] (A) to node[left,scale=0.8] {$Fu\;$} (A');
\draw[->] (B) to node[right,scale=0.8] {$\;Fu'$} (B');

\node at (4,-.75) {$\bullet$};
\node at (5.5, -.75) {$\bullet$};

\node[scale=0.8] at (4.75,-.75) {$\beta$};
\end{tz}
    there is a unique square $\sq{\alpha}{a}{c}{u}{u'}$ in $\mathbb A$ such that $F\alpha=\beta$.
\end{enumerate}
\end{defn}

\begin{rmk}
In $\TwoCat$, one can prove that a $2$-functor $F\colon\A\to\B$ is a biequivalence if and only if there is a pseudo functor $G\colon\B\to\A$ together with pseudo natural equivalences $\id_\A\simeq GF$ and $FG\simeq \id_\B$. Under certain hypotheses, we can show a similar characterization of double biequivalences using  \emph{horizontal} pseudo natural equivalences. This is done in \cref{subsec:whitehead}.
\end{rmk}

Similarly to the definition of double biequivalence, we take inspiration from the Lack fibrations to define a notion of \emph{double fibrations}.

\begin{defn}\label{doublefib}
Given double categories $\AA$ and $\BB$, a double functor $F\colon \AA\to \BB$ is a \textbf{double fibration} if 
\begin{enumerate}
    \item[(df1)] for every object $C\in \AA$ and every horizontal equivalence $b\colon B\xrightarrow{\simeq} FC$ in $\BB$, there is a horizontal equivalence $a\colon A\xrightarrow{\simeq} C$ in $\AA$ such that $Fa=b$,
    \item[(df2)] for every horizontal morphism $c\colon A\to C$ in $\AA$ and for every vertically invertible square $\sq{\beta}{b}{Fc}{e_{FA}}{e_{FC}}$ in $\BB$ as depicted below left, there is a vertically invertible square $\sq{\alpha}{a}{c}{e_A}{e_C}$ in $\AA$ as depicted below right such that $F\alpha=\beta$,
\begin{tz}
\node (A) at (0,0) {$FA$};
\node (B) at (1.5,0) {$FC$};
\node (A') at (0,-1.5) {$FA$};
\node (B') at (1.5,-1.5) {$FC$};
\draw[->] (A) to node[above, scale=0.8] {$b$} (B);
\draw[->] (A') to node[below, scale=0.8] {$Fc$} (B');
\draw[d] (A) to (A');
\draw[d] (B) to (B');

\node at (0,-.75) {$\bullet$};
\node at (1.5, -.75) {$\bullet$};

\node[scale=0.8] at (.6,-.75) {$\beta$};
\node[scale=0.8] at (.9,-.75) {$\vcong$};

\node (A) at (4,0) {$A$};
\node (B) at (5.5,0) {$C$};
\node (A') at (4,-1.5) {$A$};
\node (B') at (5.5,-1.5) {$C$};
\draw[->] (A) to node[above, scale=0.8] {$a$} (B);
\draw[->] (A') to node[below, scale=0.8] {$c$} (B');
\draw[d] (A) to (A');
\draw[d] (B) to (B');

\node at (4,-.75) {$\bullet$};
\node at (5.5, -.75) {$\bullet$};

\node[scale=0.8] at (4.6,-.75) {$\alpha$};
\node[scale=0.8] at (4.9,-.75) {$\vcong$};
\end{tz}
    \item[(df3)] for every vertical morphism $u'\colon C\arrowdot C'$ in $\AA$ and every weakly horizontally invertible square $\sq{\beta}{\simeq}{\simeq}{v}{Fu'}$ in $\BB$ as depicted below left, there is a weakly horizontally invertible square $\sq{\alpha}{\simeq}{\simeq}{u}{u'}$ in $\AA$ as depicted below right such that $F\alpha=\beta$.
\begin{tz}
\node (A) at (0,0) {$B$};
\node (B) at (1.5,0) {$FC$};
\node (A') at (0,-1.5) {$B'$};
\node (B') at (1.5,-1.5) {$FC'$};
\draw[->] (A) to node[above, scale=0.8] {$\simeq$} (B);
\draw[->] (A') to node[below, scale=0.8] {$\simeq$} (B');
\draw[->] (A) to node[left,scale=0.8] {$v\;$} (A');
\draw[->] (B) to node[right,scale=0.8] {$\;Fu'$} (B');

\node at (0,-.75) {$\bullet$};
\node at (1.5, -.75) {$\bullet$};

\node[scale=0.8] at (.6,-.75) {$\beta$};
\node[scale=0.8] at (.9,-.75) {$\simeq$};

\node (A) at (4,0) {$A$};
\node (B) at (5.5,0) {$C$};
\node (A') at (4,-1.5) {$A'$};
\node (B') at (5.5,-1.5) {$C'$};
\draw[->] (A) to node[above, scale=0.8] {$\simeq$} (B);
\draw[->] (A') to node[below, scale=0.8] {$\simeq$} (B');
\draw[->] (A) to node[left,scale=0.8] {$u\;$} (A');
\draw[->] (B) to node[right,scale=0.8] {$\;u'$} (B');

\node at (4,-.75) {$\bullet$};
\node at (5.5, -.75) {$\bullet$};

\node[scale=0.8] at (4.6,-.75) {$\alpha$};
\node[scale=0.8] at (4.9,-.75) {$\simeq$};
\end{tz}
\end{enumerate}
\end{defn}

By requiring that a double functor is both a double biequivalence and a double fibration, we get a notion of \emph{double trivial fibration}, which can be described as follows. 

\begin{defn} \label{doubletrivfib}
Given double categories $\AA$ and $\BB$, a double functor $F\colon \AA\to \BB$ is a \textbf{double trivial fibration} if it satisfies (db4) of \cref{doublequiv}, and the following conditions:
\begin{enumerate}
    \item[(dt1)] for every object $B\in \BB$, there is an object $A\in \AA$ such that $B=FA$,
    \item[(dt2)] for every pair of objects $A,C\in \AA$ and every horizontal morphism $b\colon FA\to FC$ in $\BB$, there is a horizontal morphism $a\colon A\to C$ in $\AA$ such that $b=Fa$, and
    \item[(dt3)] for every vertical morphism $v\colon B\arrowdot B'$ in $\BB$, there is a vertical morphism $u\colon A\arrowdot A'$ in $\AA$ such that $v=Fu$.
\end{enumerate}
\end{defn}

\begin{rmk}
Note that (dt2) says that a double trivial fibration is \emph{full} on horizontal morphisms, while (dt3) says that a double trivial fibration is only \emph{surjective} on vertical morphisms.
\end{rmk}

We can characterize double biequivalences and double fibrations through biequivalences and Lack fibrations in $\TwoCat$. Recall the functors $\bfH, \cv\colon \DblCat\to \TwoCat$ defined in \cref{section:prelim}, which respectively  extract from a double category its underlying horizontal $2$-category and a $2$-category whose objects and morphisms are given by its vertical morphisms and squares. 
Then double biequivalences and double fibrations can be characterized as the double functors whose images under both $\bfH$ and $\cv$ are biequivalences or Lack fibrations. This is intuitively sound, since horizontal equivalences and weakly horizontally invertible squares were defined to be the equivalences in the $2$-categories induced by $\bfH$ and~$\cv$, respectively. We state these characterizations here, and defer their proofs to \cref{subsec:charfibeq}.

\begin{prop}\label{charequiv}
Given double categories $\AA$ and $\BB$, a double functor $F\colon \AA\to \BB$ is a double biequivalence in $\DblCat$ if and only if the $2$-functors $\bfH F\colon \bfH\AA\to \bfH \BB$ and $\cv F\colon \cv \AA\to \cv\BB$ are biequivalences in $\TwoCat$.
\end{prop}

\begin{prop} \label{charfib}
Given double categories $\AA$ and $\BB$, a double functor $F\colon \AA\to \BB$ is a double fibration in $\DblCat$ if and only if the $2$-functors $\bfH F\colon \bfH\AA\to \bfH \BB$ and $\cv F\colon \cv \AA\to \cv\BB$ are Lack fibrations in~$\TwoCat$.
\end{prop}

As a corollary, we get a similar characterization for double trivial fibrations.

\begin{cor} \label{chartrivfib}
Given double categories $\AA$ and $\BB$, a double functor $F\colon \AA\to \BB$ is a double trivial fibration in $\DblCat$ if and only if the $2$-functors $\bfH F\colon \bfH\AA\to \bfH \BB$ and $\cv F\colon \cv \AA\to \cv\BB$ are trivial fibrations in the Lack model structure on $\TwoCat$.
\end{cor}

\begin{proof}
It is a routine exercise to check that a double functor that is both a double biequivalence and a double fibration is precisely a double trivial fibration as defined in \cref{doubletrivfib}. Therefore, this follows directly from \cref{charequiv,charfib}.
\end{proof}

In order to build a model structure on $\DblCat$ with these classes of morphisms as its weak equivalences and (trivial) fibrations, we make use of the notion of \emph{right-induced model structure}. Given a model category $\cm$ and an adjunction 
\begin{diagram}[baseline=(current bounding box.center)]
\label{adjunction}
\node (A) at (0,0) {$\mathcal M$};
\node (B) at (2,0) {$\mathcal N$};
\node at ($(B.east)-(0,4pt)$) {,};
\draw[->] ($(A.east)+(0,.25cm)$) to [bend left] node[above,scale=0.8]{$L$} ($(B.west)+(0,.25cm)$);
\draw[->] ($(B.west)+(0,-.25cm)$) to [bend left] node[below,scale=0.8]{$R$} ($(A.east)-(0,.25cm)$);
\node[scale=0.8] at ($(A.east)!0.5!(B.west)$) {$\bot$};
\end{diagram}
we can, under certain conditions, induce a model structure on $\mathcal N$ along the right adjoint~$R$, in which a weak equivalence (resp.~fibration) is a morphism $F$ in $\mathcal{N}$ such that $RF$ is a weak equivalence (resp.\ fibration) in $\mathcal{M}$. 

\cref{charequiv,charfib} suggest that the model structure on $\DblCat$ we desire, with double biequivalences as the weak equivalences and double fibrations as the fibrations, corresponds to the right-induced model structure, if it exists, along the adjunction
\begin{tz}
\node (A) at (0,0) {$\TwoCat\times \TwoCat$};
\node (B) at (3.25,0) {$\DblCat$};
\node at ($(B.east)-(0,4pt)$) {,};
\draw[->] ($(A.east)+(0,.25cm)$) to [bend left] node[above,scale=0.8]{$\bbH\sqcup \LV$} ($(B.west)+(0,.25cm)$);
\draw[->] ($(B.west)+(0,-.25cm)$) to [bend left] node[below,scale=0.8]{$(\bfH,\cv)$} ($(A.east)-(0,.25cm)$);
\node[scale=0.8] at ($(A.east)!0.5!(B.west)$) {$\bot$};
\end{tz}
where each copy of $\TwoCat$ is endowed with the Lack model structure. To prove the existence of this model structure, we use results by Garner, Hess, K\c{e}dziorek, Riehl, and Shipley in~\cite{GKR,HKRS}. In particular, we use the following theorem, inspired by the original Quillen Path Object Argument \cite{Quillen}. 

\begin{thm} \label{path}
Let $\cm$ be an accessible model category, and let $\cn$ be a locally presentable category. Suppose we have an adjunction $L\dashv R$ between them as in (\ref{adjunction}). Suppose moreover that every object in $\cm$ is fibrant and that, for every object $X\in \cn$, there is a factorization
\[ X \stackrel{W}{\longrightarrow} \mathrm{Path}(X) \stackrel{P}{\longrightarrow} X\times X \]
of the diagonal morphism in $\cn$ such that $RP$ is a fibration in $\cm$ and $RW$ is a weak equivalence in $\cm$. Then the right-induced model structure on $\cn$ exists.
\end{thm}

\begin{proof}
This follows directly from \cite[Theorem 6.2]{Moser}, which is the dual of \cite[Theorem~2.2.1]{HKRS}. Indeed, if every object in $\cm$ is fibrant, then the underlying fibrant replacement of conditions~(i) and (ii) of \cite[Theorem 6.2]{Moser} are trivially given by the identity.
\end{proof}

Our strategy is then to construct a path object $\cp\AA$ for a double category $\AA$ together with double functors $W$ and $P$ factorizing the diagonal morphism $\AA\to \AA\times \AA$, such that their images under $(\bfH,\cv)$ give a weak equivalence and a fibration in $\TwoCat\times \TwoCat$ respectively. 

\begin{defn} \label{def:pathobject}
Let $\AA$ be a double category. We define a \textbf{path object} for $\AA$ as the double category $\cp\AA\coloneqq [\bbH\Eadj,\AA]_\ps$, where the $2$-category $\Eadj$ is the free-living adjoint equivalence. It comes with a factorization of the diagonal double functor 
\[ \AA\xrightarrow{W} \cp\AA\xrightarrow{P} \AA\times\AA,  \] 
 where $W$ is the double functor $\AA\cong [\mathbbm 1,\AA]_\ps\to [\bbH \Eadj,\AA]_\ps=\cp\AA$ induced by the unique map $\bbH \Eadj\to \mathbbm 1$ and $P$ is the double functor $\cp\AA=[\bbH \Eadj,\AA]_\ps\to [\mathbbm 1 \sqcup \mathbbm 1,\AA]_\ps\cong \AA\times \AA$ induced by the inclusion $\mathbbm 1\sqcup \mathbbm 1\to \bbH\Eadj$ at the two endpoints. Note that, since the composite $\mathbbm 1\sqcup \mathbbm 1\to \bbH\Eadj\to \mathbbm 1$ is the unique map, the composite $PW$ is the diagonal double functor $\AA\to\AA\times\AA$.
\end{defn}

\begin{prop} \label{prop:pathobj}
For every double category $\AA$, the path object of \cref{def:pathobject} 
\[ \AA\xrightarrow{W} \cp\AA\xrightarrow{P} \AA\times\AA,  \] 
is such that $(\bfH,\cv)W$ is a weak equivalence and $(\bfH,\cv)P$ is a fibration in $\TwoCat\times \TwoCat$.
\end{prop}

\begin{proof}
We first prove that $\bfH W$ and $\cv W$ are biequivalences in $\TwoCat$. By \cref{lem:HVpreservehom}, we have commutative squares
\begin{tz}
\node (A) at (0,0) {$\bfH[\mathbbm 1,\AA]_\ps$};
\node (B) at (0,-1.5) {$\Psd[\mathbbm 1,\bfH\AA]$};
\node (C) at (3,0) {$\bfH[\bbH\Eadj,\AA]_\ps$};
\node (D) at (3,-1.5) {$\Psd[\Eadj,\bfH\AA]$};

\draw[->] (A) to node[above,scale=0.8] {$\bfH W$} (C);
\draw[->] (A) to node[left,scale=0.8] {$\cong$} (B);
\draw[->] (C) to node[right,scale=0.8] {$\cong$} (D);
\draw[->] (B) to node[below,scale=0.8] {$(\bfH W)^\sharp$} (D);

\node (A) at (6,0) {$\cv[\mathbbm 1,\AA]_\ps$};
\node (B) at (6,-1.5) {$\Psd[\mathbbm 1,\cv\AA]$};
\node (C) at (9,0) {$\cv[\bbH\Eadj,\AA]_\ps$};
\node (D) at (9,-1.5) {$\Psd[\Eadj,\cv\AA]$};

\draw[->] (A) to node[above,scale=0.8] {$\cv W$} (C);
\draw[->] (A) to node[left,scale=0.8] {$\cong$} (B);
\draw[->] (C) to node[right,scale=0.8] {$\cong$} (D);
\draw[->] (B) to node[below,scale=0.8] {$(\cv W)^\sharp$} (D);
\end{tz}
where the $2$-functors $(\bfH W)^\sharp$ and $(\cv W)^\sharp$ are induced by the unique map $\Eadj\to \mathbbm 1$. As the inclusion $\mathbbm 1\to \Eadj$ is a trivial cofibration in $\TwoCat$ and $\bfH\AA$ and $\cv\AA$ are fibrant $2$-categories, by monoidality of the Lack model structure, we get that the induced $2$-functors \[ R\colon \Psd[\Eadj,\bfH\AA] \to \Psd[\mathbbm 1,\bfH\AA] \ \ \text{and} \ \ S\colon \Psd[\Eadj,\cv\AA] \to \Psd[\mathbbm 1,\cv\AA] \]
 are trivial fibrations in $\TwoCat$. As $R(\bfH W)^\sharp$ and $S(\cv W)^\sharp$ compose to the identity, by $2$-out-of-$3$, we get that $(\bfH W)^\sharp$ and $(\cv W)^\sharp$ are biequivalences. Again, by $2$-out-of-$3$ applied to the commutative squares above, we conclude that $\bfH W$ and $\cv W$ are biequivalences. 

Similarly, one can show that $\bfH P$ and $\cv P$ are Lack fibrations in $\TwoCat$, since the $2$-functor $\mathbbm 1\sqcup \mathbbm 1\to \Eadj$ is a cofibration in $\TwoCat$. Therefore, the induced $2$-functors
\[ \Psd[\mathbbm 1\sqcup \mathbbm 1,\bfH\AA] \to \Psd[\Eadj,\bfH\AA] \ \ \text{and} \ \ \Psd[\mathbbm 1\sqcup \mathbbm 1,\cv\AA] \to \Psd[\Eadj,\cv\AA] \]
are fibrations in $\TwoCat$, by monoidality of the Lack model structure.
\end{proof}

We are finally ready to prove the existence of the right-induced model structure on $\DblCat$ along the adjunction $\bbH\sqcup \LV\dashv (\bfH,\cv)$.

\begin{thm}\label{thm:modelstructonDblCat}
Consider the adjunction 
\begin{tz}
\node (A) at (0,0) {$\TwoCat\times \TwoCat$};
\node (B) at (3.25,0) {$\DblCat$};
\node at ($(B.east)-(0,4pt)$) {,};
\draw[->] ($(A.east)+(0,.25cm)$) to [bend left] node[above,scale=0.8]{$\bbH\sqcup \LV$} ($(B.west)+(0,.25cm)$);
\draw[->] ($(B.west)+(0,-.25cm)$) to [bend left] node[below,scale=0.8]{$(\bfH,\cv)$} ($(A.east)-(0,.25cm)$);
\node[scale=0.8] at ($(A.east)!0.5!(B.west)$) {$\bot$};
\end{tz}
where each copy of $\TwoCat$ is endowed with the Lack model structure. Then the right-induced model structure on $\DblCat$ exists. In particular, a double functor is a weak equivalence (resp.~fibration) in this model structure if and only if it is a double biequivalence (resp.~double fibration).
\end{thm}

\begin{proof}
We first describe the weak equivalences and fibrations in the right-induced model structure on $\DblCat$. These are given by the double functors $F$ such that $(\bfH,\cv)F$ is a weak equivalence (resp.~fibration) in $\TwoCat\times \TwoCat$, or equivalently, such that both $\bfH F$ and $\cv F$ are biequivalences (resp.~Lack fibrations) in $\TwoCat$.  Then it follows from \cref{charequiv,charfib} that the weak equivalences (resp.~fibrations) in $\DblCat$ are precisely the double biequivalences (resp.~double fibrations). 

We now prove the existence of the model structure. For this purpose, we want to apply \Cref{path} to our setting. First note that $\TwoCat$ and $\DblCat$ are locally presentable, and that the Lack model structure on $\TwoCat$ is cofibrantly generated. In particular, this implies that the product $\TwoCat\times \TwoCat$ endowed with two copies of the Lack model structure is combinatorial, hence accessible. Moreover, every pair of $2$-categories is fibrant in~${\TwoCat\times \TwoCat}$, since every object is fibrant in the Lack model structure. Finally, for every double category~$\AA$, \cref{prop:pathobj} gives a factorization 
\[ \AA\xrightarrow{W} \cp\AA\xrightarrow{P} \AA\times \AA  \]
such that $W$ is a double biequivalence and $P$ is a double fibration. By \cref{path}, this proves that the right-induced model structure along $(\bfH,\cv)$ on $\DblCat$ exists.
\end{proof}

\begin{rmk}
 Note that every double category is fibrant in this model structure. Indeed, this follows directly from the fact that it is right-induced from a model structure in which every object is fibrant.
\end{rmk}

\section{Generating (trivial) cofibrations and cofibrant objects}\label{section:generatingcofibs}

In this section, we take a closer look at the (trivial) cofibrations and cofibrant objects in our model structure on $\DblCat$. In \cref{subsec:gensets}, we show that this model structure is cofibrantly generated. For this, we give sets of generating cofibrations and generating trivial cofibrations, arising from corresponding sets in the Lack model structure on $\TwoCat$. Moreover, a careful study of the lifting properties of (trivial) cofibrations in $\DblCat$ reveals the existence of smaller and more descriptive sets of generating (trivial) cofibrations.

In \cref{subsec:cofibr}, we turn our attention to the cofibrant double categories. In the Lack model structure on $\TwoCat$, the cofibrant objects are precisely the $2$-categories whose underlying categories are free, by \cite[Theorem 4.8]{Lack2Cat}. We get a similar characterization of the cofibrant objects in our model structure. Since double trivial fibrations are full on horizontal morphisms and fully faithful on squares, the underlying horizontal category of a cofibrant double category is also free. However, the underlying vertical category is not only required to be free, but, in addition, it cannot contain any composition of morphisms. This is intuitively coming from the fact that double trivial fibrations are only surjective on vertical morphisms instead of full. To prove this result, we first characterize cofibrations in $\DblCat$ through their underlying horizontal and vertical ($1$-)functors.

\subsection{Generating sets of (trivial) cofibrations} \label{subsec:gensets}

Recall from \cref{thm:LackMS} that the Lack model structure on $\TwoCat$ is cofibrantly generated. As our model structure on $\DblCat$ is right-induced from two copies of $\TwoCat$ along the adjunction $\bbH\sqcup \LV \dashv (\bfH,\cv)$, it is also cofibrantly generated with sets of generating (trivial) cofibrations induced by the left adjoint $\bbH\sqcup \LV$, as stated in the next result. 

\begin{prop}\label{lem:gencofIJ}
Let $\mathcal I_2$ and $\mathcal J_2$ denote sets of generating cofibrations and generating trivial cofibrations, respectively, for the Lack model structure on $\TwoCat$. Then, the sets of morphisms in $\DblCat$
\[ \mathcal I=\{ \bbH i,\; \bbH i\times \vtwo\mid i\in \mathcal I_2 \}, \ \ \text{and} \ \  \mathcal J=\{ \bbH j,\; \bbH j\times \vtwo\mid j\in \mathcal J_2\} \]
give sets of generating cofibrations and generating trivial cofibrations, respectively, for the model structure on $\DblCat$ of \cref{thm:modelstructonDblCat}.
\end{prop}

\begin{proof}
Since the model structure on $\DblCat$ is right-induced from two copies of the Lack model structure on $\TwoCat$ along the adjunction $\bbH\sqcup \LV\dashv (\bfH,\cv)$, sets of generating cofibrations and of generating trivial cofibrations can be given by the images under the left adjoint $\bbH\sqcup \LV$ of the fixed sets of generating cofibrations and generating trivial cofibrations in~$\TwoCat\times \TwoCat$.

Let $i$ and $i'$ be generating cofibrations in $\mathcal I_2$ in $\TwoCat$. Then $\bbH i$ and $\LV i=\bbH i\times \vtwo$ are cofibrations in $\DblCat$. To see this apply $\bbH\sqcup \LV$ to the cofibrations $(i,\id_\emptyset)$ and $(\id_\emptyset,i)$, respectively. Similarly, $\bbH i'$ and $\LV i'=\bbH i'\times \vtwo$ are cofibrations in $\DblCat$. Since coproducts of cofibrations are cofibrations, then $(\bbH\sqcup \LV) (i,i')=\bbH i\sqcup \LV i'$ can be obtained from $\bbH i$ and $\LV i'=\bbH i'\times \vtwo$. This shows that $\mathcal I$ is a set of generating cofibrations of $\DblCat$. 

Similarly, we can show that $\mathcal J$ is a set of generating trivial cofibrations of $\DblCat$.
\end{proof}

However, we can find sets of generating (trivial) cofibrations, which are both smaller and more descriptive than the ones given above, by using the characterization of fibrations and trivial fibrations in our model structure given in \cref{charfib,chartrivfib}. 

\begin{notation} 
Let $\mathbb S$ be the double category free on a square, $\delta \mathbb S$ be its boundary, and $\mathbb S_2$ be the double category free on two squares with the same boundary.
\begin{tz}
    \node at (-.8,-.75) {$\mathbb S=$};
    \node (A) at (0,0) {$0$};
    \node (B) at (1.5,0) {$1$};
    \node (A') at (0,-1.5) {$0'$};
    \node (B') at (1.5,-1.5) {$1'$};
    \node at (1.9,-.75) {;};
    \draw[->] (A) to (B);
    \draw[->] (A') to (B');
    \draw[->] (A) to (A');
    \draw[->] (B) to (B');
    \node[scale=0.8] at (.75,-.75) {$\alpha$};
    \node at (0,-.75) {$\bullet$};
    \node at (1.5,-.75) {$\bullet$};
    
    \node at (3,-.75) {$\delta \mathbb S=$};
    \node (A) at (3.8,0) {$0$};
    \node (B) at (5.3,0) {$1$};
    \node (A') at (3.8,-1.5) {$0'$};
    \node (B') at (5.3,-1.5) {$1'$};
    \node at (5.7,-.75) {;};
    \draw[->] (A) to (B);
    \draw[->] (A') to (B');
    \draw[->] (A) to (A');
    \draw[->] (B) to (B');
    \node at (3.8,-.75) {$\bullet$};
    \node at (5.3,-.75) {$\bullet$};
    
    \node at (6.8,-.75) {$\mathbb S_2=$};
    \node (A) at (7.6,0) {$0$};
    \node (B) at (9.1,0) {$1$};
    \node (A') at (7.6,-1.5) {$0'$};
    \node (B') at (9.1,-1.5) {$1'$};
    \draw[->] (A) to (B);
    \draw[->] (A') to (B');
    \draw[->] (A) to (A');
    \draw[->] (B) to (B');
    \node[scale=0.8] at (8.1,-.75) {$\alpha_0$};
    \node[scale=0.8] at (8.6,-.75) {$\alpha_1$};
    \node at (7.6,-.75) {$\bullet$};
    \node at (9.1,-.75) {$\bullet$};
\end{tz}
We fix notation for the following double functors, which form a set of generating cofibrations for our model structure on $\DblCat$:
\begin{itemize}
    \item the unique map $I_1\colon \emptyset\to \mathbbm{1}$,
    \item the inclusion $I_2\colon \mathbbm{1}\sqcup \mathbbm{1}\to \bbH\mathbbm{2}$, 
    \item the unique map $I_3\colon \emptyset\to \vtwo$, 
    \item the inclusion $I_4\colon \delta \mathbb S\to \mathbb S$, and
    \item the double functor $I_5\colon \mathbb S_2\to \mathbb S$ sending both squares in $\mathbb S_2$ to the non-trivial square of~$\mathbb S$.
\end{itemize}
 We also fix notation for the following double functors, which form a set of generating trivial cofibrations for our model structure on $\DblCat$: 
\begin{itemize}
    \item the inclusion $J_1 \colon \mathbbm{1}\to \bbH \Eadj $, where the $2$-category $\Eadj$ is the free-living adjoint equivalence,
    \item the inclusion $J_2 \colon \bbH\mathbbm{2}\to \bbH C_\mathrm{inv} $, where the $2$-category $C_\mathrm{inv}$ is the free-living invertible $2$-cell, and 
    \item the inclusion $J_3 \colon \vtwo\to \bbH \Eadj \times \vtwo$; note that the double category $\bbH \Eadj \times \vtwo$ is the free-living weakly horizontally invertible square (with horizontal adjoint equivalence data).
\end{itemize}
\end{notation}

\begin{prop}\label{prop:alternate_gen_cofibs}
In the model structure on $\DblCat$ of \cref{thm:modelstructonDblCat}, a set of generating cofibrations is given by 
\[ \mathcal I'=\{ I_1\colon \emptyset\to \mathbbm{1},\;I_2\colon \mathbbm{1}\sqcup \mathbbm{1}\to \bbH\mathbbm{2},\;I_3\colon \emptyset\to \vtwo,\;I_4\colon \delta \mathbb S\to \mathbb S,\;I_5\colon \mathbb S_2\to \mathbb S \} \]
and a set of generating trivial cofibrations is given by 
\[ \mathcal J'=\{J_1\colon \mathbbm{1}\to \bbH \Eadj,\;J_2\colon \bbH\mathbbm{2}\to \bbH C_\mathrm{inv},\;J_3\colon \vtwo\to \bbH \Eadj\times \vtwo \}. \]
\end{prop}

\begin{proof}
It is a routine exercise to check that a double functor is a double trivial fibration as defined in \cref{doubletrivfib} if and only if it has the right-lifting property with respect to the cofibrations in $\mathcal I'$, and that a double functor is a double fibration as defined in \cref{doublefib} if and only if it has the right-lifting property with respect to the trivial cofibrations of $\mathcal J'$. This shows that $\mathcal I'$ and $\mathcal J'$ are sets of generating cofibrations and generating trivial cofibration for~$\DblCat$, respectively.
\end{proof}

\subsection{Cofibrations and cofibrant double categories} \label{subsec:cofibr}

Our next goal is to provide a characterization of the cofibrations in $\DblCat$. 
In \cite[Lemma 4.1]{Lack2Cat}, Lack shows that a $2$-functor is a cofibration in $\TwoCat$ if and only if its underlying functor has the left lifting property with respect to all surjective on objects and full functors. A similar result applies to our model structure; indeed, we show that a double functor is a cofibration in $\DblCat$ if and only if its underlying horizontal and vertical functors satisfy respective lifting properties.

\begin{notation} \label{Notation:underlying}
We write $U\colon \TwoCat\to \Cat$ for the functor that sends a $2$-category to its underlying category.
\end{notation}

\begin{rmk}
The functor $U\bfH\colon \DblCat\to \Cat$, which sends a double category to its underlying category of objects and horizontal morphisms, has a right adjoint. It is given by the functor $R_h\colon \Cat\to \DblCat$ that sends a category $\cc$ to the double category with the same objects as $\cc$, horizontal morphisms given by the morphisms of $\cc$, a unique vertical morphism between every pair of objects, and a unique square $\sq{!}{f}{g}{!}{!}$ for every pair of morphisms $f,g$ in $\cc$.  
\end{rmk}

\begin{rmk}
The functor $U\bfV\colon \DblCat\to \Cat$, which sends a double category to its underlying category of objects and vertical morphisms, has a right adjoint. It is given by the functor $R_v\colon \Cat\to \DblCat$ that sends a category $\cc$ to the double category with the same objects as~$\cc$, a unique horizontal morphism between every pair of objects, vertical morphisms given by the morphisms of $\cc$, and a unique square $\sq{!}{!}{!}{u}{v}$ for every pair of morphisms $u,v$ in $\cc$.
\end{rmk}

\begin{prop} \label{char:cof}
A double functor $F\colon \AA\to \BB$ is a cofibration in $\DblCat$ if and only if 
\begin{rome}
 \item the underlying horizontal functor $U\bfH F\colon U\bfH \AA\to U\bfH\BB$ has the left lifting property with respect to all surjective on objects and full functors, and
 \item the underlying vertical functor $U\bfV F\colon U\bfV \AA\to U\bfV\BB$ has the left lifting property with respect to all surjective on objects and surjective on morphisms functors.
\end{rome}
\end{prop}

\begin{proof}
Suppose first that $F\colon \AA\to \BB$ is a cofibration in $\DblCat$, i.e., it has the left lifting property with respect to all double trivial fibrations. In order to show (i), let $P\colon \mathcal X\to \mathcal Y$ be a surjective on objects and full functor. By the adjunction $U\bfH \dashv R_h$, saying that $U\bfH F$ has the left lifting property with respect to $P$ is equivalent to saying that $F$ has the left lifting property with respect to $R_hP$. We now prove this latter statement. 

Note that the double functor $R_hP\colon R_h\mathcal X\to R_h\mathcal Y$ is surjective on objects and full on horizontal morphisms, since $P$ is so. Moreover, by construction of $R_h$, there is exactly one vertical morphism and one square for each boundary in both its source and target; therefore $R_hP$ is surjective on vertical morphisms and fully faithful on squares. Hence~$R_h P$ is a double trivial fibration, and $F$ has the left lifting property with respect to $R_h P$ since it is a cofibration in $\DblCat$. 

Similarly, one can show that (ii) holds, by considering the adjunction $U\bfV \dashv R_v$ and replacing fullness by surjectivity on morphisms. 

Now suppose that $F\colon \AA\to \BB$ satisfies (i) and (ii). Let $P\colon \mathbb X\to \mathbb Y$ be a double trivial fibration and consider a commutative square as below left. We want to find a lift $L\colon \BB\to \mathbb X$ in this square as depicted below. Using~(ii), since $U\bfV P$ is surjective on objects and surjective on morphisms, we have a lift~$L_v$ in the below middle diagram. Now, using (i), since $U\bfH P$ is surjective on objects and full, we can choose a lift $L_h$ in the below right diagram such that $L_h$ coincides with $L_v$ on objects. This comes from the fact that, by fullness of~$U\bfH P$, we can first choose an assignment on objects and then choose a compatible assignment on morphisms.
\begin{tz}
\node[](1) {$\AA$}; 
\node[right of=1](2) {$\mathbb X$}; 
\node[below of=1](3) {$\BB$};
\node[below of=2](4) {$\mathbb Y$}; 
\draw[->] (1) to node[above,la]{$G$} (2);
\draw[->] (1) to node[left,la]{$F$} (3); 
\draw[->] (2) to node[right,la]{$P$} (4);
\draw[->] (3) to node[below,la]{$Q$} (4);
\draw[->,dashed] (3) to node[pos=0.4,above,la,xshift=-5pt]{$L$} (2);

\node[right of=2,xshift=1cm](1) {$U\bfV \AA$}; 
\node[right of=1,xshift=.5cm](2) {$U\bfV \mathbb X$}; 
\node[below of=1](3) {$U\bfV \BB$};
\node[below of=2](4) {$U\bfV \mathbb Y$}; 
\draw[->] (1) to node[above,la]{$U\bfV G$} (2);
\draw[->] (1) to node[left,la]{$U\bfV F$} (3); 
\draw[->] (2) to node[right,la]{$U\bfV P$} (4);
\draw[->] (3) to node[below,la]{$U\bfV Q$} (4);
\draw[->,dashed] (3) to node[pos=0.4,above,la,xshift=-5pt]{$L_v$} (2);

\node[right of=2,xshift=1cm](1) {$U\bfH \AA$}; 
\node[right of=1,xshift=.5cm](2) {$U\bfH \mathbb X$}; 
\node[below of=1](3) {$U\bfH \BB$};
\node[below of=2](4) {$U\bfH \mathbb Y$}; 
\draw[->] (1) to node[above,la]{$U\bfH G$} (2);
\draw[->] (1) to node[left,la]{$U\bfH F$} (3); 
\draw[->] (2) to node[right,la]{$U\bfH P$} (4);
\draw[->] (3) to node[below,la]{$U\bfH Q$} (4);
\draw[->,dashed] (3) to node[pos=0.4,above,la,xshift=-5pt]{$L_h$} (2);
\end{tz}
Then, since $P\colon \mathbb X\to \mathbb Y$ is fully faithful on squares, the assignment on objects, horizontal morphisms, and vertical morphisms given by $L_h$ and $L_v$ uniquely extend to a double functor $L\colon \BB\to \mathbb Y$, which gives the desired lift.
\end{proof}

\begin{rmk}
Note that the functor $\bfH\colon \DblCat\to \TwoCat$ preserves cofibrations, since a $2$-functor is a cofibration in $\TwoCat$ if and only if its underlying functor has the left lifting property with respect to all surjective on objects and full functors by \cite[Lemma 4.1]{Lack2Cat}.
\end{rmk}

To understand what it means for a double functor to satisfy (i) and (ii) of \cref{char:cof}, we state a characterization of the functors in $\Cat$ which have the left lifting property with respect to all surjective on objects and full (resp.~surjective on morphisms) functors.

\begin{lemma} \label{surjfull}
A functor $F\colon \A\to \B$ has the left lifting property with respect to surjective on objects and full (resp.\ surjective) on morphisms functors if and only if
\begin{rome}
\item the functor $F$ is injective on objects and faithful, and
\item there are functors $I\colon \B\to \cc$ and $R\colon \cc\to \B$ such that $RI=\id_{\B}$, where the category $\cc$ is obtained from the image of $F$ by freely adjoining objects and then freely adjoining morphisms between specified objects (resp.\ by freely adjoining objects and morphisms). 
\end{rome}
Moreover, a functor $\emptyset\to \A$ has the left lifting property with respect to surjective on objects and full (resp.\ surjective) on morphisms functors if and only if the category $\A$ is free (resp.\ a disjoint union of copies of $\mathbbm{1}$ and $\mathbbm{2}$). 
\end{lemma}

\begin{proof}
The statement about ``full on morphisms'' follows directly from \cite[Corollary 4.12]{Lack2Cat}. For the ``surjective on morphisms'' case, the proof is analogous, replacing $\mathbbm 1\sqcup \mathbbm 1\to \mathbbm 2$ by $\emptyset\to \mathbbm 2$. 

The second statement about $\emptyset\to \A$ follows from the fact that a retract of a free category is itself free, and similarly for disjoint unions of copies of $\mathbbm{1}$ and $\mathbbm{2}$.
\end{proof}

\begin{rmk} \label{injectivityofcof}
From \cref{char:cof,surjfull}, it is straightforward to see that a cofibration in $\DblCat$ is in particular injective on objects, and faithful on horizontal morphisms and vertical morphisms.
\end{rmk}

Finally, we use the above results to obtain a characterization of the cofibrant double categories in terms of their underlying horizontal and vertical categories. 

\begin{prop} \label{char:cofibrant}
A double category $\AA$ is cofibrant if and only if its underlying horizontal category $U\bfH \AA$ is free and its underlying vertical category $U\bfV \AA$ is a disjoint union of copies of $\mathbbm{1}$ and $\mathbbm{2}$.
\end{prop}

\begin{proof}
This follows directly from \cref{char:cof,surjfull} applied to the unique double functor $\emptyset\to \AA$.
\end{proof}

\section{Fibrations, weak equivalences, and Whitehead theorems}\label{section:weakequivsandfibrations}

The purpose of this section is to describe the weak equivalences and fibrations of our model structure. \cref{subsec:charfibeq} provides proofs of \cref{charequiv,charfib}, which claim that the weak equivalences and fibrations of the right-induced model structure on $\DblCat$ of \cref{thm:modelstructonDblCat} are precisely the double biequivalences of \cref{doublequiv} and the double fibrations of \cref{doublefib}. 

In \cref{subsec:whitehead} we turn our attention to another characterization of the weak equivalences, known as the Whitehead theorem. Recall that, in the $2$-categorical case, a $2$-functor is a biequivalence if and only if it has a pseudo inverse up to pseudo natural equivalence (see \cite[Theorem~7.4.1]{JohYau}). A similar statement does not hold in general for double biequivalences. However, when assuming that the target double category of a double functor has no composites of vertical morphisms (a condition that the cofibrant double categories in our model structure satisfy), we can show that a double functor is a double biequivalence if and only if it has a pseudo inverse up to horizontal pseudo natural equivalence. In particular, we retrieve the usual Whitehead theorem for model categories applied to our setting, and also the Whitehead theorem for $2$-categories stated above. Another version of the Whitehead theorem for double biequivalences is given in \cite[Theorem 8.1]{WHI}, which in turn holds for the fibrant objects of the model structure on $\DblCat$ defined therein.

\subsection{Characterizations of weak equivalences and fibrations} \label{subsec:charfibeq}

We first prove \cref{charequiv}, dealing with weak equivalences. In order to characterize the double functors $F$ such that $(\bfH,\cv)F$ is a weak equivalence in $\TwoCat\times\TwoCat$, we express what it means for~$\bfH F$ and $\cv F$ to be biequivalences in $\TwoCat$; this is done by translating (b1-3) of \cref{biequiv} for these $2$-functors.

\begin{rmk}\label{Hequiv}
Let $F\colon \AA\to \BB$ be a double functor. Then $\bfH F\colon \bfH \AA\to \bfH \BB$ is a biequivalence in $\TwoCat$ if and only if $F$ satisfies (db1-2) of \cref{doublequiv}, and the following condition:
\begin{enumerate}
    \item[(hb3)] for every pair of horizontal morphisms $a,c\colon A\to C$ in $\AA$ and every square in $\BB$ of the form
\begin{tz}
\node (A) at (0,0) {$FA$};
\node (B) at (1.5,0) {$FC$};
\node (A') at (0,-1.5) {$FA$};
\node (B') at (1.5,-1.5) {$FC$};
\node at ($(B'.east)-(0,4pt)$) {,};
\draw[->] (A) to node[above, scale=0.8] {$Fa$} (B);
\draw[->] (A') to node[below, scale=0.8] {$Fc$} (B');
\draw[d] (A) to (A');
\draw[d] (B) to (B');

\node at (0,-.75) {$\bullet$};
\node at (1.5, -.75) {$\bullet$};

\node[scale=0.8] at (.75,-.75) {$\beta$};
\end{tz}
    there is a unique square $\sq{\alpha}{a}{c}{e_A}{e_C}$ in $\AA$ such that $F\alpha=\beta$.
\end{enumerate}
\end{rmk}

\begin{rmk}\label{Vequiv}
Let $F\colon \AA\to \BB$ be a double functor. Then $\cv F\colon \cv \AA\to \cv \BB$ is a biequivalence in~$\TwoCat$ if and only if $F$ satisfies (db3) of \cref{doublequiv}, and the following conditions: 
\begin{enumerate}
    \item[(vb2)] for every pair of vertical morphisms $u\colon A\arrowdot A'$ and $u'\colon C\arrowdot C'$ in $\AA$ and every square $\sq{\beta}{b}{d}{Fu}{Fu'}$ in $\BB$, there is a square $\sq{\alpha}{a}{c}{u}{u'}$ in $\AA$ and two vertically invertible squares in $\BB$ such that the following pasting equality holds,
\begin{tz}
\node (A) at (0,0){$FA$};
\node (B) at (1.5,0){$FC$};
\node (A') at (0,-1.5){$FA$};
\node (B') at (1.5,-1.5){$FC$};
\node (A'') at (0,-3){$FA'$};
\node (B'') at (1.5,-3){$FC'$};
\draw[d] (A) to (A');
\draw[d] (B) to (B');
\draw[->] (A) to node[scale=0.8,above]{$b$} (B);
\draw[->] (A') to node[scale=0.8,above]{$Fa$} (B');
\draw[->] (A') to node[scale=0.8,left]{$Fu\;$} (A'');
\draw[->] (B') to node[scale=0.8,right]{$\;Fu'$}(B'');
\draw[->] (A'') to node[scale=0.8,below]{$Fc$} (B'');

\node at (0,-.75) {$\bullet$};
\node at (1.5,-.75) {$\bullet$};
\node at (0,-2.25) {$\bullet$};
\node at (1.5,-2.25) {$\bullet$};

\node[scale=0.8] at (.75,-.75) {$\vcong$};
\node[scale=0.8] at (.75,-2.25) {$F\alpha$};

\node at (2.5,-1.5) {$=$};

\node (A) at (3.5,0){$FA$};
\node (B) at (5,0){$FC$};
\node (A') at (3.5,-1.5){$FA'$};
\node (B') at (5,-1.5){$FC'$};
\node (A'') at (3.5,-3){$FA'$};
\node (B'') at (5,-3){$FC'$};
\draw[d] (A') to (A'');
\draw[d] (B') to (B'');
\draw[->] (A) to node[scale=0.8,above]{$b$} (B);
\draw[->] (A') to node[scale=0.8,below]{$d$} (B');
\draw[->] (A) to node[scale=0.8,left]{$Fu\;$} (A');
\draw[->] (B) to node[scale=0.8,right]{$\;Fu'$}(B');
\draw[->] (A'') to node[scale=0.8,below]{$Fc$} (B'');

\node at (3.5,-.75) {$\bullet$};
\node at (5,-.75) {$\bullet$};
\node at (3.5,-2.25) {$\bullet$};
\node at (5,-2.25) {$\bullet$};

\node[scale=0.8] at (4.25,-.75) {$\beta$};
\node[scale=0.8] at (4.25,-2.25) {$\vcong$};
\end{tz}
    \item[(vb3)] for every tuple of squares $\sq{\alpha}{a}{c}{u}{u'}$ and $\sq{\alpha'}{a'}{c'}{u}{u'}$ in $\AA$, and $\tau_0$ and $\tau_1$ in $\BB$ as in the pasting equality below left, there are unique squares $\sq{\sigma_0}{a}{a'}{e_A}{e_C}$ and $\sq{\sigma_1}{c}{c'}{e_{A'}}{e_{C'}}$ in $\AA$ satisfying the pasting equality below right, and with the property that $F\sigma_0=\tau_0$ and $F\sigma_1=\tau_1$.
\end{enumerate}
\begin{tz}
\node (A) at (0,0){$FA$};
\node (B) at (1.5,0){$FC$};
\node (A') at (0,-1.5){$FA$};
\node (B') at (1.5,-1.5){$FC$};
\node (A'') at (0,-3){$FA'$};
\node (B'') at (1.5,-3){$FC'$};
\draw[d] (A) to (A');
\draw[d] (B) to (B');
\draw[->] (A) to node[scale=0.8,above]{$Fa$} (B);
\draw[->] (A') to node[scale=0.8,above]{$Fa'$} (B');
\draw[->] (A') to node[scale=0.8,left]{$Fu\;$} (A'');
\draw[->] (B') to node[scale=0.8,right]{$\;Fu'$}(B'');
\draw[->] (A'') to node[scale=0.8,below]{$Fc'$} (B'');

\node at (0,-.75) {$\bullet$};
\node at (1.5,-.75) {$\bullet$};
\node at (0,-2.25) {$\bullet$};
\node at (1.5,-2.25) {$\bullet$};

\node[scale=0.8] at (.75,-.75) {$\tau_0$};
\node[scale=0.8] at (.75,-2.25) {$F\alpha'$};

\node at (2.5,-1.5) {$=$};

\node (A) at (3.5,0){$FA$};
\node (B) at (5,0){$FC$};
\node (A') at (3.5,-1.5){$FA'$};
\node (B') at (5,-1.5){$FC'$};
\node (A'') at (3.5,-3){$FA'$};
\node (B'') at (5,-3){$FC'$};
\draw[d] (A') to (A'');
\draw[d] (B') to (B'');
\draw[->] (A) to node[scale=0.8,above]{$Fa$} (B);
\draw[->] (A') to node[scale=0.8,below]{$Fc$} (B');
\draw[->] (A) to node[scale=0.8,left]{$Fu\;$} (A');
\draw[->] (B) to node[scale=0.8,right]{$\;Fu'$}(B');
\draw[->] (A'') to node[scale=0.8,below]{$Fc'$} (B'');

\node at (3.5,-.75) {$\bullet$};
\node at (5,-.75) {$\bullet$};
\node at (3.5,-2.25) {$\bullet$};
\node at (5,-2.25) {$\bullet$};

\node[scale=0.8] at (4.25,-.75) {$F\alpha$};
\node[scale=0.8] at (4.25,-2.25) {$\tau_1$};

\node (A) at (8,0){$A$};
\node (B) at (9.5,0){$C$};
\node (A') at (8,-1.5){$A$};
\node (B') at (9.5,-1.5){$C$};
\node (A'') at (8,-3){$A'$};
\node (B'') at (9.5,-3){$C'$};
\draw[d] (A) to (A');
\draw[d] (B) to (B');
\draw[->] (A) to node[scale=0.8,above]{$a$} (B);
\draw[->] (A') to node[scale=0.8,above]{$a'$} (B');
\draw[->] (A') to node[scale=0.8,left]{$u\;$} (A'');
\draw[->] (B') to node[scale=0.8,right]{$\;u'$}(B'');
\draw[->] (A'') to node[scale=0.8,below]{$c'$} (B'');

\node at (8,-.75) {$\bullet$};
\node at (9.5,-.75) {$\bullet$};
\node at (8,-2.25) {$\bullet$};
\node at (9.5,-2.25) {$\bullet$};

\node[scale=0.8] at (8.75,-.75) {$\sigma_0$};
\node[scale=0.8] at (8.75,-2.25) {$\alpha'$};

\node at (10.5,-1.5) {$=$};

\node (A) at (11.5,0){$A$};
\node (B) at (13,0){$C$};
\node (A') at (11.5,-1.5){$A'$};
\node (B') at (13,-1.5){$C'$};
\node (A'') at (11.5,-3){$A'$};
\node (B'') at (13,-3){$C'$};
\draw[d] (A') to (A'');
\draw[d] (B') to (B'');
\draw[->] (A) to node[scale=0.8,above]{$a$} (B);
\draw[->] (A') to node[scale=0.8,below]{$c$} (B');
\draw[->] (A) to node[scale=0.8,left]{$u\;$} (A');
\draw[->] (B) to node[scale=0.8,right]{$\;u'$}(B');
\draw[->] (A'') to node[scale=0.8,below]{$c'$} (B'');

\node at (11.5,-.75) {$\bullet$};
\node at (13,-.75) {$\bullet$};
\node at (11.5,-2.25) {$\bullet$};
\node at (13,-2.25) {$\bullet$};

\node[scale=0.8] at (12.25,-.75) {$\alpha$};
\node[scale=0.8] at (12.25,-2.25) {$\sigma_1$};

\end{tz}
\end{rmk}

The reader may have noticed that condition (db4) in \cref{doublequiv} regarding fully faithfulness on squares has not been mentioned so far. The following lemma says that the additional conditions (hb3) and (vb2-3) introduced in \cref{Hequiv,Vequiv} actually imply (db4).

\begin{lemma} \label{Fullfaithcell}
Suppose that $F\colon \AA\to \BB$ is a double functor satisfying (hb3) of \cref{Hequiv}, and~(vb2-3) of \cref{Vequiv}. Then $F$ satisfies (db4) of \cref{doublequiv}.
\end{lemma}

\begin{proof}
Suppose $\sq{\beta}{Fa}{Fc}{Fu}{Fu'}$ is a square in $\BB$ as in (db4) of \cref{doublequiv}. By (vb2) of \cref{Vequiv}, there is a square $\sq{\overline{\alpha}}{\overline{a}}{\overline{c}}{u}{u'}$ in $\AA$ and two vertically invertible squares $\psi_0$, $\psi_1$ in $\BB$ such that the following pasting equality holds.
\begin{tz}
\node (A) at (0,0){$FA$};
\node (B) at (1.5,0){$FC$};
\node (A') at (0,-1.5){$FA$};
\node (B') at (1.5,-1.5){$FC$};
\node (A'') at (0,-3){$FA'$};
\node (B'') at (1.5,-3){$FC'$};
\draw[d] (A) to (A');
\draw[d] (B) to (B');
\draw[->] (A) to node[scale=0.8,above]{$Fa$} (B);
\draw[->] (A') to node[scale=0.8,above]{$F\overline{a}$} (B');
\draw[->] (A') to node[scale=0.8,left]{$Fu\;$} (A'');
\draw[->] (B') to node[scale=0.8,right]{$\;Fu'$}(B'');
\draw[->] (A'') to node[scale=0.8,below]{$F\overline{c}$} (B'');

\node at (0,-.75) {$\bullet$};
\node at (1.5,-.75) {$\bullet$};
\node at (0,-2.25) {$\bullet$};
\node at (1.5,-2.25) {$\bullet$};

\node[scale=0.8] at (.9,-.75) {$\vcong$};
\node[scale=0.8] at (.6,-.75) {$\psi_0$};
\node[scale=0.8] at (.75,-2.25) {$F\overline\alpha$};

\node at (2.5,-1.5) {$=$};

\node (A) at (3.5,0){$FA$};
\node (B) at (5,0){$FC$};
\node (A') at (3.5,-1.5){$FA'$};
\node (B') at (5,-1.5){$FC'$};
\node (A'') at (3.5,-3){$FA'$};
\node (B'') at (5,-3){$FC'$};
\draw[d] (A') to (A'');
\draw[d] (B') to (B'');
\draw[->] (A) to node[scale=0.8,above]{$Fa$} (B);
\draw[->] (A') to node[scale=0.8,below]{$Fc$} (B');
\draw[->] (A) to node[scale=0.8,left]{$Fu\;$} (A');
\draw[->] (B) to node[scale=0.8,right]{$\;Fu'$}(B');
\draw[->] (A'') to node[scale=0.8,below]{$F\overline{c}$} (B'');

\node at (3.5,-.75) {$\bullet$};
\node at (5,-.75) {$\bullet$};
\node at (3.5,-2.25) {$\bullet$};
\node at (5,-2.25) {$\bullet$};

\node[scale=0.8] at (4.25,-.75) {$\beta$};
\node[scale=0.8] at (4.4,-2.25) {$\vcong$};
\node[scale=0.8] at (4.1,-2.25) {$\psi_1$};
\end{tz}
By (hb3) of \cref{Hequiv}, there are unique squares $\sq{\varphi_0}{a}{\overline{a}}{e_A}{e_C}$ and $\sq{\varphi_1}{c}{\overline{c}}{e_{A'}}{e_{C'}}$ in $\AA$ such that $F\varphi_0=\psi_0$ and $F\varphi_1=\psi_1$. Moreover, the squares $\varphi_0$ and $\varphi_1$ are vertically invertible by the unicity condition in (hb3). Therefore, the square $\alpha$ given by the following vertical pasting
\begin{tz}
\node (A) at (-3.5,-1.5) {$A$};
\node (B) at (-2,-1.5) {$C$};
\node (A') at (-3.5,-3) {$A'$};
\node (B') at (-2,-3) {$C'$};
\draw[->] (A) to node[above, scale=0.8] {$a$} (B);
\draw[->] (A') to node[below, scale=0.8] {$c$} (B');
\draw[->] (A) to node[left,scale=0.8] {$u\;$} (A');
\draw[->] (B) to node[right,scale=0.8] {$\;u'$} (B');

\node at (-3.5,-2.25) {$\bullet$};
\node at (-2, -2.25) {$\bullet$};

\node[scale=0.8] at (-2.75,-2.25) {$\alpha$};

\node at (-1,-2.25) {$=$};

\node (A) at (0,0){$A$};
\node (B) at (1.5,0){$C$};
\node (A') at (0,-1.5){$A$};
\node (B') at (1.5,-1.5){$C$};
\node (A'') at (0,-3){$A'$};
\node (B'') at (1.5,-3){$C'$};
\node (A''') at (0,-4.5){$A'$};
\node (B''') at (1.5,-4.5){$C'$};
\draw[d] (A) to (A');
\draw[d] (B) to (B');
\draw[->] (A) to node[scale=0.8,above]{$a$} (B);
\draw[->] (A') to node[scale=0.8,above]{$\overline{a}$} (B');
\draw[->] (A') to node[scale=0.8,left]{$u\;$} (A'');
\draw[->] (B') to node[scale=0.8,right]{$\;u'$}(B'');
\draw[->] (A'') to node[scale=0.8,below]{$\overline{c}$} (B'');
\draw[->] (A''') to node[scale=0.8,below]{$c$} (B''');
\draw[d] (A'') to (A''');
\draw[d] (B'') to (B''');

\node at (0,-.75) {$\bullet$};
\node at (1.5,-.75) {$\bullet$};
\node at (0,-2.25) {$\bullet$};
\node at (1.5,-2.25) {$\bullet$};
\node at (0,-3.75) {$\bullet$};
\node at (1.5,-3.75) {$\bullet$};

\node[scale=0.8] at (.9,-.75) {$\vcong$};
\node[scale=0.8] at (.6,-.75) {$\varphi_0$};
\node[scale=0.8] at (.75,-2.25) {$\overline\alpha$};
\node[scale=0.8] at (1,-3.8) {$\vcong$};
\node[scale=0.8] at (.6,-3.75) {$\varphi_1^{-1}$};
\end{tz}
is such that $F\alpha=\beta$. This settles the matter of the existence of the square $\alpha$. Now suppose there are two squares $\sq{\alpha}{a}{c}{u}{u'}$ and $\sq{\alpha'}{a}{c}{u}{u'}$ in $\AA$ such that $F\alpha=\beta=F\alpha'$. Take $\tau_0=e_{Fa}$ and $\tau_1=e_{Fc}$ in (vb3) of \cref{Vequiv}. This gives unique squares $\sigma_0$ and $\sigma_1$ in $\AA$ such that the following pasting equality holds
\begin{tz}
\node (A) at (0,0){$A$};
\node (B) at (1.5,0){$C$};
\node (A') at (0,-1.5){$A$};
\node (B') at (1.5,-1.5){$C$};
\node (A'') at (0,-3){$A'$};
\node (B'') at (1.5,-3){$C'$};
\draw[d] (A) to (A');
\draw[d] (B) to (B');
\draw[->] (A) to node[scale=0.8,above]{$a$} (B);
\draw[->] (A') to node[scale=0.8,above]{$a$} (B');
\draw[->] (A') to node[scale=0.8,left]{$u\;$} (A'');
\draw[->] (B') to node[scale=0.8,right]{$\;u'$}(B'');
\draw[->] (A'') to node[scale=0.8,below]{$c$} (B'');

\node at (0,-.75) {$\bullet$};
\node at (1.5,-.75) {$\bullet$};
\node at (0,-2.25) {$\bullet$};
\node at (1.5,-2.25) {$\bullet$};

\node[scale=0.8] at (.75,-.75) {$\sigma_0$};
\node[scale=0.8] at (.75,-2.25) {$\alpha'$};

\node at (2.5,-1.5) {$=$};

\node (A) at (3.5,0){$A$};
\node (B) at (5,0){$C$};
\node (A') at (3.5,-1.5){$A'$};
\node (B') at (5,-1.5){$C'$};
\node (A'') at (3.5,-3){$A'$};
\node (B'') at (5,-3){$C'$};
\draw[d] (A') to (A'');
\draw[d] (B') to (B'');
\draw[->] (A) to node[scale=0.8,above]{$a$} (B);
\draw[->] (A') to node[scale=0.8,below]{$c$} (B');
\draw[->] (A) to node[scale=0.8,left]{$u\;$} (A');
\draw[->] (B) to node[scale=0.8,right]{$\;u'$}(B');
\draw[->] (A'') to node[scale=0.8,below]{$c$} (B'');

\node at (3.5,-.75) {$\bullet$};
\node at (5,-.75) {$\bullet$};
\node at (3.5,-2.25) {$\bullet$};
\node at (5,-2.25) {$\bullet$};

\node[scale=0.8] at (4.25,-.75) {$\alpha$};
\node[scale=0.8] at (4.25,-2.25) {$\sigma_1$};

\end{tz}
     and $F\sigma_0=e_{Fa}$ and $F\sigma_1=e_{Fc}$. By unicity in (hb3), we must have $\sigma_0=e_a$ and $\sigma_1=e_c$. Replacing $\sigma_0$ and $\sigma_1$ by $e_a$ and $e_c$ in the pasting diagram above implies that~$\alpha=\alpha'$. This proves unicity. 
\end{proof}

We can now use the above results to prove \cref{charequiv}, giving the desired characterization of the weak equivalences in our model structure on $\DblCat$.

\begin{proof}[Proof of \cref{charequiv}]
Suppose that $F\colon \AA\to \BB$ is a double functor such that both $\bfH F$ and $\cv F$ are biequivalences in $\TwoCat$. By \cref{Hequiv,Vequiv}, we directly have that $F$ satisfies (db1-3) of \cref{doublequiv}. Moreover, by \cref{Fullfaithcell}, we also have that $F$ satisfies (db4) of \cref{doublequiv}. This shows that $F$ is a double biequivalence. 

Now suppose that $F\colon \AA\to \BB$ is a double biequivalence. We want to show that both~$\bfH F$ and $\cv F$ are biequivalences in $\TwoCat$. To show that $\bfH F$ is a biequivalence, it suffices to show that (hb3) of \cref{Hequiv} is satisfied; this follows directly from taking $u$ and $u'$ to be vertical identities in (db4) of \cref{doublequiv}. 

It remains to show that $\cv F$ is a biequivalence; we do so by proving (vb2-3) of \cref{Vequiv}. To prove (vb2), let $u\colon A\arrowdot A'$ and $u'\colon C\arrowdot C'$ be vertical morphisms in $\AA$ and $\beta$ be a square in $\BB$ of the form
\begin{tz}
\node (A) at (3.5,0){$FA$};
\node (B) at (5,0){$FC$};
\node (A') at (3.5,-1.5){$FA'$};
\node (B') at (5,-1.5){$FC'$};
\node at ($(B'.east)-(0,4pt)$) {.};
\draw[->] (A) to node[scale=0.8,above]{$b$} (B);
\draw[->] (A') to node[scale=0.8,below]{$d$} (B');
\draw[->] (A) to node[scale=0.8,left]{$Fu\;$} (A');
\draw[->] (B) to node[scale=0.8,right]{$\;Fu'$}(B');

\node at (3.5,-.75) {$\bullet$};
\node at (5,-.75) {$\bullet$};

\node[scale=0.8] at (4.25,-.75) {$\beta$};
\end{tz}
By (db2) of \cref{doublequiv}, there are horizontal morphisms $a\colon A\to C$ and $c\colon A'\to C'$ in~$\AA$ and vertically invertible squares $\sq{\varphi_0}{b}{Fa}{e_{FA}}{e_{FC}}$ and $\sq{\varphi_1}{d}{Fc}{e_{FA'}}{e_{FC'}}$ in $\BB$. By (db4) of \cref{doublequiv}, there is a unique square $\sq{\alpha}{a}{c}{u}{u'}$ in $\mathbb A$ such that 
\begin{tz}
\node (A) at (-3.5,-1.5) {$FA$};
\node (B) at (-2,-1.5) {$FC$};
\node (A') at (-3.5,-3) {$FA'$};
\node (B') at (-2,-3) {$FC'$};
\draw[->] (A) to node[above, scale=0.8] {$Fa$} (B);
\draw[->] (A') to node[below, scale=0.8] {$Fc$} (B');
\draw[->] (A) to node[left,scale=0.8] {$Fu\;$} (A');
\draw[->] (B) to node[right,scale=0.8] {$\;Fu'$} (B');

\node at (-3.5,-2.25) {$\bullet$};
\node at (-2, -2.25) {$\bullet$};

\node[scale=0.8] at (-2.75,-2.25) {$F\alpha$};

\node at (-1,-2.25) {$=$};

\node (A) at (0,0){$FA$};
\node (B) at (1.5,0){$FC$};
\node (A') at (0,-1.5){$FA$};
\node (B') at (1.5,-1.5){$FC$};
\node (A'') at (0,-3){$FA'$};
\node (B'') at (1.5,-3){$FC'$};
\node (A''') at (0,-4.5){$FA'$};
\node (B''') at (1.5,-4.5){$FC'$};
\node at ($(B'''.east)-(0,4pt)$) {,};
\draw[d] (A) to (A');
\draw[d] (B) to (B');
\draw[->] (A) to node[scale=0.8,above]{$Fa$} (B);
\draw[->] (A') to node[scale=0.8,above]{$b$} (B');
\draw[->] (A') to node[scale=0.8,left]{$Fu\;$} (A'');
\draw[->] (B') to node[scale=0.8,right]{$\;Fu'$}(B'');
\draw[->] (A'') to node[scale=0.8,below]{$d$} (B'');
\draw[->] (A''') to node[scale=0.8,below]{$Fc$} (B''');
\draw[d] (A'') to (A''');
\draw[d] (B'') to (B''');

\node at (0,-.75) {$\bullet$};
\node at (1.5,-.75) {$\bullet$};
\node at (0,-2.25) {$\bullet$};
\node at (1.5,-2.25) {$\bullet$};
\node at (0,-3.75) {$\bullet$};
\node at (1.5,-3.75) {$\bullet$};

\node[scale=0.8] at (1,-.75) {$\vcong$};
\node[scale=0.8] at (.6,-.7) {$\varphi_0^{-1}$};
\node[scale=0.8] at (.75,-2.25) {$\beta$};
\node[scale=0.8] at (.9,-3.75) {$\vcong$};
\node[scale=0.8] at (.6,-3.75) {$\varphi_1$};
\end{tz}
which gives (vb2). Finally, we prove (vb3). Suppose we have a pasting equality in~$\BB$ as below left.
\begin{tz}
\node (A) at (0,0){$FA$};
\node (B) at (1.5,0){$FC$};
\node (A') at (0,-1.5){$FA$};
\node (B') at (1.5,-1.5){$FC$};
\node (A'') at (0,-3){$FA'$};
\node (B'') at (1.5,-3){$FC'$};
\draw[d] (A) to (A');
\draw[d] (B) to (B');
\draw[->] (A) to node[scale=0.8,above]{$Fa$} (B);
\draw[->] (A') to node[scale=0.8,above]{$Fa'$} (B');
\draw[->] (A') to node[scale=0.8,left]{$Fu\;$} (A'');
\draw[->] (B') to node[scale=0.8,right]{$\;Fu'$}(B'');
\draw[->] (A'') to node[scale=0.8,below]{$Fc'$} (B'');

\node at (0,-.75) {$\bullet$};
\node at (1.5,-.75) {$\bullet$};
\node at (0,-2.25) {$\bullet$};
\node at (1.5,-2.25) {$\bullet$};

\node[scale=0.8] at (.75,-.75) {$\tau_0$};
\node[scale=0.8] at (.75,-2.25) {$F\alpha'$};

\node at (2.5,-1.5) {$=$};

\node (A) at (3.5,0){$FA$};
\node (B) at (5,0){$FC$};
\node (A') at (3.5,-1.5){$FA'$};
\node (B') at (5,-1.5){$FC'$};
\node (A'') at (3.5,-3){$FA'$};
\node (B'') at (5,-3){$FC'$};
\draw[d] (A') to (A'');
\draw[d] (B') to (B'');
\draw[->] (A) to node[scale=0.8,above]{$Fa$} (B);
\draw[->] (A') to node[scale=0.8,below]{$Fc$} (B');
\draw[->] (A) to node[scale=0.8,left]{$Fu\;$} (A');
\draw[->] (B) to node[scale=0.8,right]{$\;Fu'$}(B');
\draw[->] (A'') to node[scale=0.8,below]{$Fc'$} (B'');

\node at (3.5,-.75) {$\bullet$};
\node at (5,-.75) {$\bullet$};
\node at (3.5,-2.25) {$\bullet$};
\node at (5,-2.25) {$\bullet$};

\node[scale=0.8] at (4.25,-.75) {$F\alpha$};
\node[scale=0.8] at (4.25,-2.25) {$\tau_1$};

\node (A) at (8,0){$A$};
\node (B) at (9.5,0){$C$};
\node (A') at (8,-1.5){$A$};
\node (B') at (9.5,-1.5){$C$};
\node (A'') at (8,-3){$A'$};
\node (B'') at (9.5,-3){$C'$};
\draw[d] (A) to (A');
\draw[d] (B) to (B');
\draw[->] (A) to node[scale=0.8,above]{$a$} (B);
\draw[->] (A') to node[scale=0.8,above]{$a'$} (B');
\draw[->] (A') to node[scale=0.8,left]{$u\;$} (A'');
\draw[->] (B') to node[scale=0.8,right]{$\;u'$}(B'');
\draw[->] (A'') to node[scale=0.8,below]{$c'$} (B'');

\node at (8,-.75) {$\bullet$};
\node at (9.5,-.75) {$\bullet$};
\node at (8,-2.25) {$\bullet$};
\node at (9.5,-2.25) {$\bullet$};

\node[scale=0.8] at (8.75,-.75) {$\sigma_0$};
\node[scale=0.8] at (8.75,-2.25) {$\alpha'$};

\node at (10.5,-1.5) {$=$};

\node (A) at (11.5,0){$A$};
\node (B) at (13,0){$C$};
\node (A') at (11.5,-1.5){$A'$};
\node (B') at (13,-1.5){$C'$};
\node (A'') at (11.5,-3){$A'$};
\node (B'') at (13,-3){$C'$};
\draw[d] (A') to (A'');
\draw[d] (B') to (B'');
\draw[->] (A) to node[scale=0.8,above]{$a$} (B);
\draw[->] (A') to node[scale=0.8,below]{$c$} (B');
\draw[->] (A) to node[scale=0.8,left]{$u\;$} (A');
\draw[->] (B) to node[scale=0.8,right]{$\;u'$}(B');
\draw[->] (A'') to node[scale=0.8,below]{$c'$} (B'');

\node at (11.5,-.75) {$\bullet$};
\node at (13,-.75) {$\bullet$};
\node at (11.5,-2.25) {$\bullet$};
\node at (13,-2.25) {$\bullet$};

\node[scale=0.8] at (12.25,-.75) {$\alpha$};
\node[scale=0.8] at (12.25,-2.25) {$\sigma_1$};

\end{tz}
Applying (db4) of \cref{doublequiv} to $\tau_0$ and $\tau_1$ gives unique squares $\sq{\sigma_0}{a}{a'}{e_A}{e_C}$ and $\sq{\sigma_1}{c}{c'}{e_{A'}}{e_{C'}}$ in $\AA$ such that $F\sigma_0=\tau_0$ and $F\sigma_1=\tau_1$. Moreover, by unicity in (db4) of \cref{doublequiv}, we have the pasting equality above right, since applying $F$ to each vertical composite yields the same squares in $\BB$. This proves~(vb3), and thus concludes the proof.
\end{proof}

Now we turn our attention to \cref{charfib}, dealing with  fibrations. Our treatment is analogous to that of weak equivalences: in order to characterize the functors $F$ such that $(\bfH,\cv)F$ is a fibration, we express what it means for $\bfH F$ and $\cv F$ to be Lack fibrations in~$\TwoCat$; this is done by translating (f1-2) of \cref{Lackfib} for these $2$-functors.

\begin{rmk}\label{Hfib}
Let $F\colon \AA\to \BB$ be a double functor. Then $\bfH F\colon \bfH \AA\to \bfH \BB$ is a fibration in~$\TwoCat$ if and only if $F$ satisfies (df1-2) of \cref{doublefib}.
\end{rmk}

\begin{rmk}\label{Vfib}
Let $F\colon \AA\to \BB$ be a double functor. Then $\cv F\colon \cv \AA\to \cv \BB$ is a fibration in~$\TwoCat$ if and only if $F$ satisfies (df3) of \cref{doublefib}, and the following condition:
\begin{enumerate}
    \item[(vf2)] for every square $\sq{\alpha'}{a'}{c'}{u}{u'}$ in $\AA$ and every square $\sq{\beta}{b}{d}{Fu}{Fu'}$ in $\BB$, together with vertically invertible squares $\tau_0$ and $\tau_1$ in $\BB$ as in the pasting equality below left, there is a square $\sq{\alpha}{a}{c}{u}{u'}$ in $\AA$, together with vertically invertible squares $\sigma_0$ and $\sigma_1$ in $\AA$ as in the pasting equality below right, such that $F\alpha=\beta$, $F\sigma_0=\tau_0$, and $F\sigma_1=\tau_1$.
\end{enumerate}
\begin{tz}
\node (A) at (0,0){$FA$};
\node (B) at (1.5,0){$FC$};
\node (A') at (0,-1.5){$FA$};
\node (B') at (1.5,-1.5){$FC$};
\node (A'') at (0,-3){$FA'$};
\node (B'') at (1.5,-3){$FC'$};
\draw[d] (A) to (A');
\draw[d] (B) to (B');
\draw[->] (A) to node[scale=0.8,above]{$b$} (B);
\draw[->] (A') to node[scale=0.8,above]{$Fa'$} (B');
\draw[->] (A') to node[scale=0.8,left]{$Fu\;$} (A'');
\draw[->] (B') to node[scale=0.8,right]{$\;Fu'$}(B'');
\draw[->] (A'') to node[scale=0.8,below]{$Fc'$} (B'');

\node at (0,-.75) {$\bullet$};
\node at (1.5,-.75) {$\bullet$};
\node at (0,-2.25) {$\bullet$};
\node at (1.5,-2.25) {$\bullet$};

\node[scale=0.8] at (.6,-.75) {$\tau_0$};
\node[scale=0.8] at (.9,-.75) {$\vcong$};
\node[scale=0.8] at (.75,-2.25) {$F\alpha'$};

\node at (2.5,-1.5) {$=$};

\node (A) at (3.5,0){$FA$};
\node (B) at (5,0){$FC$};
\node (A') at (3.5,-1.5){$FA'$};
\node (B') at (5,-1.5){$FC'$};
\node (A'') at (3.5,-3){$FA'$};
\node (B'') at (5,-3){$FC'$};
\draw[d] (A') to (A'');
\draw[d] (B') to (B'');
\draw[->] (A) to node[scale=0.8,above]{$b$} (B);
\draw[->] (A') to node[scale=0.8,below]{$d$} (B');
\draw[->] (A) to node[scale=0.8,left]{$Fu\;$} (A');
\draw[->] (B) to node[scale=0.8,right]{$\;Fu'$}(B');
\draw[->] (A'') to node[scale=0.8,below]{$Fc'$} (B'');

\node at (3.5,-.75) {$\bullet$};
\node at (5,-.75) {$\bullet$};
\node at (3.5,-2.25) {$\bullet$};
\node at (5,-2.25) {$\bullet$};

\node[scale=0.8] at (4.25,-.75) {$\beta$};
\node[scale=0.8] at (4.1,-2.25) {$\tau_1$};
\node[scale=0.8] at (4.4,-2.25) {$\vcong$};

\node (A) at (8,0){$A$};
\node (B) at (9.5,0){$C$};
\node (A') at (8,-1.5){$A$};
\node (B') at (9.5,-1.5){$C$};
\node (A'') at (8,-3){$A'$};
\node (B'') at (9.5,-3){$C'$};
\draw[d] (A) to (A');
\draw[d] (B) to (B');
\draw[->] (A) to node[scale=0.8,above]{$a$} (B);
\draw[->] (A') to node[scale=0.8,above]{$a'$} (B');
\draw[->] (A') to node[scale=0.8,left]{$u\;$} (A'');
\draw[->] (B') to node[scale=0.8,right]{$\;u'$}(B'');
\draw[->] (A'') to node[scale=0.8,below]{$c'$} (B'');

\node at (8,-.75) {$\bullet$};
\node at (9.5,-.75) {$\bullet$};
\node at (8,-2.25) {$\bullet$};
\node at (9.5,-2.25) {$\bullet$};

\node[scale=0.8] at (8.6,-.75) {$\sigma_0$};
\node[scale=0.8] at (8.9,-.75) {$\vcong$};
\node[scale=0.8] at (8.75,-2.25) {$\alpha'$};

\node at (10.5,-1.5) {$=$};

\node (A) at (11.5,0){$A$};
\node (B) at (13,0){$C$};
\node (A') at (11.5,-1.5){$A'$};
\node (B') at (13,-1.5){$C'$};
\node (A'') at (11.5,-3){$A'$};
\node (B'') at (13,-3){$C'$};
\draw[d] (A') to (A'');
\draw[d] (B') to (B'');
\draw[->] (A) to node[scale=0.8,above]{$a$} (B);
\draw[->] (A') to node[scale=0.8,below]{$c$} (B');
\draw[->] (A) to node[scale=0.8,left]{$u\;$} (A');
\draw[->] (B) to node[scale=0.8,right]{$\;u'$}(B');
\draw[->] (A'') to node[scale=0.8,below]{$c'$} (B'');

\node at (11.5,-.75) {$\bullet$};
\node at (13,-.75) {$\bullet$};
\node at (11.5,-2.25) {$\bullet$};
\node at (13,-2.25) {$\bullet$};

\node[scale=0.8] at (12.25,-.75) {$\alpha$};
\node[scale=0.8] at (12.1,-2.25) {$\sigma_1$};
\node[scale=0.8] at (12.4,-2.25) {$\vcong$};
\end{tz}
\end{rmk}

We can now use the above remarks to provide a proof of \cref{charfib}, giving the desired characterization of the fibrations in our model structure. 

\begin{proof}[Proof of \cref{charfib}]
It is clear that if a double functor $F\colon \AA\to \BB$ is such that both~$\bfH F$ and $\cv F$ are Lack fibrations in $\TwoCat$, then it is a double fibration, by \cref{Hfib,Vfib}.

Suppose now that $F\colon \AA\to \BB$ is a double fibration. By \cref{Hfib}, we directly get that~$\bfH F$ is a Lack fibration in $\TwoCat$. To show that $\cv F$ is also a Lack fibration, it suffices to show that (vf2) of \cref{Vfib} is satisfied. Let $\sq{\alpha'}{a'}{c'}{u}{u'}$ be a square in $\AA$ and $\sq{\beta}{b}{d}{Fu}{Fu'}$ be a square in $\BB$, together with vertically invertible squares $\tau_0$ and $\tau_1$ in $\BB$ as in the leftmost pasting equality diagram in (vf2). By (df2) of \cref{doublefib}, there are vertically invertible squares $\sq{\sigma_0}{a}{a'}{e_A}{e_C}$ and $\sq{\sigma_1}{c}{c'}{e_{A'}}{e_{C'}}$ in $\AA$ such that $F\sigma_0=\tau_0$ and $F\sigma_1=\tau_1$. Then the square $\alpha$ given by the vertical composite 
\begin{tz}
\node (A) at (-3.5,-1.5) {$A$};
\node (B) at (-2,-1.5) {$C$};
\node (A') at (-3.5,-3) {$A'$};
\node (B') at (-2,-3) {$C'$};
\draw[->] (A) to node[above, scale=0.8] {$a$} (B);
\draw[->] (A') to node[below, scale=0.8] {$c$} (B');
\draw[->] (A) to node[left,scale=0.8] {$u\;$} (A');
\draw[->] (B) to node[right,scale=0.8] {$\;u'$} (B');

\node at (-3.5,-2.25) {$\bullet$};
\node at (-2, -2.25) {$\bullet$};

\node[scale=0.8] at (-2.75,-2.25) {$\alpha$};

\node at (-1,-2.25) {$=$};

\node (A) at (0,0){$A$};
\node (B) at (1.5,0){$C$};
\node (A') at (0,-1.5){$A$};
\node (B') at (1.5,-1.5){$C$};
\node (A'') at (0,-3){$A'$};
\node (B'') at (1.5,-3){$C'$};
\node (A''') at (0,-4.5){$A'$};
\node (B''') at (1.5,-4.5){$C'$};
\draw[d] (A) to (A');
\draw[d] (B) to (B');
\draw[->] (A) to node[scale=0.8,above]{$a$} (B);
\draw[->] (A') to node[scale=0.8,above]{$a'$} (B');
\draw[->] (A') to node[scale=0.8,left]{$u\;$} (A'');
\draw[->] (B') to node[scale=0.8,right]{$\;u'$}(B'');
\draw[->] (A'') to node[scale=0.8,below]{$c'$} (B'');
\draw[->] (A''') to node[scale=0.8,below]{$c$} (B''');
\draw[d] (A'') to (A''');
\draw[d] (B'') to (B''');

\node at (0,-.75) {$\bullet$};
\node at (1.5,-.75) {$\bullet$};
\node at (0,-2.25) {$\bullet$};
\node at (1.5,-2.25) {$\bullet$};
\node at (0,-3.75) {$\bullet$};
\node at (1.5,-3.75) {$\bullet$};

\node[scale=0.8] at (.9,-.75) {$\vcong$};
\node[scale=0.8] at (.6,-.7) {$\sigma_0$};
\node[scale=0.8] at (.75,-2.25) {$\alpha'$};
\node[scale=0.8] at (1,-3.75) {$\vcong$};
\node[scale=0.8] at (.6,-3.75) {$\sigma_1^{-1}$};
\end{tz}
is such that $F\alpha=\beta$, which proves (vf2).
\end{proof}

\subsection{Homotopy equivalences and the Whitehead theorem} \label{subsec:whitehead}

Any model category satisfies a Whitehead theorem, stating that the weak equivalences between cofibrant-fibrant objects are precisely the homotopy equivalences; i.e., the morphisms $f\colon X\to Y$ such that there is a morphism $g\colon Y\to X$ with the property that $fg$ and $gf$ are homotopic to the identity. 
We begin by studying what the notion of homotopy entails in our setting; for this, let us first introduce the notion of horizontal pseudo natural equivalences.

\begin{defn} \label{def:pseudoeq}
Let $F,G\colon \AA\to \BB$ be double functors. A horizontal pseudo natural transformation $h\colon F\Rightarrow G$ is a \textbf{horizontal pseudo natural equivalence} if 
\begin{rome}
\item the horizontal morphism $h_A\colon FA\to GA$ is a horizontal equivalence in $\BB$, for each object $A\in \AA$, and
\item the square $\sq{h_u}{h_A}{h_{A'}}{Fu}{Gu}$ is weakly horizontally invertible in $\BB$, for each vertical morphism $u\colon A\arrowdot A'$ in $\AA$.
\end{rome} 
If the horizontal morphisms $h_A\colon FA\to GA$ are in addition horizontal adjoint equivalences in $\BB$, we say that $h$ is a \textbf{horizontal pseudo natural adjoint equivalence}. 

We write $h\colon F\simeq G$ for such a horizontal pseudo natural transformation.
\end{defn}

\begin{rmk} \label{rem:psdareeq}
By \cite[Lemma A.3.3]{Moserforth}, a horizontal pseudo natural (adjoint) equivalence as above is precisely an (adjoint) equivalence in the $2$-category $\bfH[\AA,\BB]_\ps$, or equivalently, a horizontal (adjoint) equivalence in the double category $[\AA,\BB]_\ps$.
\end{rmk}

With this definition in hand, and recalling that every double category is fibrant, we can see that two double functors between cofibrant double categories are homotopic via the path object constructed in \cref{def:pathobject} whenever they are related by a horizontal pseudo natural equivalence. 

\begin{prop} \label{lem:righthomotopy}
Let $\AA$ and $\BB$ be cofibrant double categories and let $F,G\colon \AA\to \BB$ be double functors. Then $F$ and $G$ are homotopic via the path object $\cp \BB$ of \cref{def:pathobject} if and only if there is a horizontal pseudo natural adjoint equivalence $F\simeq G$.
\end{prop}

\begin{proof}
Recall that the path object $\cp\BB$ of \cref{def:pathobject} is given by the pseudo hom double category $[\bbH \Eadj,\BB]_\ps$, where the $2$-category $\Eadj$ is the free-living adjoint equivalence $\{0\xrightarrow{\simeq} 1\}$.  Therefore, a homotopy between double functors $F,G\colon \AA\to \BB$ via $\cp \BB$ is a double functor $h\colon \AA\to [\bbH\Eadj,\BB]_\ps$
such that $Ph=(F,G)$ or, equivalently, a double functor \[\widehat{h}\colon \bbH\Eadj \to [\AA,\BB]_\ps\] such that $\widehat{h}(0)=F$ and $\widehat{h}(1)=G$. This corresponds to a horizontal pseudo natural adjoint equivalence $F\simeq G$ by \cref{rem:psdareeq}.
\end{proof}

\begin{rmk} \label{rem:usualWhitehead}
By the usual Whitehead theorem (see, for example, \cite[Lemma 4.24]{DS}), a morphism between cofibrant-fibrant objects in a model category is a weak equivalence if and only if it is a homotopy equivalence. Hence,  since all double categories are fibrant in the model structure of \cref{thm:modelstructonDblCat}, we can use \cref{lem:righthomotopy} to characterize double biequivalences between cofibrant objects in $\DblCat$ as those double functors which admit an inverse up to horizontal pseudo natural adjoint equivalence, i.e., double functors $F\colon \AA\to \BB$ such that there is a double functor $G\colon \BB\to \AA$ together with horizontal pseudo natural adjoint equivalences $\id_\AA\simeq GF$ and $FG\simeq \id_\BB$.
\end{rmk}

In our double categorical setting, we can prove a version of the Whitehead theorem for a wider class of weak equivalences, by only imposing a condition on their target double categories. However, in some cases, the homotopy inverse is not a strict double functor anymore, but it is rather pseudo in the horizontal direction.

\begin{defn}
A \textbf{horizontally pseudo double functor} $F\colon \AA\to \BB$ consists of maps on objects, horizontal morphisms, vertical morphisms, and squares, which are compatible with domains and codomains. These maps preserve identities and compositions of vertical morphisms and squares strictly, but they preserve identities and compositions of horizontal morphisms only up to vertically invertible squares. These are submitted to associativity, unitality, and naturality conditions. See \cite[Definition 3.5.1]{Grandis} for details (note, however, that our definition has reversed the roles of the horizontal and vertical directions). 

If $F$ strictly preserves horizontal identities, we say that $F$ is \textbf{normal}. 
\end{defn}

\begin{rmk}
Analogously to \cref{def:pseudohortransf,def:pseudoeq}, we have notions of horizontal and vertical pseudo natural transformations, modifications, and horizontal pseudo natural equivalences between horizontally pseudo double functors. See \cite[\S 3.8]{Grandis} for precise definitions; note that our definition has reversed the roles of the horizontal and vertical directions.
\end{rmk}

Our class of double biequivalences contains in particular the double functors that have a horizontally pseudo inverse up to horizontal pseudo natural equivalence. 

\begin{prop} \label{prop:horeqarewe}
Let $F\colon \AA\to \BB$ be a double functor. If there is a normal horizontally pseudo double functor ${G\colon \BB\to \AA}$ together with horizontal pseudo natural equivalences $\eta\colon \id_\AA\simeq GF$ and $\epsilon\colon FG\simeq \id_\BB$, then $F$ is a double biequivalence. 
\end{prop}

\begin{proof}
Under these assumptions, the double functor $F$ is in particular a horizontal biequivalence as introduced in \cite[Definition 8.7]{WHI}. Therefore $F$ is a double biequivalence by \cite[Proposition 8.10]{WHI}. 
\end{proof}

By only requiring that the target of a double biequivalence $F$ does not contain any non-trivial composites of vertical morphisms, we can construct a horizontally pseudo double functor which gives a homotopy inverse of $F$. As the construction of this homotopy inverse is practically identical to the one in \cite[Proposition 8.11]{WHI}, we only specify here the data of the pseudo inverse and of one of the horizontal pseudo natural equivalences, and refer the reader to the proof of \cite[Proposition 8.11]{WHI} for details.  

\begin{thm}\label{thm:Whitehead1}
Let $\AA$ and $\BB$ be double categories such that the underlying vertical category $U\bfV\BB$ is a disjoint union of copies of $\mathbbm 1$ and $\mathbbm 2$. Then a double functor $F\colon \AA\to \BB$ is a double biequivalence if and only if there is a normal horizontally pseudo double functor ${G\colon \BB\to \AA}$, and horizontal pseudo natural equivalences $\eta\colon \id_\AA\simeq GF$ and $\epsilon\colon FG\simeq \id_\BB$. 
\end{thm}

\begin{proof}
By \cref{prop:horeqarewe}, we directly get the converse implication. 

Now suppose that $F$ is a double biequivalence. We highlight the definition of the horizontally pseudo double functor $G\colon \BB\to \AA$ and the horizontal pseudo natural equivalence $\epsilon\colon FG\Rightarrow \id_\BB$ on objects and vertical morphisms as it is the only part of the construction that differs from \cite[Proposition 8.11]{WHI}. One can easily check that the rest of the proof of \cite[Proposition 8.11]{WHI} does not depend on the weakly horizontally invariant condition that is not required in this statement, and thus can be applied verbatim.

To define $G$ and $\epsilon$ on objects and vertical morphisms, we give the values of $G$ and $\epsilon$ on each copy of $\mathbbm 1$ and $\mathbbm 2$ in $U\bfV\BB$.
\begin{itemize}
    \item Given a copy of the form $B\colon \mathbbm 1\to U\bfV\BB$, by (db1) applied to the object $B\in \BB$, we get an object $A\in \AA$ and a horizontal equivalence $f\colon FA\xrightarrow{\simeq} B$ in~$\BB$. We set $GB\coloneqq A$ and $\epsilon_B\coloneqq f\colon FGB\xrightarrow{\simeq} B$. 
    \item Given a copy of the form $v\colon \mathbbm 2\to U\bfV\BB$, by (db3) applied to the vertical morphism $v\colon B\arrowdot B'$ in $\BB$, we get a vertical morphism $u\colon A\arrowdot A'$ in $\AA$ and a weakly horizontally invertible square $\beta$ in $\BB$ as follows.
\begin{tz}
\node[](1) {$FA$}; 
\node[below of=1](2) {$FA'$}; 
\node[right of=1](3) {$B$}; 
\node[below of=3](4) {$B'$}; 
\draw[->] (1) to node[above,la] {$f$} node[below,la] {$\simeq$} (3);
\draw[->] (2) to node[above,la] {$\simeq$} node[below,la] {$g$} (4);
\draw[->,pro] (1) to node[left,la] {$Fu$} (2);
\draw[->,pro] (3) to node[right,la] {$v$} (4);
\node[la,xshift=-7pt] at ($(1)!0.5!(4)$) {$\beta$};
\node[la,xshift=7pt] at ($(1)!0.5!(4)$) {$\simeq$};
\end{tz}
    We set $GB\coloneqq A$, $GB'\coloneqq A'$, and $Gv\coloneqq u$, and we set $\epsilon_B\coloneqq f\colon FGB\xrightarrow{\simeq} B$, $\epsilon_{B'}\coloneqq g\colon FGB'\xrightarrow{\simeq} B'$, and $\sq{\epsilon_v\coloneqq \beta}{\epsilon_B}{\epsilon_{B'}}{FGv}{v}$.
\end{itemize}
As there are no composites of vertical morphisms in $\BB$, $G$ and $\epsilon$ are trivially compatible with vertical morphisms.  
\end{proof}

\begin{rmk}
If we further require that the double category $\BB$ in \cref{thm:Whitehead1} is cofibrant, we can construct the weak inverse $G\colon \BB\to \AA$ of $F$ in such a way that it is a strict double functor, since the underlying horizontal category of $\BB$ is free. This subsumes the usual Whitehead theorem mentioned in \cref{rem:usualWhitehead}.
\end{rmk}

Finally, as a horizontal double category has a discrete underlying vertical category, the result applies in particular to the case where $\BB$ is horizontal. We then retrieve the Whitehead theorem for $2$-categories, which can be found in \cite[Theorem 7.4.1]{JohYau}. 

\begin{cor}
Let $\A$ and $\B$ be $2$-categories. Then a $2$-functor $F\colon \A\to \B$ is a biequivalence if and only if there is a normal pseudo functor $G\colon \B\to \A$ together with pseudo natural equivalences $\eta\colon \id_\A\simeq GF$ and $\epsilon\colon FG\simeq \id_\B$. 
\end{cor}

\begin{proof}
Since $F$ is a biequivalence if and only if $\bbH F$ is a double biequivalence, as we will see in \cref{thm:leftandrightindLHH}, and $\bbH\B$ is horizontal, we can apply \cref{thm:Whitehead1} to $\bbH F\colon \bbH\A\to \bbH\B$. Then $\bbH F$ is a double biequivalence if and only if there is a normal horizontally pseudo double functor $G'\colon \bbH\B\to \bbH\A$ together with horizontal pseudo natural equivalences $\eta'\colon \id_{\bbH\A}\simeq G'(\bbH F)$ and $\epsilon'\colon (\bbH F)G'\simeq \id_{\bbH\B}$. As normal horizontally pseudo double functors and horizontal pseudo natural equivalence between double categories in the image of $\bbH$ are equivalently normal pseudo functors and pseudo natural equivalences between their preimages, the data $(G',\eta',\epsilon')$ for $\bbH F$ uniquely correspond to a data $(G,\eta,\epsilon)$ for $F$ as required. 
\end{proof}

\section{Quillen pairs between \texorpdfstring{$\DblCat$}{DblCat}, \texorpdfstring{$\TwoCat$}{2Cat}, and \texorpdfstring{$\Cat$}{Cat}} \label{Sec:Quillenpairs}

In this paper, the model structure on $\DblCat$ was constructed in such a way as to be compatible with the Lack model structure on $\TwoCat$. We show in \cref{subsec:QuillenpairH} that the horizontal embedding $\bbH\colon \TwoCat\to \DblCat$ is both left and right Quillen, and homotopically fully faithful. This implies that the functor $\bbH$ embeds the homotopy category of $2$-categories into that of double categories in a reflective and coreflective way. Among other things, this means that the functor $\bbH$ creates all homotopy limits and colimits. Finally, our last results in \cref{subsec:QuillenpairH} show that the Lack model structure on $\TwoCat$ is both left- and right-induced from our model structure on $\DblCat$ along $\bbH$. 

In \cref{subsec:QuillenpairCat}, we give a Quillen pair between $\Cat$ and $\DblCat$, which horizontally embeds the canonical homotopy theory of $\Cat$ into that of $\DblCat$ in a reflective way. We also show that the canonical model structure on $\Cat$ is right-induced from ours.

\subsection{Quillen pairs involving \texorpdfstring{$\bbH$}{H}}\label{subsec:QuillenpairH}

We present here the two Quillen pairs involving the functor $\bbH\colon \TwoCat\to \DblCat$ and its right and left adjoints.

\begin{prop} \label{QuillenadjHH}
The adjunction 
\begin{tz}
\node (A) at (0,0) {$\TwoCat$};
\node (B) at (2.5,0) {$\DblCat$};
\draw[->] ($(A.east)+(0,.25cm)$) to [bend left] node[above,scale=0.8]{$\bbH$} ($(B.west)+(0,.25cm)$);
\draw[->] ($(B.west)+(0,-.25cm)$) to [bend left] node[below,scale=0.8]{$\bfH$} ($(A.east)-(0,.25cm)$);
\node[scale=0.8] at ($(A.east)!0.5!(B.west)$) {$\bot$};
\end{tz}
is a Quillen pair, where $\TwoCat$ is endowed with the Lack model structure and $\DblCat$ is endowed with the model structure of \cref{thm:modelstructonDblCat}. Moreover, its derived unit is levelwise an identity; in particular, this says that the functor $\bbH$ is homotopically fully faithful. 
\end{prop}

\begin{proof}
Since $(\bfH,\cv)\colon \DblCat\to \TwoCat\times \TwoCat$ and the projection $\mathrm{pr}_1\colon \TwoCat\times \TwoCat\to \TwoCat$ are right Quillen, then so is their composite $\bfH\colon \DblCat\to \TwoCat$, which proves that $\bbH\dashv \bfH$ is a Quillen pair. Moreover, since every object in $\DblCat$ is fibrant, the derived unit of the adjunction $\bbH\dashv \bfH$ is given by the components of the unit at cofibrant objects, and is therefore levelwise an identity, by \cref{prop:adjHH}.
\end{proof}

The functor $\bbH\colon \TwoCat\to \DblCat$ also admits a left adjoint $L$. Indeed, this is given by the Adjoint Functor Theorem, since $\bbH$ preserves all limits and colimits, and the categories involved are locally presentable. The next theorem shows that $\bbH$ is also a right Quillen functor.

\begin{thm} \label{thm:LHQuillenadj} 
The adjunction 
\begin{tz}
\node (A) at (0,0) {$\DblCat$};
\node (B) at (2.5,0) {$\TwoCat$};
\draw[->] ($(A.east)+(0,.25cm)$) to [bend left] node[above,scale=0.8]{$L$} ($(B.west)+(0,.25cm)$);
\draw[->] ($(B.west)+(0,-.25cm)$) to [bend left] node[below,scale=0.8]{$\bbH$} ($(A.east)-(0,.25cm)$);
\node[scale=0.8] at ($(A.east)!0.5!(B.west)$) {$\bot$};
\end{tz}
is a Quillen pair, where $\TwoCat$ is endowed with the Lack model structure and $\DblCat$ is endowed with the model structure of \cref{thm:modelstructonDblCat}.
\end{thm}

\begin{proof}
We show that $\bbH$ is right Quillen, i.e.,~it preserves fibrations and trivial fibrations. 

Let $F\colon \A\to \B$ be a fibration in $\TwoCat$; we prove that $\bbH F\colon \bbH\A\to \bbH\B$ is a double fibration in $\DblCat$. Since $\bfH\bbH F=F$ and $F$ is a fibration, (df1-2) of \cref{doublefib} are satisfied. It remains to show (df3) of \cref{doublefib}. Let us consider a weakly horizontally invertible square in $\bbH\B$
\begin{tz}
\node (A) at (0,0) {$B$};
\node (B) at (1.5,0) {$FC$};
\node (A') at (0,-1.5) {$B$};
\node (B') at (1.5,-1.5) {$FC$}; 
\node at ($(B'.east)-(0,4pt)$) {.};
\draw[->] (A) to node[above, scale=0.8]{$\simeq$} node[below,scale=0.8] {$b$} (B);
\draw[->] (A') to node[above, scale=0.8]{$\simeq$} node[below, scale=0.8] {$d$} (B');
\draw[d] (A) to (A');
\draw[d] (B) to (B');

\node at (0,-.75) {$\bullet$};
\node at (1.5, -.75) {$\bullet$};

\node[scale=0.8] at (.6,-.75) {$\beta$};
\node[scale=0.8] at (.9,-.75) {$\vcong$};
\end{tz}
Note that its vertical boundaries must be trivial, since all vertical morphisms in $\bbH \B$ are identities. Then the square $\beta$ is, in particular, vertically invertible by \cref{lem:weaklyhorinvvsverinv}. Since $F$ is a fibration in $\TwoCat$, there is an equivalence $c\colon A\xrightarrow{\simeq} C$ such that $Fc=d$, by (f1) of \cref{Lackfib}. Now $\beta$ can be rewritten as 
\begin{tz}
\node (A) at (0,0) {$FA$};
\node (B) at (1.5,0) {$FC$};
\node (A') at (0,-1.5) {$FA$};
\node (B') at (1.5,-1.5) {$FC$};
\node at ($(B'.east)-(0,4pt)$) {.};
\draw[->] (A) to node[above, scale=0.8]{$\simeq$} node[below,scale=0.8] {$b$} (B);
\draw[->] (A') to node[above, scale=0.8]{$\simeq$} node[below, scale=0.8] {$Fc$} (B');
\draw[d] (A) to (A');
\draw[d] (B) to (B');

\node at (0,-.75) {$\bullet$};
\node at (1.5, -.75) {$\bullet$};

\node[scale=0.8] at (.6,-.75) {$\beta$};
\node[scale=0.8] at (.9,-.75) {$\vcong$};
\end{tz}
Then $\beta$ is equivalently an invertible $2$-cell $\beta\colon b\cong Fc$ in $\B$. Since $F$ is a fibration in $\TwoCat$, there is a morphism $a\colon A\to C$ in $\A$ and an invertible $2$-cell $\alpha \colon a\cong c$ in $\A$ such that $F\alpha=\beta$, by (f2) of \cref{Lackfib}. In particular, since $c$ is an equivalence in $\A$, then so is $a$. This gives a vertically invertible square in $\bbH\A$ of the form
\begin{tz}
\node (A) at (0,0) {$A$};
\node (B) at (1.5,0) {$C$};
\node (A') at (0,-1.5) {$A$};
\node (B') at (1.5,-1.5) {$C$};
\draw[->] (A) to node[above, scale=0.8]{$\simeq$} node[below, scale=0.8] {$a$} (B);
\draw[->] (A') to node[above, scale=0.8]{$\simeq$} node[below, scale=0.8] {$c$} (B');
\draw[d] (A) to (A');
\draw[d] (B) to (B');

\node at (0,-.75) {$\bullet$};
\node at (1.5, -.75) {$\bullet$};

\node[scale=0.8] at (.6,-.75) {$\alpha$};
\node[scale=0.8] at (.9,-.75) {$\vcong$};
\end{tz}
such that $F\alpha=\beta$; furthermore, by \cref{lem:weaklyhorinvvsverinv}, the square $\alpha$ is weakly horizontally invertible. This shows that $\bbH F$ is a double fibration. 

Now let $F\colon \A\to \B$ be a trivial fibration in $\TwoCat$. We show that $\bbH F\colon \bbH\A\to \bbH\B$ is a double trivial fibration in $\DblCat$. Since $\bfH\bbH F=F$ and $F$ is a trivial fibration, it satisfies (dt1-2) of \cref{doubletrivfib}. Then (dt3) of \cref{doubletrivfib} follows from the fact that~$F$ is surjective on objects, since all vertical morphisms are identities. Finally, (dt4) of \cref{doubletrivfib} is a direct consequence of $F$ being fully faithful on $2$-cells, since all squares in $\bbH\A$ and $\bbH\B$ are equivalently $2$-cells in $\A$ and $\B$, respectively. This shows that $\bbH F$ is a double trivial fibration, and concludes the proof of $L\dashv \bbH$ being a Quillen pair. 
\end{proof}

\begin{rmk}
As we have seen in \cref{QuillenadjHH}, the functor $\bbH$ is homotopically fully faithful, and therefore the derived counit of the adjunction $L\dashv \bbH$ is levelwise a  biequivalence. 
\end{rmk}

\begin{rmk} \label{rem:Hprescofibwe}
As a consequence of \cref{QuillenadjHH,thm:LHQuillenadj}, we can see that the functor $\bbH\colon \TwoCat\to \DblCat$ preserves all cofibrations, fibrations, and weak equivalences. Indeed, the fact that it preserves cofibrations and fibrations follows from the fact that $\bbH$ is both left and right Quillen, while the fact that it preserves weak equivalences is a consequence of Ken Brown's Lemma (see \cite[Lemma 1.1.12]{Hovey}), since all objects in $\TwoCat$ are fibrant. Moreover, since $\bbH$ is homotopically fully faithful by \cref{QuillenadjHH}, this says that the homotopy theory of $\TwoCat$ is reflectively and coreflectively embedded in that of~$\DblCat$ via the functor~$\bbH$.  
\end{rmk}

In fact, more is true: the Lack model structure on $\TwoCat$ is actually both left- and right-induced from our model structure on $\DblCat$ along the horizontal embedding $\bbH$, which implies that $\bbH$ also reflects cofibrations, fibrations, and weak equivalences. 

\begin{thm} \label{thm:leftandrightindLHH}
The Lack model structure on $\TwoCat$ is both left- and right-induced along the adjunctions
\begin{tz}
\node[](A) {$\TwoCat$};
\node[right of=A,rr](B) {$\DblCat$};
\node at ($(B.east)-(0,4pt)$) {,};
\draw[->] ($(B.west)+(0,.25cm)$) to [bend right=50] node[above,la]{$L$} ($(A.east)+(0,.25cm)$);
\draw[->] (A) to node[la,over] {$\bbH$} (B);
\draw[->] ($(B.west)+(0,-.25cm)$) to [bend left=50] node[below,la]{$\bfH$} ($(A.east)-(0,.25cm)$);
\node[la] at ($(A.east)!0.5!(B.west)+(0,.35cm)$) {$\bot$};
\node[la] at ($(A.east)!0.5!(B.west)-(0,.35cm)$) {$\bot$};
\end{tz}
where $\DblCat$ is endowed with the model structure of \cref{thm:modelstructonDblCat}.
\end{thm}

\begin{proof}
To show this result, it is enough to prove that a $2$-functor $F\colon \A\to \B$ is a biequivalence (resp.~Lack fibration, cofibration) in $\TwoCat$ if and only if the double functor $\bbH F\colon\bbH\A\to \bbH\B$ is a double biequivalence (resp.~double fibration, cofibration) in $\DblCat$, as a model structure is uniquely determined by its classes of weak equivalences and fibrations, or alternatively by its classes of weak equivalences and cofibrations.

By \cref{rem:Hprescofibwe}, we have that if $F$ is a biequivalence (resp.~Lack fibration, cofibration) in $\TwoCat$, then $\bbH F$ is a double biequivalence (resp.~double fibration, cofibration) in $\DblCat$, as $\bbH$ preserves all of these classes of morphisms.

Conversely, if $\bbH F$ is a double biequivalence (resp.~double fibration), then $\bfH\bbH F=F$ is a biequivalence (resp.~Lack fibration) by definition of the model structure on $\DblCat$.

It remains to show that if $\bbH F$ is a cofibration, then so is $F$. For this, suppose that $\bbH F$ is a cofibration in $\DblCat$; we show that $F$ has the left lifting property with respect to all trivial fibrations in $\TwoCat$. Let $P\colon \mathcal X\to \mathcal Y$ be a trivial fibration in $\TwoCat$ and suppose we have a commutative square as below.
\begin{tz}
\node (A) at (0,0) {$\A$};
\node (B) at (0,-1.5) {$\B$};
\node (A') at (1.5,0) {$\mathcal X$}; 
\node (B') at (1.5,-1.5) {$\mathcal Y$};

\draw[->] (A) to node[above,scale=0.8] {$G$} (A');
\draw[->] (A) to node[left,scale=0.8] {$F$} (B);
\draw[->] (B) to node[below,scale=0.8] {$H$} (B');
\draw[->] (A') to node[right,scale=0.8] {$P$} (B');
\end{tz}
Since $\bbH$ preserves trivial fibrations, we have that $\bbH P$ is a double trivial fibration. Then, as $\bbH F$ is a cofibration, there is a lift in the diagram below left. By the adjunction $\bbH\dashv \bfH$, this corresponds to a lift in the diagram below right, which concludes the proof. 
\[  \begin{tikzpicture}
\node (A) at (-.5,0) {$\bbH \A$};
\node (B) at (-.5,-1.5) {$\bbH \B$};
\node (A') at (1.5,0) {$\bbH \mathcal X$}; 
\node (B') at (1.5,-1.5) {$\bbH \mathcal Y$};

\draw[->] (A) to node[above,scale=0.8] {$\bbH G$} (A');
\draw[->] (A) to node[left,scale=0.8] {$\bbH F$} (B);
\draw[->] (B) to node[below,scale=0.8] {$\bbH H$} (B');
\draw[->] (A') to node[right,scale=0.8] {$\bbH P$} (B');
\draw[dotted,->] (B) to (A');

\node (A) at (4.5,0) {$\A$};
\node (B) at (4.5,-1.5) {$\B$};
\node (A') at (6.5,0) {$\bfH \bbH \mathcal X=\mathcal X$}; 
\node (B') at (6.5,-1.5) {$\bfH \bbH \mathcal Y=\mathcal Y$}; 

\draw[->] (A) to node[above,scale=0.8] {$G$} (A');
\draw[->] (A) to node[left,scale=0.8] {$F$} (B);
\draw[->] (B) to node[below,scale=0.8] {$H$} (B');
\draw[->] (A') to node[right,scale=0.8] {$\bfH \bbH P=P$} (B');
\draw[dotted,->] (B) to (A');
\end{tikzpicture}\qedhere \]
\end{proof}

We saw that the derived unit (resp.~counit) of the adjunction $\bbH\dashv \bfH$ (resp.~$L\dashv \bbH$) is levelwise a biequivalence. However, these adjunctions are not expected to be Quillen equivalences, since the homotopy theory of double categories should be richer than that of $2$-categories. This is indeed the case, as shown in the following remarks.

\begin{rmk}
The components of the derived counit of the adjunction $\bbH\dashv \bfH$ are not double biequivalences. To see this, consider the double category $\vtwo$ free on a vertical morphism. Since $\bfH\vtwo\cong \mathbbm{1}\sqcup \mathbbm{1}$ is cofibrant in $\TwoCat$, the component of the derived counit at $\vtwo$ is given by the component of the counit
\[ \epsilon_{\vtwo}\colon \bbH\bfH(\vtwo)\cong\mathbbm{1}\sqcup \mathbbm{1}\to \vtwo, \]
which is not a double biequivalence, as it does not satisfy (db3) of \cref{doublequiv}.
\end{rmk}

\begin{rmk}
The components of the derived unit of the adjunction $L\dashv \bbH$ are not double biequivalences. By \cref{prop:alternate_gen_cofibs}, the unique map $I_3\colon \emptyset\to \vtwo$ is a generating cofibration in $\DblCat$, so that $\vtwo$ is cofibrant. Since all objects in $\TwoCat$ are fibrant, the component of the derived unit at $\vtwo$ is given by the component of the unit
\[ \eta_{\vtwo}\colon \vtwo\to \bbH L(\vtwo)\cong \mathbbm{1}, \]
which is not a double biequivalence, as it does not satisfy (db2) of \cref{doublequiv}. Note that the isomorphism above comes from the fact that the left adjoint $L$ collapses the vertical structure and thus $L\vtwo\cong\mathbbm{1}$. 
\end{rmk}

\begin{rmk} \label{LVQuillenadj}
Since we induced the model structure on $\DblCat$ along $\bbH\sqcup \LV\dashv (\bfH,\cv)$, we also get that the adjunction $\LV\dashv \cv$ forms a Quillen pair between $\TwoCat$ and $\DblCat$. However, note that neither the derived unit nor counit of $\LV\dashv \cv$ are levelwise weak equivalences.
\end{rmk}

\subsection{Quillen pairs to \texorpdfstring{$\Cat$}{Cat}}\label{subsec:QuillenpairCat}

The category $\Cat$ of categories and functors also admits a model structure, called the \emph{canonical model structure}, in which the weak equivalences are the equivalences of categories and the fibrations are the isofibrations. As shown by Lack in~\cite{Lack2Cat}, with this model structure, the homotopy theory of categories is reflectively embedded in the homotopy theory of $2$-categories. Combining this result with the one of \cref{thm:LHQuillenadj}, we get that the homotopy theory of categories is also reflectively embedded in that of double categories.

\begin{notation} \label{Notation:discrete}
We write $D\colon \Cat\to \TwoCat$ for the functor that sends a category to the $2$-category with the same objects and morphisms, and with only identity $2$-cells. We write $P\colon \TwoCat\to \Cat$ for its left adjoint. In particular, the functor $P$ sends a $2$-category $\A$ to the category $P\A$ with the same objects as $\A$ and with hom sets $P\A(A,B)=\pi_0(\A(A,B))$, for every pair of objects $A,B\in \A$, where $\pi_0\colon\Cat\to\Set$ is the functor sending a category to its set of connected components.
\end{notation}

\begin{rmk} 
In \cite[Theorem 8.2]{Lack2Cat}, Lack shows that the adjunction $P\dashv D$ is a Quillen pair between the Lack model structure on $\TwoCat$ and the canonical model structure on $\Cat$, whose derived counit is levelwise a weak equivalence, but the derived unit is not. Composing this Quillen pair with the one of \cref{thm:LHQuillenadj}, we get a Quillen adjunction $PL\dashv \bbH D$ between the model structure on $\DblCat$ of \cref{thm:modelstructonDblCat} and the canonical model structure on $\Cat$ whose derived counit is levelwise an equivalence of categories. In particular, the horizontal embedding $\bbH D\colon \Cat\to \DblCat$ is homotopically fully faithful.
\end{rmk}

The above remark guarantees that the functors $D\colon \Cat\to \TwoCat$ and $\bbH D\colon \Cat\to \DblCat$ preserve fibrations and weak equivalences, since all categories are fibrant. Furthermore, the following results imply that these two functors create fibrations and weak equivalences, since the canonical model structure on $\Cat$ is right-induced from the ones on $\TwoCat$ and~$\DblCat$.

\begin{prop} \label{prop:PDrightinduced}
The canonical model structure on $\Cat$ is right-induced from the adjunction
\begin{tz}
\node (A) at (0,0) {$\TwoCat$};
\node (B) at (2.25,0) {$\Cat$};
\node at ($(B.east)-(0,4pt)$) {,};
\draw[->] ($(A.east)+(0,.25cm)$) to [bend left] node[above,scale=0.8]{$P$} ($(B.west)+(0,.25cm)$);
\draw[->] ($(B.west)+(0,-.25cm)$) to [bend left] node[below,scale=0.8]{$D$} ($(A.east)-(0,.25cm)$);
\node[scale=0.8] at ($(A.east)!0.5!(B.west)$) {$\bot$};
\end{tz}
where $\TwoCat$ is endowed with the Lack model structure.
\end{prop}

\begin{proof}
Let $F\colon\mathcal{C}\to\mathcal{D}$ be a functor in $\Cat$. It suffices to show that $F$ is an equivalence (resp.~isofibration) if and only if $DF$ is a biequivalence (resp.~Lack fibration). Indeed, both statements can easily seen to be true, due to the fact that, for every category $\mathcal{A}$, the $2\text{-category}$~$D\mathcal{A}$ only has trivial $2$-cells, and thus a morphism in $D\mathcal{A}$ is an equivalence precisely if it is an isomorphism in $\A$.
\end{proof}

\begin{cor}\label{cor:cat_from_dblcat}
The canonical model structure on $\Cat$ is right-induced from the adjunction
\begin{tz}
\node (A) at (0,0) {$\DblCat$};
\node (B) at (2.5,0) {$\Cat$};
\node at ($(B.east)-(0,4pt)$) {,};
\draw[->] ($(A.east)+(0,.25cm)$) to [bend left] node[above,scale=0.8]{$PL$} ($(B.west)+(0,.25cm)$);
\draw[->] ($(B.west)+(0,-.25cm)$) to [bend left] node[below,scale=0.8]{$\bbH D$} ($(A.east)-(0,.25cm)$);
\node[scale=0.8] at ($(A.east)!0.5!(B.west)$) {$\bot$};
\end{tz}
where $\DblCat$ is endowed with the model structure of \cref{thm:modelstructonDblCat}. 
\end{cor}

\begin{proof}
This follows directly from \cref{prop:PDrightinduced,thm:leftandrightindLHH}.
\end{proof}

\section{\texorpdfstring{$\TwoCat$}{2Cat}-enrichment of the model structure on \texorpdfstring{$\DblCat$}{DblCat}} \label{Sec:enrichment}

The aim of this section is to provide a $\TwoCat$-enrichment on $\DblCat$ which is compatible with the model structure introduced in \cref{thm:modelstructonDblCat}. Recall that a model category $\cm$ is said to be \emph{enriched} over a closed monoidal category $\cn$ that is also a model category, if it is a tensored and cotensored $\cn$-enriched category and it satisfies the pushout-product axiom (see for example \cite[\S 5]{Moser} for more details). In particular, the category $\cn$ is said to be a \emph{monoidal model category} if its model structure is enriched over itself.

In \cite{Lack2Cat}, it is shown that the Lack model structure is not monoidal with respect to the cartesian product. However, it is established that such a compatibility exists when considering instead the closed symmetric monoidal structure on $\TwoCat$ given by the Gray tensor product, as stated in \cref{2Catenrichedmodel}.

Similarly, the category of double categories admits two closed symmetric monoidal structures, given by the cartesian product, and by an analogue of the Gray tensor product introduced by B\"ohm in \cite{Bohm}. We show in \cref{subsec:Dblnotmonoidal} that the category $\DblCat$ is not a monoidal model category with respect to either of these monoidal structures. 

Nevertheless, the Gray tensor product on $\DblCat$ is not entirely unrelated to our model structure. By restricting this tensor product in one of the variables to $\TwoCat$ along $\bbH$, we obtain in \cref{subsec:enrichment} a $\TwoCat$-enrichment on $\DblCat$ which is compatible with our model structure.

\subsection{The model structure on \texorpdfstring{$\DblCat$}{DblCat} is not monoidal} \label{subsec:Dblnotmonoidal} As shown in the remark below, a similar argument to Lack's \cite[Example 7.2]{Lack2Cat} also applies in the case of $\DblCat$ to see that the model structure on $\DblCat$ is not monoidal with respect to the cartesian product.

\begin{rmk}
By \cref{prop:alternate_gen_cofibs}, the inclusion $I_2\colon \mathbbm 1\sqcup \mathbbm 1\to \bbH\mathbbm 2$ is a generating cofibration in~$\DblCat$. However, the pushout product $\pushout{I_2}{I_2}{}$ with respect to the cartesian product is the double functor from the non-commutative square of horizontal morphisms to the commutative square of horizontal morphisms, as in \cite[Example 7.2]{Lack2Cat}. Since cofibrations in $\DblCat$ are in particular faithful on horizontal morphisms by \cref{injectivityofcof}, the pushout-product $\pushout{I_2}{I_2}{}$ cannot be a cofibration in $\DblCat$.
\end{rmk}

As we mentioned before, a Gray tensor product for double categories is introduced by B\"ohm in~\cite{Bohm}. In the same vein as in  the $2$-categorical case, the corresponding internal homs make use of the notions of horizontal and vertical pseudo natural transformations, and modifications between them, which were described in \cref{def:pseudohortransf}.

\begin{prop}[{\cite[\S 3]{Bohm}}] \label{prop:Bohm}
There is a symmetric monoidal structure on $\DblCat$ given by the Gray tensor product 
 \[ \otimes_\Gray\colon \DblCat\times \DblCat\to \DblCat. \]
 Moreover, this monoidal structure is closed: for all double categories $\AA$, $\BB$, and $\bC$, there is an isomorphism
\[ \DblCat(\AA\otimes_\Gray \BB, \bC)\cong \DblCat(\AA,[\BB,\bC]_\ps), \]
natural in $\AA$, $\BB$ and $\bC$, where $[-,-]_\ps$ is the pseudo hom double category given in \cref{def:pseudohomdouble}.
\end{prop}

Remember from \cref{char:cof} that our cofibrations are not as well behaved in the vertical direction as in the horizontal direction; e.g., the underlying vertical category of a cofibrant double category is only a disjoint union of copies of $\mathbbm 1$ and $\mathbbm 2$ rather than a free category. As a consequence, our model structure is not compatible with the Gray tensor product on $\DblCat$, as we show below. 

\begin{notation}
Let $I\colon \AA\to \BB$ and $J\colon \AA'\to \BB'$ be double functors in $\DblCat$. We write $\pushout{I}{J}{\Gray}$ for their pushout-product
\[ \pushprod{I}{\AA}{\BB}{J}{\AA'}{\BB'}{\Gray} \]
with respect to the Gray tensor product $\otimes_\Gray$ on $\DblCat$.
\end{notation}

\begin{rmk} \label{NotcompatiblewithGray}
The model structure defined in \cref{thm:modelstructonDblCat} is not compatible with the Gray tensor product $\otimes_\Gray$. To see this, recall that $I_3\colon \emptyset\to \vtwo$ is a generating cofibration in $\DblCat$ by \cref{prop:alternate_gen_cofibs}. However the pushout-product \[
\pushout{I_3}{I_3}{\Gray}\colon \emptyset\to \vtwo\otimes_\Gray\vtwo \]
is not a cofibration, where $ \vtwo\otimes_\Gray\vtwo$ is the double category generated by the following data
\begin{tz}
\node (A) at (0,0) {$0$};
\node (B) at (0,-1.5) {$0'$};
\node (C) at (0,-3) {$1'$};
\node (A') at (1.5,0) {$0$};
\node (B') at (1.5,-1.5) {$1$};
\node (C') at (1.5,-3) {$1'$};
\node at ($(C'.east)-(0,4pt)$) {.};
\draw[d] (A) to (A');
\draw[d] (C) to (C');
\draw[->] (A) to (B);
\draw[->] (B) to (C);
\draw[->] (A') to (B');
\draw[->] (B') to (C');

\node at (0,-.75) {$\bullet$};
\node at (1.5,-.75) {$\bullet$};
\node at (0,-2.25) {$\bullet$};
\node at (1.5,-2.25) {$\bullet$};

\node[scale=0.8] at (.75,-1.5) {$\cong$};
\end{tz}
Indeed, since the underlying vertical category of $ \vtwo\otimes_\Gray\vtwo$ has non-trivial composites of vertical morphisms, this is not a cofibrant double category by \cref{char:cofibrant}.
\end{rmk}

\subsection{\texorpdfstring{$\TwoCat$}{2Cat}-enrichment of the model structure on \texorpdfstring{$\DblCat$}{DblCat}}\label{subsec:enrichment} By restricting the Gray tensor product on $\DblCat$ along $\bbH$ in one of the variables, we get rid of the issue concerning the vertical structure that obstructs the compatibility with the model structure of \cref{thm:modelstructonDblCat}. With this variation, we show that $\DblCat$ is a tensored and cotensored $\TwoCat$-enriched category, and that the corresponding enrichment is now compatible with our model structure. 

\begin{defn}\label{def:tensorprod}
We define the tensoring functor $\otimes\colon \TwoCat\times \DblCat\to \DblCat$ to be the composite
\begin{tz}
\node(A) at (-2,0) {$\TwoCat\times \DblCat$};
\node(B) at (2,0) {$\DblCat\times \DblCat$};
\node (C) at (5.25,0) {$\DblCat$.};
\draw[->] (A) to node[above,scale=0.8] {$\bbH\times \id$} (B);
\draw[->] (B) to node[above,scale=0.8] {$\otimes_\Gray$} (C);
\end{tz}
\end{defn}

\begin{prop} \label{DblCattensored}
The category $\DblCat$ is enriched, tensored, and cotensored over $\TwoCat$, with 
\begin{rome}
    \item hom $2$-categories given by $\bfH[\AA,\BB]_\ps$, for all $\AA,\BB\in \DblCat$,
    \item tensors given by $\cc\otimes \AA$, for all $\AA\in \DblCat$ and $\cc\in \TwoCat$, where $\otimes$ is the tensoring functor of \cref{def:tensorprod}, and
    \item cotensors given by $[\bbH\cc,\BB]_\ps$, for all $\BB\in \DblCat$ and $\cc\in \TwoCat$,
\end{rome}
where $[-,-]_\ps$ is the pseudo hom double category of \cref{def:pseudohomdouble}.
\end{prop}

\begin{proof}
This follows directly from the definition of $\otimes$, and the universal properties of the tensor $\otimes_{\Gray}$ and of the adjunction $\bbH\dashv \bfH$. 
\end{proof}

We now present the main result of this section.

\begin{thm} \label{2Catenrichment}
The model structure on $\DblCat$ of \cref{thm:modelstructonDblCat} is a $\TwoCat$-enriched model structure, where the enrichment is given by $\bfH[-,-]_\ps$. 
\end{thm}

The rest of this section is devoted to the proof of this theorem. With that goal, we first prove several auxiliary lemmas.

\begin{notation}
Let $i\colon \A\to \B$ and $j\colon \A'\to \B'$ be $2$-functors in $\TwoCat$, and let $I\colon \AA\to \BB$ be a double functor in $\DblCat$. We denote by $i\square_2 j$ the pushout-product
\[ \pushprod{i}{\A}{\B}{j}{\A'}{\B'}{2} \]
with respect to the Gray tensor product $\otimes_2$ on $\TwoCat$ (see \cref{def:Gray2}), and we denote by $\pushout{i}{I}{}$ the pushout-product 
\[ \pushprod{i}{\A}{\B}{I}{\AA}{\BB}{} \]
with respect to the tensoring functor $\otimes\colon \TwoCat\times \DblCat\to \DblCat$. In particular, we have that $\pushout{i}{I}{}=\pushout{\bbH i}{I}{\Gray}$.
\end{notation}

\begin{lemma} \label{lem:Handotimes}
Let $\A$ and $\B$ be $2$-categories. There is an isomorphism of double categories 
\[ \A\otimes \bbH\B\cong \bbH(\A\otimes_2 \B), \]
natural in $\A$ and $\B$. 
\end{lemma}

\begin{proof}
By the universal properties of $\otimes$ and $\otimes_2$, the adjunction $\bbH\dashv \bfH$, and \cref{lem:HVpreservehom}, we have an isomorphism 
\begin{align*}
    \DblCat(\A\otimes\bbH\B,\bC)&\cong \TwoCat(\A,\bfH[\bbH\B,\bC]_\ps)\cong \TwoCat(\A,\Psd[\B,\bfH\bC]) \\ &\cong \TwoCat(\A\otimes_2\B,\bfH\bC)\cong \DblCat(\bbH(\A\otimes_2\B),\bC),
\end{align*}
for every double category $\bC$, which is natural in $\A$, $\B$, and $\bC$. The result then follows from the Yoneda lemma.
\end{proof}

\begin{rmk}
In particular, the natural isomorphism  $\bfH[\bbH(-),-]_{\ps}\cong \Psd[-,\bfH(-)]$ implies that the adjunction $\bbH\dashv \bfH$ is enriched with respect to the $\TwoCat$-enrichments $\bfH[-,-]_{\ps}$ and $\Psd[-,-]$ of $\DblCat$ and $\TwoCat$, respectively.
\end{rmk}

\begin{lemma} \label{lem:Vtwoandotimes}
Let $\A$ be a $2$-category. There is an isomorphism of double categories 
\[ \A\otimes \vtwo\cong \bbH\A\times \vtwo, \]
natural in $\A$.
\end{lemma}

\begin{proof}
By the universal properties of $\otimes$ and $\times$, and the fact that $\bfH[\vtwo,\BB]_\ps= \bfH[\vtwo,\BB]$ for all $\BB\in \DblCat$ by the proof of \cref{lem:HVpreservehom}, we have an isomorphism
\begin{align*} 
\DblCat(\A\otimes\vtwo,\BB)&\cong \TwoCat(\A,\bfH[\vtwo,\BB]_\ps)= \TwoCat(\A,\bfH[\vtwo,\BB])\\ &\cong \DblCat(\bbH\A, [\vtwo,\BB])\cong \DblCat(\bbH\A\times \vtwo,\BB), 
\end{align*}
for every double category $\BB$, which is natural in $\A$ and $\BB$. The result then follows from the Yoneda lemma.
\end{proof}

\begin{lemma} \label{lem:pushprodcof}
Let $i\colon \A\to \B$ and $j\colon \A'\to \B'$ be $2$-functors in $\TwoCat$. There are isomorphisms
\[ \pushout{i}{\bbH j}{}\cong \bbH(\pushout{i}{j}{2}) \ \ \text{and} \ \ \pushout{i}{(\bbH j\times \vtwo)}{}\cong \bbH(\pushout{i}{j}{2})\times \vtwo \]
in the arrow category $\DblCat^{\mathbbm 2}$.
\end{lemma}

\begin{proof}
Since $\bbH$ is a left adjoint, it preserves pushouts. Moreover, by \cref{lem:Handotimes}, we have that it is compatible with the tensors $\otimes$ and $\otimes_2$. Therefore, $\pushout{i}{\bbH j}{}\cong \bbH(\pushout{i}{j}{2})$. By \cref{lem:Vtwoandotimes}, by associativity of $\otimes_\Gray$, and by the first isomorphism, we then get that
\[ \pushout{i}{(\bbH j\times \vtwo)}{}\cong \pushout{i}{(j\otimes \vtwo)}{}\cong \pushout{(i}{\bbH j)\otimes_\Gray \vtwo}{}\cong (\pushout{i}{j}{2})\otimes \vtwo\cong \bbH(\pushout{i}{j}{2})\times \vtwo. \qedhere \]
\end{proof}

We are now ready to prove \cref{2Catenrichment}.

\begin{proof}[Proof of \cref{2Catenrichment}]
Recall from \cref{lem:gencofIJ} that a set $\mathcal I$ of generating cofibrations and a set $\mathcal J$ of generating trivial cofibrations for the model structure on $\DblCat$ are given by morphisms of the form $\bbH j$ and $\LV j=\bbH j\times \vtwo$, where $j$ is a generating cofibration or a generating trivial cofibration in $\TwoCat$, respectively. 

We show that the pushout-product of a generating cofibration in $\mathcal I$ with any (trivial) cofibration in $\TwoCat$ is a (trivial) cofibration in $\DblCat$, and that the pushout-product of a generating trivial cofibration in $\mathcal J$ with any cofibration in $\TwoCat$ is a trivial cofibration in~$\DblCat$. 

Given cofibrations $i$ and $j$ in $\TwoCat$, we know by \cref{lem:pushprodcof} that 
\[ \pushout{i}{\bbH j}{}\cong \bbH(\pushout{i}{j}{2}) \ \ \text{and} \ \ \pushout{i}{(\bbH j\times \vtwo)}{}\cong \bbH(\pushout{i}{j}{2})\times \vtwo=\LV(\pushout{i}{j}{2}), \]
and by \cref{2Catenrichedmodel} that $\pushout{i}{j}{2}$ is also a cofibration in $\TwoCat$, which is trivial when either~$i$ or~$j$ is. Since $\bbH$ and $\LV$ preserve (trivial) cofibrations by \cref{QuillenadjHH,LVQuillenadj}, then $\bbH(\pushout{i}{j}{2})$ and $\LV(\pushout{i}{j}{2})$ are cofibrations in $\DblCat$, which are trivial if either $i$ or $j$ is. Taking $j$ to be a generating cofibration or generating trivial cofibration in $\TwoCat$, we get the desired results. 
\end{proof}

\section{Comparison with other model structures on \texorpdfstring{$\DblCat$}{DblCat}} \label{Sec:comparison}

In \cite{FPP}, Fiore, Paoli, and Pronk construct several model structures on the category $\DblCat$ of double categories. We show in this section that our model structure on $\DblCat$ is not related to their model structures in the following sense: the identity adjunction on $\DblCat$ is not a Quillen pair between the model structure of \cref{thm:modelstructonDblCat} and any of the model structures of \cite{FPP}. This is not surprising, since our model structure was constructed in such a way that the functor $\bbH\colon \TwoCat\to \DblCat$ embeds the homotopy theory of $\TwoCat$ into that of~$\DblCat$, while there seems to be no such relation between their model structures on~$\DblCat$ and the Lack model structure on $\TwoCat$, e.g.~see end of Section 9 in \cite{FPP}. Further evidence is given by the fact that our double biequivalences are $2$-categorically induced, while the weak equivalences in the model structures of \cite{FPP} are rather $1$-categorically induced. 

We start by recalling the categorical model structures on $\DblCat$ constructed in \cite{FPP}. Since our primary interest is to compare them to our model structure, we only describe the weak equivalences; the curious reader is encouraged to visit their paper for further details.

The first model structure we recall is induced from the canonical model structure on $\Cat$ by means of the \emph{vertical nerve}.

\begin{defn}[{\cite[Definition 5.1]{FPP}}]
The \textbf{vertical nerve} of double categories is the functor 
\[ N_v\colon \DblCat\to \Cat^{\Delta^{\op}} \]
sending a double category $\AA$ to the simplicial category $N_v\AA$ such that $(N_v\AA)_0$ is the category of objects and horizontal morphisms of $\AA$, $(N_v\AA)_1$ is the category of vertical morphisms and squares of $\AA$, and $(N_v\AA)_n=(N_v\AA)_1\times_{(N_v\AA)_0}\ldots\times_{(N_v\AA)_0}(N_v\AA)_1$, for $n\geq 2$.
\end{defn}

\begin{prop}[{\cite[Theorem 7.17]{FPP}}] \label{prop:NvMS}
There is a model structure on $\DblCat$ in which a double functor $F$ is a weak equivalence if and only if $N_v F$ is levelwise an equivalence of categories. 
\end{prop}

\begin{rmk}\label{Rmk:verticalnerve}
Note that, for a double category $\AA$, the categories $(N_v\AA)_0$ and $(N_v\AA)_1$ are precisely the underlying categories $U\bfH\AA$ and $U\cv\AA$, respectively. One could also define a model structure on $\DblCat$ which is right-induced along a $2$-categorical version of the vertical nerve $\widehat{N_v}\colon \DblCat\to\TwoCat^{\Delta^{\op}}$ from the projective model structure on $\TwoCat^{\Delta^{\op}}$, where $\TwoCat$ is endowed with the Lack model structure. Our model structure can then be seen as a truncated version of this $2$-categorical vertical nerve model structure.
\end{rmk}

The next model structure on $\DblCat$ requires a different perspective. For a $2$-category~$\A$ whose underlying $1$-category $U\A$ admits limits and colimits, there is a model structure on~$U\A$ in which the weak equivalences are precisely the equivalences of the $2$-category $\A$; see~\cite{LackTriv}. When applying this construction to the $2$-category $\DblCat_h$ of double categories, double functors, and horizontal natural transformations, one obtains the following model structure on~$\DblCat$; see \cite[\S 8.4]{FPP}. 

\begin{prop}\label{prop:trivial}
There is a model structure on $\DblCat$, called the \emph{trivial model structure}, in which a double functor $F\colon \AA\to \BB$ is a weak equivalence if and only if it is an equivalence in the $2$-category $\DblCat_h$, i.e.,~there is a double functor $G\colon \BB\to \AA$ and two horizontal natural isomorphisms $\id_\AA\cong GF$ and $FG\cong \id_\BB$. 
\end{prop}

\begin{rmk}
By comparing this to our Whitehead theorems (see \cref{subsec:whitehead}), we see that the weak equivalences in the model structure of \cref{prop:trivial} require stricter conditions than double biequivalences. Indeed, the units and counits in the statement above are horizontal \emph{strict} natural \emph{isomorphisms}, while in our Whitehead theorems they are horizontal \emph{pseudo} natural \emph{equivalences}. 
This further supports our claim that the weak equivalences in our model structure are a $2$-categorical analogue, and therefore carry more information, than the weak equivalences already present in the literature. Indeed, the $2$-category $\DblCat_h$, used to define the weak equivalences of \cref{prop:trivial}, can be promoted to a $3$-category with hom $2$-categories given by $\bfH[-,-]_\ps$, and this is the setting in which our work takes place. 
\end{rmk}

The last model structure is of a more algebraic flavor. Let $T$ be a $2$-monad on a $2$-ca\-te\-gory~$\A$. In \cite{LackTriv}, Lack gives a construction of a model structure on the category of $T$-algebras, in which the weak equivalences are the morphisms of $T$-algebras whose underlying morphism in $\A$ is an equivalence. In particular, double categories can be seen as the algebras over a $2$-monad on the $2$-category $\Cat(\mathrm{Graph})$ whose objects are the category objects in graphs; see \cite[\S 9]{FPP}. This gives the following model structure.

\begin{prop}\label{prop:algebra}
There is a model structure on $\DblCat$, called the \emph{algebra model structure}, in which a double functor $F$ is a weak equivalence if and only if its underlying morphism in the $2$-category $\Cat(\mathrm{Graph})$ is an equivalence.
\end{prop}

\begin{rmk} \label{rem:weinalg}
In \cite[Corollary 8.29 and Theorems 8.52 and 9.1]{FPP}, Fiore, Paoli, and Pronk show that the model structures on $\DblCat$ of \cref{prop:NvMS,prop:trivial,prop:algebra} coincide with model structures given by Grothendieck topologies, when double categories are seen as internal categories to $\Cat$. Then, it follows from \cite[Propositions 8.24 and 8.38]{FPP} that a weak equivalence in the algebra model structure is in particular a weak equivalence in the model structure induced by the vertical nerve $N_v$. 
\end{rmk}

\begin{rmk}
At this point, we must mention that \cite{FP,FPP} define other model structures on~$\DblCat$, which are not equivalent to any of the above. However, these are Thomason-like model structures, and are therefore not expected to have any relation to our model structure, which is categorical.
\end{rmk}

We now proceed to compare these three model structures on $\DblCat$ to the one defined in \cref{thm:modelstructonDblCat}. Our strategy will be to find a trivial cofibration in our model structure that is not a weak equivalence in any of the other model structures. Let $\Eadj$ be the free-living adjoint equivalence $2$-category $\{0\xrightarrow{\simeq} 1\}$. By \cref{prop:alternate_gen_cofibs}, the inclusion double functor $J_1\colon \mathbbm{1}\to \bbH\Eadj$ at $0$ is a generating trivial cofibration in our model structure on $\DblCat$.

\begin{lemma} \label{Hjnotwe}
The double functor $J_1\colon \mathbbm{1}\to \bbH\Eadj$ is not a weak equivalence in any of the model structures on $\DblCat$ of \cref{prop:NvMS,prop:trivial,prop:algebra}.
\end{lemma}

\begin{proof}
We first prove that $J_1$ is not a weak equivalence in the model structure on $\DblCat$ of \cref{prop:NvMS} induced by the vertical nerve. For this, we need to show that \[ N_v(J_1)\colon N_v(\mathbbm{1})=\Delta\mathbbm{1}\to N_v(\bbH\Eadj) \]
is not a levelwise equivalence of categories. Indeed, the category $N_v(\bbH\Eadj)_0$ is the free category generated by $\{0\rightleftarrows 1\}$ which is not equivalent to $\mathbbm{1}$.

By \cref{rem:weinalg}, a weak equivalence in the algebra model structure on $\DblCat$ of \cref{prop:algebra} is in particular a weak equivalence in the model structure induced by the vertical nerve. Therefore $J_1$ is not a weak equivalence in the algebra model structure either.

Finally, we show that $J_1$ is not a weak equivalence in the trivial model structure on $\DblCat$ of \cref{prop:trivial}. If $J_1$ was an equivalence in the $2$-category $\DblCat_h$, then its weak inverse would be given by the unique double functor $!\colon \bbH\Eadj\to \mathbbm{1}$ and we would have a horizontal natural isomorphism $\id_{\bbH\Eadj}\cong J_1\, !$, where $J_1\, !$ is constant at $0$. But such a horizontal natural isomorphism does not exist since $1$ is not isomorphic to $0$ in $\bbH\Eadj$. Therefore $J_1$ is not an equivalence. 
\end{proof}

\begin{prop}
The identity adjunction on $\DblCat$ is not a Quillen pair between the model structure of \cref{thm:modelstructonDblCat} and any of the model structures of \cref{prop:NvMS,prop:trivial,prop:algebra}.
\end{prop}

\begin{proof}
We consider the identity functor $\id\colon \DblCat\to \DblCat$ from the model structure of \cref{thm:modelstructonDblCat} to any of the other model structures of \cref{prop:NvMS,prop:trivial,prop:algebra}, and we show that it is neither left nor right Quillen. 

Since $J_1$ is a trivial cofibration in the model structure of \cref{thm:modelstructonDblCat}, but is not a weak equivalence in any of the other model structures by \cref{Hjnotwe}, we see that $\id$ does not preserve trivial cofibrations; therefore, it is not left Quillen. Moreover, every object is fibrant in the model structure of \cref{thm:modelstructonDblCat}, so that if $\id$ was right Quillen, it would preserve all weak equivalences by Ken Brown's Lemma (see \cite[Lemma 1.1.12]{Hovey}). However, it does not preserve the weak equivalence~$J_1$, and thus it is not right Quillen.
\end{proof}

\section{Model structure for weak double categories}\label{Sec:modelstructwkDblCat}

In this section, we turn our attention to \emph{weak} double categories, and show that our results for double categories still hold in this weaker setting. In particular, in \cref{subsec:modelstructwkDblCat} we establish the existence of a model structure on the category $\wkDblCat$ of weak double categories and strict double functors, which is right-induced from two copies of the Lack model structure of~\cite{LackBicat} on $\Bicats$, the category of bicategories and strict functors.  

As in the strict context, the horizontal embedding of bicategories into weak double categories is both left and right Quillen, and homotopically fully faithful; thus the homotopy theory of bicategories can be read off of the homotopy of weak double categories. Moreover, the Lack model structure on $\Bicats$ is  both left- and right-induced from this model structure on weak double categories along the horizontal embedding.

Finally, in \cref{subsec:quilleneq}, we show the interplay between the model structures on the categories considered so far. More precisely, we prove that the inclusion of double categories into weak double categories is the right part of a Quillen equivalence, which gives rise 
to the following commutative square of right Quillen and homotopically fully faithful functors 
\begin{tz}
\node (A) at (0,0) {$\TwoCat$};
\node (B) at (2.5,0) {$\Bicats$};
\node (C) at (0,-1.5) {$\DblCat$};
\node (D) at (2.5,-1.5) {$\wkDblCat$};
\node at ($(D.east)-(0,4pt)$) {.};
\draw[->] (A) to node[above,scale=0.8] {$I_2$} node[below,scale=0.8] {$\simeq_\mathrm{Q}$} (B);
\draw[->] (C) to node[above,scale=0.8] {$I$} node[below,scale=0.8] {$\simeq_\mathrm{Q}$} (D);
\draw[->] (A) to node[left, scale=0.8] {$\bbH$} (C);
\draw[->] (B) to node[right, scale=0.8] {$\bbH^{w}$} (D);
\end{tz}

\subsection{Constructing the model structure for \texorpdfstring{$\wkDblCat$}{wkDblCat}}\label{subsec:modelstructwkDblCat}

Let us first define weak double categories, by which we mean double categorical structures that are weak in the horizontal direction but strict in the vertical one.

\begin{defn}
A \textbf{weak double category} $\BB$ consists of objects, horizontal morphisms, vertical morphisms, and squares, together with horizontal and vertical compositions and identities such that composition of vertical morphisms is strictly associative and unital, as for double categories, but composition of horizontal morphisms is associative and unital up to a vertically invertible square with trivial vertical boundaries.
\end{defn}

\begin{notation}
We write $\wkDblCat$ for the category of weak double categories and strict double functors between them. 
\end{notation}

\begin{rmk}
A bicategory $\B$ can be seen as a horizontal weak double category~$\bbH^{w}\B$ with only trivial vertical morphisms. This induces a functor $\bbH^{w}\colon \Bicats\to \wkDblCat$.
\end{rmk}

\begin{rmk}
As in the case of $\TwoCat$ and $\DblCat$, we have the following adjunctions 
\begin{tz}
\node (B) at (3,0) {$\wkDblCat$};
\node (A) at (0,0) {$\Bicats$};
\node at ($(B.east)-(0,4pt)$) {,};
\draw[->] ($(A.east)+(0,.25cm)$) to [bend left] node[above,scale=0.8]{$\bbH^{w}$} ($(B.west)+(0,.25cm)$);
\draw[->] ($(B.west)+(0,-.25cm)$) to [bend left] node[below,scale=0.8]{$\bfH^{w}$} ($(A.east)-(0,.25cm)$);
\node[scale=0.8] at ($(A.east)!0.5!(B.west)$) {$\bot$};

\node (B) at (9,0) {$\wkDblCat$};
\node (A) at (6,0) {$\Bicats$};
\node at ($(B.east)-(0,4pt)$) {,};
\draw[->] ($(A.east)+(0,.25cm)$) to [bend left] node[above,scale=0.8]{$\LV^{w}$} ($(B.west)+(0,.25cm)$);
\draw[->] ($(B.west)+(0,-.25cm)$) to [bend left] node[below,scale=0.8]{$\cv^{w}$} ($(A.east)-(0,.25cm)$);
\node[scale=0.8] at ($(A.east)!0.5!(B.west)$) {$\bot$};
\end{tz}
where the functors involved are defined as the functors $\bbH$, $\bfH$, $\LV$, and $\cv$ of \cref{section:prelim} on objects, morphisms, and cells. Note that the functor $\cv^{w}$ can be described as $\bfH^{w}[\vtwo^{w},-]$, where $\vtwo^{w}$ is the weak double category having two objects $0$ and $1$, a vertical morphism $u\colon 0\arrowdot 1$, and non-trivial unitors for the horizontal identities.
\end{rmk}

Just as in the case of $\TwoCat$, Lack defines a model structure on $\Bicats$, whose weak equivalences and fibrations are also given by the biequivalences and the Lack fibrations; see \cite{LackBicat}. 
Similarly, in $\wkDblCat$, we define double biequivalences and double fibrations as in \cref{doublequiv,doublefib}. Then, analogues of \cref{charequiv,charfib} hold in this weak context, so that these double biequivalences and double fibrations correspond to the weak equivalences and fibrations in the right-induced model structure on $\wkDblCat$ along $(\bfH^{w},\cv^{w})$; this is the content of the following result. 

\begin{thm}\label{thm:modelstructWeakDbl}
Consider the adjunction 
\begin{tz}
\node (A) at (0,0) {$\Bicats\times \Bicats$};
\node (B) at (3.75,0) {$\wkDblCat$};
\node at ($(B.east)-(0,4pt)$) {,};
\draw[->] ($(A.east)+(0,.25cm)$) to [bend left] node[above,scale=0.8]{$\bbH^{w}\sqcup \LV^{w}$} ($(B.west)+(0,.25cm)$);
\draw[->] ($(B.west)+(0,-.25cm)$) to [bend left] node[below,scale=0.8]{$(\bfH^{w},\cv^{w})$} ($(A.east)-(0,.25cm)$);
\node[scale=0.8] at ($(A.east)!0.5!(B.west)$) {$\bot$};
\end{tz}
where $\Bicats$ is endowed with the Lack model structure. Then the right-induced model structure on $\wkDblCat$ exists. In particular, a strict double functor is a weak equivalence (resp.~fibration) in this model structure if and only if it is a double biequivalence (resp.~double fibration).
\end{thm}

\begin{proof}
The proof is given by applying \cref{path}. Indeed, the model structure on $\Bicats$ is combinatorial and all objects are fibrant (see \cite{LackBicat}). Furthermore, a path object for a weak double category $\BB$ can be given by $[(\bbH\Eadj)^w,\BB]_\ps$, where $(\bbH\Eadj)^w$ is the free-living horizontal adjoint equivalence \emph{weak} double category. 
\end{proof}

\begin{rmk}
 Note that every weak double category is fibrant in this model structure.
\end{rmk}

Similarly to the strict case, we can show that the horizontal embedding of bicategories into weak double categories has very good homotopical properties with respect to Lack's model structure and the above-constructed model structure. 

\begin{thm}
The adjunctions 
\begin{tz}
\node[](A) {$\Bicats$};
\node[right of=A,xshift=2cm](B) {$\wkDblCat$};
\node at ($(B.east)-(0,4pt)$) {,};
\draw[->] ($(B.west)+(0,.25cm)$) to [bend right=50] node[above,la]{$L^w$} ($(A.east)+(0,.25cm)$);
\draw[->] (A) to node[la,over] {$\bbH^w$} (B);
\draw[->] ($(B.west)+(0,-.25cm)$) to [bend left=50] node[below,la]{$\bfH^w$} ($(A.east)-(0,.25cm)$);
\node[la] at ($(A.east)!0.5!(B.west)+(0,.35cm)$) {$\bot$};
\node[la] at ($(A.east)!0.5!(B.west)-(0,.35cm)$) {$\bot$};
\end{tz}
are both Quillen pairs between the Lack model structure on $\Bicats$ and the model structure on $\wkDblCat$ of \cref{thm:modelstructWeakDbl}. Moreover, the functor $\bbH^w$ is homotopically fully faithful, and the Lack model structure on $\Bicats$ is both left- and right-induced from our model structure on~$\wkDblCat$ along $\bbH^w$. 
\end{thm}

\begin{proof}
The proof proceeds as in \cref{QuillenadjHH,thm:LHQuillenadj,thm:leftandrightindLHH}.
\end{proof}

\begin{rmk}
As a consequence of the above theorem, we get that $\bbH^w\colon \Bicats\to \wkDblCat$ both preserves and reflects all cofibrations, fibrations, and weak equivalences, so that the Lack model structure on $\Bicats$ is completely determined by our model structure on $\wkDblCat$. Furthermore, since $\bbH^w$ is left and right Quillen, and homotopically fully faithful, this says that the homotopy theory of bicategories is reflectively and coreflectively embedded in that of weak double categories.
\end{rmk}

A characterization of the cofibrant objects, similar to that of \cref{char:cofibrant} and inspired by \cite[Lemma 8]{LackBicat}, also holds. In the weak setting, we do not have an underlying horizontal category, since horizontal composition is not well defined when discarding the squares. Instead, the underlying horizontal structure is given by a \emph{compositional graph}; these are directed graphs, with a chosen identity for each object, and a chosen composite for each composable pair of edges, but with neither the associativity law nor the unitality laws assumed to hold. For more details regarding compositional graphs, see \cite[\S 1]{LackBicat}.

\begin{prop} \label{prop:wkdblcof}
A weak double category is cofibrant if and only if it is a retract of a weak double category whose
\begin{rome}
\item underlying horizontal compositional graph is free, and
\item underlying vertical category is a disjoint union of copies of $\mathbbm{1}$ and $\mathbbm{2}$.
\end{rome}
\end{prop} 

\begin{rmk} \label{Hpreservecofibrant}
By \cite[Lemma 8]{LackBicat}, a bicategory is cofibrant if and only if it is a retract of a bicategory whose underlying compositional graph is free. It then follows from \cref{prop:wkdblcof} that the functor $\bfH^{w}\colon \wkDblCat\to \Bicats$ preserves cofibrant objects. 
\end{rmk}

\subsection{Quillen equivalence between \texorpdfstring{$\DblCat$}{DblCat} and \texorpdfstring{$\wkDblCat$}{wkDblCat}} \label{subsec:quilleneq}

Now that we have defined a model structure for weak double categories and have compared it to that of bicategories, we study its relation with our model structure on $\DblCat$. This gives a double categorical analogue of the relation between $2$-categories and bicategories proved by Lack, which we now recall.

\begin{rmk}
The full inclusion $I_2\colon \TwoCat\to \Bicats$ has a left adjoint, which we denote by $S_2\colon \Bicats\to \TwoCat$. This functor sends a bicategory $\B$ to the $2$-category $S_2\B$ that has
\begin{rome}
\item the same objects as $\B$, and
\item morphisms and $2$-cells obtained as a quotient of the morphisms and $2$-cells of $\B$, where the associators and unitors are universally made into identities. 
\end{rome}
More precisely, for all composable morphisms $a,b,c$ in $\B$, the composites $(cb)a$ and $c(ba)$ are identified in $S_2\B$, and, for every morphism $a\colon A\to B$ in $\B$, the composites $\id_B a$ and $a \id_A$ are identified with $a$ in $S_2\B$. Furthermore, the unit component ${\eta_2}_\B\colon \B\to I_2S_2\B$ is the quotient strict functor corresponding to the above identifications. 

As Lack explains in \cite[\S 7.5]{LackCompanion}, the functor $S_2$ is not the usual strictification, but rather the ``free strictification". Notably, it is such that, for all bicategories $\B$, the unit components ${\eta_2}_\B$ are strict functors, as opposed to the usual strictification process that produces only pseudo functors.
\end{rmk}

\begin{thm}[{\cite[Theorem 11]{LackBicat}}] \label{QeqLackmodelstruct}
The adjunction 
\begin{tz}
\node (A) at (0,0) {$\Bicats$};
\node (B) at (2.5,0) {$\TwoCat$};
\draw[->] ($(A.east)+(0,.25cm)$) to [bend left] node[above,scale=0.8]{$S_2$} ($(B.west)+(0,.25cm)$);
\draw[->] ($(B.west)+(0,-.25cm)$) to [bend left] node[below,scale=0.8]{$I_2$} ($(A.east)-(0,.25cm)$);
\node[scale=0.8] at ($(A.east)!0.5!(B.west)$) {$\bot$};
\end{tz}
is a Quillen equivalence, where $\TwoCat$ and $\Bicats$ are endowed with the Lack model structure.
\end{thm}

We now study the setting of double categories, and show that an analogous result to the $2$-categorical one is also true.

\begin{rmk}
The full inclusion $I\colon \DblCat\to \wkDblCat$ has a left adjoint, which we denote by $S\colon \wkDblCat\to \DblCat$. This functor sends a weak double category $\BB$ to the double category $S\BB$ that has
\begin{rome}
\item the same objects and vertical morphisms as $\BB$, and
\item horizontal morphisms and squares obtained as a quotient of the horizontal morphisms and squares of $\BB$, where the associators and unitors are universally made into identities. 
\end{rome}
More precisely, for all composable horizontal morphisms $a,b,c$ in $\BB$, the composites $(cb)a$ and $c(ba)$ are identified in $S\BB$, and, for every horizontal morphism $a\colon A\to B$ in $\BB$, the composites $\id_B a$ and $a \id_A$ are identified with $a$ in $S\BB$. Furthermore, the unit component $\eta_\BB\colon \BB\to IS\BB$ is described as the quotient strict double functor corresponding to the above identifications.
\end{rmk}

\begin{rmk} \label{equalities}
From the definitions of the functors involved, we directly see that
\[ I_2\bfH=\bfH^{w} I, \quad S_2\bfH^{w}=\bfH S, \quad I_2\cv=\cv^{w} I, \ \ \text{and} \ \ S_2\cv^{w}=\cv S. \]
Furthermore, if we denote by $\eta_2\colon \id\Rightarrow I_2S_2$ and $\eta\colon \id\Rightarrow IS$ the units of the adjunctions $S_2\dashv I_2$ and $S\dashv I$ respectively, then we have that 
\[ \eta_2\bfH^{w}=\bfH^{w}\eta\colon \bfH^{w}\Rightarrow I_2S_2\bfH^{w}=\bfH^{w}IS \ \ \text{and} \ \ \eta_2\cv^{w}=\cv^{w}\eta\colon \cv^{w}\Rightarrow I_2S_2\cv^{w}=\cv^{w}IS. \]
\end{rmk}

\begin{thm} \label{thm:QeqIS}
The adjunction 
\begin{tz}
\node (A) at (0,0) {$\DblCat$};
\node (B) at (3.25,0) {$\wkDblCat$};
\draw[->] ($(A.east)+(0,.25cm)$) to [bend left] node[above,scale=0.8]{$S$} ($(B.west)+(0,.25cm)$);
\draw[->] ($(B.west)+(0,-.25cm)$) to [bend left] node[below,scale=0.8]{$I$} ($(A.east)-(0,.25cm)$);
\node[scale=0.8] at ($(A.east)!0.5!(B.west)$) {$\bot$};
\end{tz}
is a Quillen equivalence, where $\DblCat$ is endowed with the model structure of \cref{thm:modelstructonDblCat} and $\wkDblCat$ with the model structure of \cref{thm:modelstructWeakDbl}. 
\end{thm}

\begin{proof}
Recall that all objects in the model structures considered in this section are fibrant, and thus the components  of the derived units of the adjunctions $S\dashv I$ and $S_2\dashv I_2$ coincide with the components of their units at cofibrant objects. We use this fact henceforth. 

Since the inclusion $I\colon \DblCat\to \wkDblCat$ reflects weak equivalences, using \cite[Corollary 1.3.16]{Hovey} it is enough to show that for every cofibrant $\BB\in \wkDblCat$ the unit component $\eta_\BB\colon \BB\to IS\BB$ is a double biequivalence in $\wkDblCat$. Let $\BB$ be a cofibrant weak double category. By construction, it suffices to show that both $\bfH^{w} \eta_\BB$ and $\cv^{w} \eta_\BB$ are biequivalences in $\Bicats$.

By \cref{Hpreservecofibrant}, the bicategory $\bfH^{w}\BB$ is cofibrant in $\Bicats$. Therefore, since $I_2\dashv S_2$ is a Quillen equivalence by \cref{QeqLackmodelstruct}, the unit component ${\eta_2}_{\bfH^{w}\BB}\colon \bfH^{w}\BB\to I_2 S_2\bfH^{w} \BB$ is a biequivalence. As $\bfH^{w} \eta=\eta_2 \bfH^{w}$ by \cref{equalities}, this shows that $\bfH^{w} \eta_\BB$ is a biequivalence.

We now prove that $\cv^w \eta_\BB$ is a biequivalence. Since $\cv^{w} \eta=\eta_2 \cv^{w}$ by \cref{equalities}, it is enough to show that the unit component ${\eta_2}_{\cv^{w}\BB}\colon \cv^{w}\BB\to I_2S_2\cv^{w}\BB$ is a biequivalence. By~\cite[Proposition 2]{LackBicat}, we know that ${\eta_2}_{\cv^{w}\BB}$ is bijective on objects, surjective on morphisms, and full on $2$-cells; it remains to show that it is faithful on $2$-cells. For this, recall that a $2$-cell in $\cv^{w}\BB$ comprises the data of two squares in~$\BB$ of the form
\begin{tz}
\node (A) at (0,0) {$A$};
\node (B) at (1.5,0) {$B$};
\node (A') at (0,-1.5) {$A$};
\node (B') at (1.5,-1.5) {$B$};
\draw[->] (A) to node[above, scale=0.8] {$a$} (B);
\draw[->] (A') to node[below, scale=0.8] {$c$} (B');
\draw[d] (A) to (A');
\draw[d] (B) to (B');

\node at (0,-.75) {$\bullet$};
\node at (1.5, -.75) {$\bullet$};

\node[scale=0.8] at (.75,-.75) {$\sigma_0$};

\node (A) at (4,0) {$A'$};
\node (B) at (5.5,0) {$B'$};
\node (A') at (4,-1.5) {$A'$};
\node (B') at (5.5,-1.5) {$B'$};
\draw[->] (A) to node[above, scale=0.8] {$b$} (B);
\draw[->] (A') to node[below, scale=0.8] {$d$} (B');
\draw[d] (A) to (A');
\draw[d] (B) to (B');

\node at (4,-.75) {$\bullet$};
\node at (5.5, -.75) {$\bullet$};

\node[scale=0.8] at (4.75,-.75) {$\sigma_1$};
\end{tz}
satisfying a certain vertical pasting equality. Note that the squares $\sigma_0$ and $\sigma_1$ are $2$-cells in~$\bfH^{w}\BB$, and that, by the first part of the proof, the unit component ${\eta_2}_{\bfH^{w}\BB}\colon \bfH^{w}\BB\to I_2S_2\bfH^{w}\BB$ is faithful on $2$-cells. Then the quotient strict double functor ${\eta_2}_{\cv^{w}\BB}\colon \cv^{w}\BB\to I_2S_2\cv^{w}\BB$ is also faithful on $2$-cells, since it is given by applying ${\eta_2}_{\bfH^{w}\BB}$ to the components $\sigma_0$ and $\sigma_1$ separately. This shows that $\cv^{w}\eta_\BB$ is a biequivalence, and concludes the proof. 
\end{proof}

\bibliographystyle{plain}
\bibliography{Reference}

\end{document}